\RequirePackage{fix-cm}
\documentclass[smallextended,numbook,runningheads]{svjour3}
\smartqed  % flush right qed marks, e.g. at end of proof
\usepackage{graphicx}
\usepackage{amsmath,amssymb,bm}
\usepackage{graphicx,graphics,color}
\usepackage[colorlinks,linkcolor=blue]{hyperref}
\usepackage{epstopdf}
\usepackage[normalem]{ulem}

\journalname{...}
\date{ \phantom{b} \vspace{45mm}\phantom{e}}
\def\makeheadbox{\hspace{-4mm}Version of \today}%, file: esep.tex}

\usepackage{color}

\def\half{\tfrac{1}{2}}
\def\Half{\frac{1}{2}}

\newcommand\bfd{{\mathbf d}}
\newcommand\bfe{{\mathbf e}}
\newcommand\bff{{\mathbf f}}
\newcommand\bfg{{\mathbf g}}

\newcommand\bfn{{\mathbf n}}
\newcommand\bfr{{\mathbf r}}

\newcommand\bfu{{\mathbf u}}
\newcommand\bfv{{\mathbf v}}
\newcommand\bfw{{\mathbf w}}
\newcommand\bfx{{\mathbf x}}
\newcommand\bfy{{\mathbf y}}
\newcommand\bfz{{\mathbf z}}
\newcommand\bfA{{\mathbf A}}

\newcommand\bfH{{\mathbf H}}
\newcommand\bfF{{\mathbf F}}

\newcommand\bfK{{\mathbf K}}
\newcommand\bfM{{\mathbf M}}
\newcommand\bfR{{\mathbf R}}
\newcommand\bbA{{\mathbf A}}
\newcommand\bbM{{\mathbf M}}

\newcommand\bfV{{\mathbf V}}

\newcommand\bfeta{{\boldsymbol \eta}}

\newcommand\bfvartheta{{\boldsymbol \vartheta}}

\newcommand\Q{Q}
\newcommand\Qh{Q_h}
\newcommand\andquad{\quad\hbox{ and }\quad}

\renewcommand\d{\hbox{\rm{d}}}

\newcommand{\Ga}{\Gamma}

\newcommand{\laplace}{\Delta}

\newcommand{\nbg}{\nabla_{\Gamma}}
\newcommand{\nbgx}{\nabla_{\Gamma[X]}}
\newcommand{\nbgh}{\nabla_{\Gamma_h}}
\newcommand{\mat}{\partial^{\bullet}}

\newcommand{\diff}{\frac{\d}{\d t}}
\newcommand{\eps}{\varepsilon}

\newcommand{\inv}{^{-1}}

\newcommand{\nb}{\nabla}

\newcommand{\pa}{\partial}
\newcommand{\R}{\mathbb{R}}

\newcommand{\spn}{\textnormal{span}}

\def \t {(t)}

\def \to {\rightarrow}
\newcommand{\vphi}{\varphi}

\newcommand{\V}{V}
\newcommand{\phiv}{\varphi^v}
\newcommand{\phin}{\varphi^\n}
\newcommand{\phiH}{\varphi^H}
\newcommand{\phiw}{\varphi^{\V}}

\newcommand{\phiwn}{\varphi^{z}}

\def \s {(s)}

\newcommand{\Ih}{\widetilde{I}_h}

\newcommand{\M}{\bfM}
\newcommand{\Ms}{\bfM^\ast}
\newcommand{\A}{\bfA}
\newcommand{\As}{\bfA^\ast}
\newcommand{\K}{\bfK}
\newcommand{\Ks}{\bfK^\ast}

\newcommand{\normK}[1]{\| #1\|_{\bfK(\xs)}}
\newcommand{\normKt}[1]{\| #1\|_{\bfK(\xs(t))}}

\newcommand{\us}{\bfu^\ast}
\newcommand{\vs}{\bfv^\ast}
\newcommand{\xs}{\bfx^\ast}
\newcommand{\ws}{\bfw^\ast}
\newcommand{\dotus}{\dot\bfu^\ast}

\newcommand{\dotxs}{\dot\bfx^\ast}
\newcommand{\dotws}{\dot\bfw^\ast}
\newcommand{\ddotus}{\ddot\bfu^\ast}

\newcommand{\eu}{\bfe_\bfu}
\newcommand{\ev}{\bfe_\bfv}
\newcommand{\ex}{\bfe_\bfx}
\newcommand{\ew}{\bfe_\bfw}

\newcommand{\doteu}{\dot\bfe_\bfu}

\newcommand{\dotex}{\dot\bfe_\bfx}
\newcommand{\dotew}{\dot\bfe_\bfw}
\newcommand{\du}{\bfd_\bfu}
\newcommand{\dv}{\bfd_\bfv}
\newcommand{\dw}{\bfd_\bfw}
\newcommand{\dotdu}{\dot\bfd_\bfu}

\newcommand{\dotdw}{\dot\bfd_\bfw}

\newcommand{\n}{\nu}

\newcommand{\dof}{N}

\def\lengthcontrol{1}
\newcommand{\detailedproof}[1]{\if \lengthcontrol 1 \blueon \noindent \hrulefill \\ (\emph{Beginning of the detailed proof.}) \emph{To shorten set the variable} $\backslash$\texttt{lengthcontrol} \emph{to $0$ in the preamble.} \\ {#1} \\ (\emph{End of the detailed proof.}) \ \noindent \hrulefill \ \blueoff \fi}
\definecolor{darkred}{rgb}{.7,0,0}

\newcommand{\redon}{\color{red} }
\newcommand{\redoff}{\color{black}}

\newcommand{\blueon}{\color{blue}}
\newcommand{\blueoff}{\color{black}}

\newcommand{\ebk}{\color{black}}

\newcommand{\eby}{\color{black}}

\begin{document}

\title{A convergent evolving finite element algorithm for Willmore flow of closed surfaces}

\titlerunning{A convergent finite element algorithm for Willmore flow}        % if too long for running head

\author{Bal\'{a}zs~Kov\'{a}cs \and
	Buyang~Li \and
	Christian~Lubich
}

\authorrunning{B.~Kov\'{a}cs, B.~Li and Ch.~Lubich} % if too long for running head

\institute{B. Kov\'{a}cs and Ch. Lubich  \at
	Mathematisches Institut, Universit\"at T\"{u}bingen,\\
	Auf der Morgenstelle 10, 72076 T\"{u}bingen, Germany \\
	\email{\{kovacs,lubich\}@na.uni-tuebingen.de}
	\and
	B. Li \at
	Department of Applied Mathematics, Hong Kong Polytechnic University,\\
	Kowloon, Hong Kong \\
	\email{buyang.li@polyu.edu.hk}
}

\date{}
% The correct dates will be entered by the editor

\maketitle

\begin{abstract}
	A proof of convergence is given for a novel evolving surface finite element semi-discretization of Willmore flow  of closed two-dimensional surfaces, and also of surface diffusion flow. The numerical method proposed and studied here discretizes fourth-order evolution equations for the normal vector and mean curvature, reformulated as a system of second-order equations, and uses these evolving geometric quantities in the velocity law interpolated to the finite element space. This numerical method admits a convergence analysis in the case of continuous finite elements of polynomial degree at least two. The error analysis combines
	stability estimates and consistency estimates to yield optimal-order $H^1$-norm error bounds for the computed surface position, velocity, normal vector and mean curvature. The
	stability analysis is based on the matrix--vector formulation of the finite element method and does not use geometric arguments. The geometry enters only into the consistency estimates. Numerical experiments illustrate and complement the theoretical results.
	
\end{abstract}

\section{Introduction}

The elastic bending energy or Willmore energy (named after Thomas Willmore, see \cite{Willmore65,Willmore_book}) of a surface $\Gamma$
%$$
%	\Ga[X(\cdot,t)] = \{ X(p,t) \,:\, p \in \Ga^0 \} ,
%$$
%described by a flow map $X:\Ga^0\times [0,T]\to \R^3$, 
is given as 
$$
W(\Gamma)=\tfrac12\int_{\Gamma} H^2  , 
$$
where $H$ is the mean curvature of the surface. {\it Willmore flow} is the $L^2$ gradient flow of surfaces for the elastic bending energy. It plays an important role in modelling lipid bilayers \cite{Helfrich}, biomembranes \cite{ElliottStinner_biomembranes}, vesicles \cite{BGN_vesicles}, regularization of phase-field systems \cite{Chen-Shen-2018}, and in the analysis of curvature on surfaces; see \cite{2014-Marques-Neves} proving the Willmore conjecture. 

The negative $L^2$ gradient of the Willmore energy $W$ for a two-dimensional surface $\Gamma$ in $\R^3$ has no tangential contribution and its normal component equals
%equals $V\nu$ with the normal vector field $\nu$ and
\begin{align}\label{Willmore_Eq}
V = \Delta_{\Gamma}H + H(\tfrac12 H^2 -2K)\quad\text{ on }\Ga,
\end{align}
with mean curvature $H=\kappa_1+\kappa_2$ (here taken without a factor $1/2$) and Gaussian curvature $K=\kappa_1\kappa_2$, where $\kappa_1$ and $\kappa_2$ are the principal curvatures on~$\Ga$.
Willmore \cite{Willmore_book} gives a proof of this result and attributes the formula to  Thomsen~\cite{Thomsen} (who mentions Schadow in 1922) and Blaschke~\cite{Blaschke}.

For  Willmore flow, $V$ of \eqref{Willmore_Eq} is taken as the normal velocity of the evolving surface. This yields a fourth-order geometric evolution equation.

%The flow map of Willmore flow,  $X:\Ga^0\times [0,T]\to \R^3$ defining the surfaces $\Ga[X(\cdot,t)] = \{ X(p,t) \,:\, p \in \Ga^0 \}$ for $0\le t \le T$, then
%satisfies the fourth-order geometric evolution equation (we omit the argument $t$)
%\begin{align}\label{Willmore_flow} 
%\begin{aligned}
%\partial_t X&= v\circ X  , \\
%v&= V \nu  \quad\text{on }\Ga[X] , %\label{Willmore_v} 
%\end{aligned}
%\end{align}
%where  $\nu$ is the normal vector field on $\Ga[X]$.
%We note that $v$ is the velocity of the evolving surface and $V$ is the normal velocity specified by \eqref{Willmore_Eq} on~$\Ga[X]$.

Numerical methods for Willmore flow %\eqref{Willmore_flow} with 
\eqref{Willmore_Eq} have been proposed in many articles. 
Algorithms based on evolving surface finite element methods  for Willmore flow were studied by Rusu \cite{Rusu}, Dziuk \cite{Dziuk_Willmore} and Barrett, Garcke \& N\"urnberg \cite{Barrett2007441,BGN2008Willmore} based on different variational formulations of \eqref{Willmore_Eq}. 
%In particular, Barrett, Garcke \& N\"urnberg used a special set of test functions in \cite{Barrett2007441,BGN2008Willmore}, similarly to their treatment for mean curvature flow in \cite{BGN2008}, to obtain almost equidistribution of mesh points. 
Bonito, Nochetto \& Pauletti \cite{Bonito-Nochetto-Sebastian-2010} considered surface finite elements for biomembranes described by Willmore flow under area and volume constraints (the Helfrich model).  Pozzi \cite{Pozzi_anisotropicWillmore} studied a numerical method for anisotropic Willmore flow. Numerical simulations indicate that these methods apparently converge in practical computations --- but to our knowledge, no proof of convergence has so far been obtained for any numerical method for the Willmore flow of closed surfaces. \ebk 
% D{\"o}rfler \& N{\"u}rnberg \cite{Dorfler-Nurnberg-2019} have considered numerical approximation to the $L^2$ gradient flow of the generalized elastic energy 
%$$E_\lambda(\Gamma[X])=\int_{\Gamma[X]}(G(H)+\lambda)\d\Gamma $$ 
%in the modeling of power loss within an optical fiber. 

%In the papers mentioned above, people have used $C^0$ finite elements based on mixed formulations of the parametric equation of the Willmore flow. 
For the Willmore flow of \emph{curves}, Bartels \cite{Bartels-2013} and Deckelnick \& Dziuk \cite{Deckelnick-Dziuk-2009}  proved convergence of finite element semi-discretizations. Pozzi \& Stinner \cite{PozziStinner_elasticcurves}  proved convergence of a semi-discrete surface finite element method for elastic flow of curves coupled to a lateral diffusion process. For the Willmore flow of \emph{graphs}, Deckelnick \& Dziuk proved convergence of a (non-evolving) $C^0$ mixed finite element method in \cite{Deckelnick-Dziuk-2006}, and a convergence result for $C^1$ finite elements was given in \cite{Deckelnick-Katz-Schieweck-2015}. 
However, convergence of a surface finite element method for the Willmore flow of {\it closed surfaces} has remained an open problem.  

Closely related to Willmore flow is the {\it surface diffusion flow}, which is the $H^{-1}$ gradient flow of the area functional
$|\Gamma|=\text{area}(\Gamma)$. The normal velocity of the evolving surface then becomes
%$$
%|\Gamma|=\int_{\Gamma}  \d\Gamma .
%$$
%It is governed by the fourth-order geometric evolution equation
\begin{align}\label{SF_Eq} 
V= \Delta_{\Gamma}H.
%\begin{aligned}
%\partial_t X&= v\circ X,  \\
%v&= V \n \quad\text{ with }\quad V= \Delta_{\Gamma[X]}H.
%\end{aligned}
\end{align}

The evolving surface described by \eqref{SF_Eq} is the limit of the zero level set of the solution to the Cahn--Hilliard equation with a concentration dependent mobility (when the parameter representing the thickness of phase transition zone tends to zero), see \cite{CahnElliottNovickCohen}. 
This model is often used to describe the diffusion-driven motion of the surface of a crystal, the motion of an interface in an alloy, and dewetting of thin solid films deposited on substrates; see \cite{Mullins-1957,Bao-2017-SIAP}. Since the equations of surface diffusion flow \eqref{SF_Eq} and Willmore flow \eqref{Willmore_Eq} have similar structure, they can often be approximated by similar numerical methods based on variational formulations of the equations. Numerical approximation to the surface diffusion flow \eqref{SF_Eq} has been studied by B\"ansch, Morin \& Nochetto \cite{Bansch-Morin-Nochetto-2005} by a surface finite element method with  mesh regularization, and by Barrett, Garcke \& N\"urnberg \cite{Barrett2007441} based on a similar variational formulation as for their Willmore flow algorithm. 
Recently, Bao et al. \cite{Bao-2017-SIAP,Bao-2017-JCP} and Zhao et al. \cite{Zhao-Jiang-Bao-2020} proposed finite element methods for the anisotropic surface diffusion flow with contact line migration in studying the evolution of solid-thin films on a substrate. 

Similar to the case of Willmore flow, convergence of finite element semi-discreti\-zations for the surface diffusion flow of \emph{graphs} and axially symmetric surfaces has been proved; see \cite{Bansch-Morin-Nochetto-2004,Deckelnick-Dziuk-Elliott-2003-SINUM,Deckelnick-Dziuk-Elliott-2005-SINUM}. However, convergence of a surface finite element algorithm for the surface diffusion flow of closed surfaces is still an open problem. 

{\it The objective of this article is to construct an evolving surface finite element algorithm that can be proved to be  convergent 
%(even of optimal order in the $H^1$-norm) 
for the Willmore flow of closed two-dimensional surfaces in three-dimensional space, and similarly for the surface diffusion flow.}

%This is accomplished by introducing a new formulation for the Willmore flow.
The key idea is to use fourth-order parabolic evolution equations for the mean curvature $H$ and the normal vector $\nu$ along the Willmore flow. We derive an algorithm based on the system that couples these evolution equations to the velocity law \eqref{Willmore_Eq}. Here, $H$ and $\nu$ are considered to be independently evolving unknowns that are not directly extracted from the surface at any given time. This is different from the previously mentioned approaches.
%, which solve \eqref{Willmore_Eq} by directly using geometric quantities  of the surface, 
%using geometric identities for surfaces or known properties of the Willmore flow.  

%The main advantage of the new formulation for the Willmore flow is that the evolution equations for $H$ and $\nu$ have complete parabolicity, which allows us to prove stability and convergence for an evolving surface finite element semi-discretization. 

This approach is motivated by our recent work on mean curvature flow in \cite{MCF}. There we used the coupled system with the evolution equations for $H$ and $\nu$, which were derived by Huisken \cite{Huisken1984} along with evolution equations for other geometric quantities, to construct the first provably convergent finite element algorithm for mean curvature flow of closed surfaces. For Willmore flow, the corresponding evolution equations of $H$ and $\nu$ are derived in this paper. Their general structure as fourth-order quasilinear parabolic equations was already derived by Kuwert \& Sch\"atzle \cite{KuwertSchaetzle_Willmore2,KuwertSchaetzle_Willmore} for proving existence results for Willmore flow, 
but the precise form of the equations as needed for computations was not given and it was not evident that the equations for $H$ and $\nu$ form a closed system that does not involve further geometric quantities.

%We solve the geometric evolution equations by using ESFEM with $C^0$ finite elements based on a mixed variational formulation, which introduces new difficulties to the stability estimates, which require us to estimating the material derivative equation of the error equation in our matrix-vector formulation. 
 
The velocity equation and the coupled fourth-order evolution equations for the normal vector and mean curvature are rewritten as a second-order differential--algebraic system. The geometric evolution equations have a particular \emph{anti-symmetric} structure. 
Based on a mixed variational formulation, we use evolving surface finite elements (of degree at least $2$) for the semi-discretization of the coupled system, which preserves the anti-symmetry of the second-order system. 
The velocity law is approximated simply by finite element interpolation, in contrast to enforcing it using a Ritz projection as in \cite{MCF} for mean curvature flow.

Optimal-order $H^1$-norm semi-discrete convergence estimates are shown for all variables $X,v,\nu,H$ by clearly separating the issues of stability and consistency. %The stability analysis is entirely free from geometric arguments.
%Since some of the variables in the system are determined by algebraic equations, we need to modify the initial values in order to obtain optimal-order error estimates. 
Convergence is shown towards sufficiently regular solutions of Willmore flow, which excludes the formation of singularities (within the considered time interval). % kb: To make clear that a singular trajectory is not completely excluded.

The main issue in the paper is proving {\it stability}, which is here understood as bounding the errors in terms of consistency defects and initial errors. For the velocity law, stability is shown using a stability bound for the interpolation of products of surface finite element functions. The main idea for the stability estimates for the second-order system for the geometric variables is to exploit the \emph{anti-symmetric} structure of the semi-discrete error equations and combine it with \emph{multiple energy estimates}, testing with both the errors and their time derivative. The structure of the energy estimates is sketched in Figure~\ref{fig:energy estimates}. The proof is performed in the matrix--vector formulation of the numerical method, and it uses technical lemmas relating different finite element surfaces that were shown in \cite[Section~4]{KLLP2017} and \cite[Section~7.1]{MCF}. 
A key step in the proof is to establish $W^{1,\infty}$-norm error bounds for all variables, which are obtained from the time-uniform $H^1$-norm error bounds using an inverse estimate. 
The stability analysis is completely independent of geometric arguments. 

The {\it consistency} analysis, i.e.~proving estimates for the defects (the residuals obtained upon inserting
appropriate projections of the exact solution into the method) and their time derivatives is based on \cite[Section~8]{KLLP2017} and \cite[Section~9]{MCF}, which in turn are based on geometric estimates  in \cite{Demlow2009,Dziuk88,DziukElliott_ESFEM,DziukElliott_L2,Kovacs2017}.
Together with error bounds for the initial values, it uses geometric error estimates, interpolation and Ritz map error bounds. Due to a non-symmetric divergence term for one of the variables a modified Ritz map is used, similar to the one in \cite{LubichMansour_wave}.

The paper is organized as follows. 
Section~\ref{section:Willmore flow}, besides introducing basic notations and geometric concepts, is mainly devoted to deriving the evolution equations for the mean curvature and the normal vector along Willmore flow. The coupled second-order system and its weak formulation, which serves as the basis of the algorithm, is presented there. 
Section~\ref{section:ESFEM} describes the evolving surface finite element semi-discretization, the matrix--vector formulation, and the modified semi-discrete problem to correct the initial values of the algebraic variables.
Section~\ref{section: main result} states the main result of the paper: optimal-order semi-discrete error bounds in the $H^1$-norm for the errors in positions, velocity, normal vector and mean curvature.
In Section~\ref{section:stability} we prove the stability result after presenting the required auxiliary results.
Section~\ref{section:Defect} contains the consistency analysis: the bounds for the defects and their time derivatives, and estimates for the errors in the initial values.
Section~\ref{section:numerics} presents numerical experiments illustrating and complementing our theoretical results.

%[???] 
%Due to the length of the paper, many proofs are greatly shortened. A version with complete proofs can be found on our homepage at \url{https://na.uni-tuebingen.de/...???}.
%

\section{Evolution equations for Willmore flow}
\label{section:Willmore flow}

\subsection{Basic notions and notation}%\footnote{This subsection is taken verbatim from \cite{MCF}.} }
\label{subsection: basic notions}

{\it (The text of this preparatory subsection is taken verbatim from \cite[Section~2.1]{MCF}.)} 
We consider the evolving two-dimensional closed surface $\Gamma(t)\subset\R^3$ as the image
$$
	\Ga(t) = \{ X(p,t) \,:\, p \in \Ga^0 \}
$$
of a smooth mapping $X:\Ga^0\times [0,T]\to \R^3$ such that $X(\cdot,t)$ is an embedding for every $t$. Here, $\Ga^0$ is a smooth closed initial surface, and $X(p,0)=p$. 
In view of the subsequent numerical discretization, it is convenient to think of $X(p,t)$ as the position at time $t$ of a moving particle with label $p$, and of $\Ga(t) $ as a collection of such particles. 
To indicate the dependence of the surface on~$X$, we will write
$$
	\Ga(t) = %\Ga[X_t] =
	\Ga[X(\cdot,t)] , \quad\hbox{ or briefly}\quad \Ga[X]
$$
when the time $t$ is clear from the context. The {\it velocity} $v(x,t)\in\R^3$ at a point $x=X(p,t)\in\Gamma(t)$  equals
\begin{equation} \label{velocity}
	\partial_t X(p,t)= v(X(p,t),t).
\end{equation}
For a known velocity field  $v$,  the position $X(p,t)$ at time $t$ of the particle with label $p$ is obtained by solving the ordinary differential equation \eqref{velocity} from $0$ to $t$ for a fixed $p$.

For a function $u(x,t)$ ($x\in \Gamma(t)$, $0\le t \le T$) we denote the {\it material derivative} (with respect to the parametrization $X$) as
$$
\mat u(x,t) = \frac \d{\d t} \,u(X(p,t),t) \quad\hbox{ for } \ x=X(p,t).
$$
For the following notions, see the review \cite{DeckelnickDE2005} or \cite[Appendix~A]{Ecker2012} or any textbook on differential geometry. 
On any regular surface $\Gamma\subset\R^3$, we denote by $\nabla_{\Ga}u:\Gamma\to\R^3$ the  {\it tangential gradient} of a function $u:\Gamma\to\R$, and in the case of a vector-valued function $u=(u_1,u_2,u_3)^T:\Gamma\to\R^3$, we let
$\nabla_{\Ga} u = (\nabla_{\Ga}u_1, \nabla_{\Ga}u_2, \nabla_{\Ga}u_3)$. We thus use the convention that the gradient of $u$ has the gradient of the components as column vectors. 
We denote by $\nabla_{\Ga} \cdot f$ the {\it surface divergence} of a vector field $f$ on $\Gamma$, 
and by %$\laplace_{\Ga[X]} u=\nabla_{\Ga[X]}\cdot \nabla_{\Ga[X]}u$
$\laplace_{\Ga} u=\nabla_{\Ga}\cdot \nabla_{\Ga}u$ the {\it Laplace--Beltrami operator} applied to $u:\Gamma\to\R$. In the case of a
vector-valued function $u=(u_1,u_2,u_3)^T:\Gamma\to\R^3$, we set componentwise $\laplace_{\Ga} u = (\laplace_{\Ga} u_1,\laplace_{\Ga} u_2,\laplace_{\Ga} u_3)^T$. 
(In the case of a tangential vector field $u$, this componentwise Laplace--Beltrami operator differs from the intrinsic definition of the Laplace--Beltrami operator on tangential vector fields.)

%{\color{blue}If $f$ is a general vector field  on $\Gamma$, then we define its surface divergence as 
%$$\nabla_{\Ga} \cdot f:=\nabla_{\Ga} \cdot P_{\rm tan}f ,$$ where $P_{\rm tan}f$ denotes the tangential component of $f$ at each point of the surface.} 

We denote the unit outer normal vector field to $\Gamma$ by $\n:\Gamma\to\R^3$. Its surface gradient contains the (extrinsic) curvature data of the surface $\Gamma$. At every $x\in\Gamma$, the matrix of the {\it extended Weingarten map},
$$
	A(x)=\nabla_\Gamma \n(x),
$$ 
is a symmetric $3\times 3$ matrix (see, e.g., \cite[Proposition~20]{Walker2015}). Apart from the eigenvalue $0$ with eigenvector $\n(x)$, its other two eigenvalues are
the principal curvatures $\kappa_1$ and $\kappa_2$ at the point $x$ on the surface. They determine the fundamental quantities
%\begin{align}
%\label{Def-H-A2}
$$
	H:={\rm tr}(A)=\kappa_1+\kappa_2, \qquad 
	|A|^2 = \kappa_1^2 +\kappa_2^2 ,
$$
%\end{align}
where $|A|$ denotes the Frobenius norm of the matrix $A$.
Here, $H$ is called the {\it mean curvature} (as in most of the literature, we do not put a factor 1/2). 
%{\it Gaussian curvature} is
%$$
%K:=  \kappa_1 \kappa_2 = \half (H^2 - |A|^2) .
%$$

\subsection{Evolution equations for normal vector and mean curvature of a surface moving under Willmore flow}
Willmore flow  sets the normal velocity $V=v\cdot\nu$ of the surface $\Ga[X]$ to $V$ of \eqref{Willmore_Eq}, that is,
\begin{equation}
\label{eq:Willmore-velocity}
	V=  \Delta_{\Ga[X]}H  +Q \quad\text{ with }\quad Q= - \half H^3 + |A|^2 H.
\end{equation}
 (In the surface diffusion flow~\eqref{SF_Eq}, we have $Q=0$.) 
 In this paper we choose the velocity \eqref{velocity} to have only a normal component, so that
\begin{equation}\label{V}
v= V\nu.
\end{equation}
The geometric quantities on the right-hand sides of \eqref{eq:Willmore-velocity}--\eqref{V} satisfy the following evolution equations.
\begin{lemma}
\label{lemma:evolution equations}
	For a regular surface $\Ga(t)=\Ga[X(\cdot,t)]$ moving under Willmore flow, the outer normal vector $\n$ and the mean curvature $H$ satisfy the following equations, with the extended Weingarten map $A=\nb_{\Ga} \n$, and with %Gaussian curvature $K$ and 
	$Q$ of \eqref{eq:Willmore-velocity} (we omit the ubiquitous argument $t$),
	\begin{align}
		\label{Eq_H}
		\mat H = &\ - \big(  \laplace_{\Ga} + |A|^2 \big) ( \Delta_{\Ga}H + Q ) , \\
		\label{Eq_n} 
		\mat \n =  &\ \big( - \laplace_{\Ga} + (HA - A^2) \big) ( \Delta_{\Ga} \n + |A|^2 \n ) + |\nb_{\Ga} H|^2 \nu \\
		&\ - 2 \big( \nabla_{\Ga} \cdot (A\nb_{\Ga} H) \big) \, \n 
		- A^2 \nb_{\Ga} H 
		- \nb_{\Ga}  Q . \nonumber
	\end{align}
\end{lemma}
To our knowledge, the exact formulation of these evolution equations has not been reported  in the literature before. In~\cite{KuwertSchaetzle_Willmore}, the general structure of the evolution equations as fourth-order quasilinear parabolic equations is given but the lower-order terms are not stated explicitly and are therefore not available for computations. 

A noteworthy feature of these evolution equations is that $\nb_{\Ga} A$ and the Hessian of $H$ do not appear, but only $\nb_{\Ga} H$ and the term $\nabla_{\Ga} \cdot (A\nb_{\Ga} H)$ in divergence form. This will allow us to work with the evolution equations for only $H$ and $\nu$ in the numerical discretization, without the need for an evolution equation for $A$ or further geometric quantities.

The following commutator formula will be needed to prove Lemma~\ref{lemma:evolution equations}. Its proof is given in the Appendix. 
\begin{lemma}
\label{lemma:gradient-laplace commutator formula}
For any regular surface $\Ga$ and for any regular func\-tion~$f$ on $\Ga$ the following commutator formula holds true:
\begin{equation}
\label{eq:gradient-laplace commutator formula}
	\begin{aligned}
		&\laplace_{\Ga} \nbg f - \nbg \laplace_{\Ga} f \\
		&= (H A - A^2) \, \nb_{\Ga} f 
		+ (\Delta_\Gamma \nu \cdot \nbg f) \nu 
		- 2\big( \nbg \cdot (A\nbg f) \big) \, \nu 
		- A^2 \nb_{\Ga} f .
	\end{aligned}
\end{equation}
\end{lemma}

\medskip\noindent
We remark that for the intrinsic Laplace--Beltrami operator applied to the tangential vector field $\nbg f$ a simpler commutator formula holds, with only the first term on the right-hand side. However, for our numerical purposes we need the componentwise Laplace--Beltrami operator as considered here. Since it can be shown that $H A - A^2 = K(I-\n\n^T)$, the first term on the right-hand side actually simplifies to $K \nb_{\Ga} f$. We prefer, however, to work with the first term as stated, since in the numerical method $\nb_{\Ga} f $ will be replaced by a vector field that is only approximately tangential.

%\bby
%\begin{remark}
%In the weak formulation, we may use integration by parts for the new term 
%\begin{align*}
%	\int_{\Gamma[X]}
%	- 2\, \nabla_{\Gamma}\cdot (A\nb_{\Ga} H) \, \nu \cdot \varphi 
%	&=  \int_{\Gamma[X]}
%	 2\, (A\nb_{\Ga} H) \cdot \nabla_{\Gamma}(\nu \cdot \varphi) \\
%	 &= \int_{\Gamma[X]}
%	 2\, (A\nb_{\Ga} H) \cdot (A\varphi) 
%	 + 
%	 2\, (A\nb_{\Ga} H) \cdot [(\nabla_{\Gamma}\varphi)\nu] . 
% \end{align*}
%\end{remark}
%\eby

\begin{proof}[of Lemma~\ref{lemma:evolution equations}]
For a general regular surface flow, 
%for the normal vector and mean curvature there holds (cf.~\cite{Mantegazza,Ecker2012}, \cite[Lemma~2.16]{BGN_survey})
%\begin{equation}
%\label{eq:normal and mean curvature PDE}
%	\nabla_{\Ga[X]} H = \Delta_{\Ga[X]} \n + |A|^2 \n .
%\end{equation}
the material derivatives of the outer normal vector and mean curvature satisfy, with $\Gamma=\Gamma[X(\cdot,t)]$ for short, 
%differential equations involving the normal velocity $V$,
\begin{align}
\label{eq:material derivative of normal}
	\mat \n = &\ - \nabla_{\Ga} V , \\
\label{eq:material derivative of mean curvature}
	\mat H = &\ - \laplace_{\Ga} V - |A|^2 V,
\end{align}
see \cite{Mantegazza,Ecker2012}, or Lemma~2.37~(ii) and Lemma~2.39~(ii) in \cite{BGN_survey}\footnote{In \cite{BGN_survey}  the opposite sign convention for the normal vector is used.}.
%[kb: The signs here are consistent with mean curvature flow $V=-H$.] 

The evolution equation  \eqref{Eq_H}  for $H$ is thus obtained directly by inserting \eqref{eq:Willmore-velocity} into \eqref{eq:material derivative of mean curvature}.

For the derivation of the evolution equation for $\n$ we use \eqref{eq:material derivative of normal} and \eqref{eq:Willmore-velocity}, and obtain
%	that on any surface $\Ga[X]$ we have (see \cite[p.~33]{Mantegazza})
%	\begin{equation*}
%		\mat \n = - (\nb_{\Ga[X]} v) \n .
%	\end{equation*}
%	
%	We start by plugging in the velocity law \eqref{eq:Willmore-velocity} into the above equation, and using the product rule and $\nb_{\Ga[X]} \n \n = 0$ again, we obtain
\begin{align*}
	\mat \n 
	%= &\ - \nb_{\Ga} \Big( \Big( \Delta_{\Ga}H - \Half H^3 + |A|^2 H \Big) \n \Big) \n \\
	= &\ - \nb_{\Ga} \bigl( \Delta_{\Ga}H + Q \bigr) .
\end{align*}

We now use the commutator formula \eqref{eq:gradient-laplace commutator formula} with $f=H$, together with the following expression for the tangential gradient of mean curvature: 
\begin{equation} \label{nabla-h}
	\nbg H = \Delta_{\Ga} \n + |A|^2 \n ,
\end{equation}
see \cite[(A.9)]{Ecker2012} or \cite[Proposition~24]{Walker2015}. This yields
\begin{align*}
	\mat \n = &\  - \nb_{\Ga} \Delta_{\Ga}H  
	- \nb_{\Ga} Q  \\
	= &\ \big( - \laplace_{\Ga} + (H A - A^2) \big) \, \nbg H + (\Delta_\Gamma \nu \cdot \nb_{\Ga} H) \nu \\
	&\ - 2 \big( \nabla_{\Gamma}\cdot (A\nb_{\Ga} H) \big) \, \n 
	- A^2 \nb_{\Ga} H 
	- \nb_{\Ga}  Q \\
	= &\ \big( - \laplace_{\Ga} + (H A - A^2) \big) (\Delta_{\Ga} \n + |A|^2 \n) + ( (\nbg H - |A|^2 \n) \cdot \nb_{\Ga} H) \nu \\
	&\ - 2 \big( \nabla_{\Gamma} \cdot (A\nb_{\Ga} H) \big) \, \n 
	- A^2 \nb_{\Ga} H 
	- \nb_{\Ga}  Q , 
\end{align*}
which becomes the stated evolution equation for the normal vector on noting that $\n\cdot \nb_{\Ga} H=0$.
\qed
\end{proof}

\subsection{The system of equations used for discretization}
Collecting the above equations, we have reformulated Willmore flow as the system of nonlinear fourth-order parabolic equations \eqref{Eq_H}--\eqref{Eq_n} on the surface coupled to the ordinary differential equations \eqref{velocity} for the surface positions 
via the velocity law  \eqref{eq:Willmore-velocity}--\eqref{V}.
%\begin{align*}
%\pa_t X = &\ v \circ X  , \\
%v = &\ \Big( \Delta_{\Ga[X]}H - \half H^3 + |A|^2 H \Big) \n , \\
%\mat H = &\ \Big( - \laplace_{\Ga[X]} - |A|^2 \Big) \Big( \Delta_{\Ga[X]}H - \half H^3 + |A|^2 H \Big) , \\
%\mat \n = &\ - \laplace_{\Ga[X]} \big(\laplace_{\Ga[X]}\n + |A|^2 \n \big) + \nb_{\Ga[X]} \Big( \half H^3 - |A|^2 H \Big) \\
%&\ + (HA - A^2)  \nb_{\Ga[X]} H  ,
%\end{align*}
%[kb: I think $\pa_t X = v$ should be clear, sometimes the notation $\circ X$ can even overcomplicate things.] 
%where, as always, $A=\nb_{\Ga[X]} \n$. 

By introducing two new variables $\V$ and $z$ (we note that $V$ is the normal velocity of the surface from \eqref{eq:Willmore-velocity}), the above system of fourth-order equations is reformulated as a system of formally second-order problems:
\begin{subequations}
\label{eq:Willmore evolution equations - system}
	\begin{alignat}{3}
%		\pa_t X = &\ v  \circ X  ,  \label{X-equation} \\
%		v = &\ \V \n , \label{v-equation}
%		\intertext{together with} 
		\mat H = &\  - \laplace_{\Ga[X]}\V - |A|^2 \V , \\
		\V = &\ \Delta_{\Ga[X]}H + Q ;
		\\[2mm]
		\mat \n = &\ - \laplace_{\Ga[X]} z + (H A - A^2) \, z 
		+ |\nabla_{\Gamma[X]}  H|^2 \nu 
		- 2 \big( \nabla_{\Gamma[X]} \cdot (A\nb_{\Ga[X]} H) \big) \, \n   \nonumber\\
		&\ \ \phantom{- \laplace_{\Ga[X]} z + K z } - A^2 \nb_{\Ga[X]} H  - \nb_{\Ga[X]}Q , \\
		z = &\ \laplace_{\Ga[X]}\n + |A|^2 \n. 
	\end{alignat}
\end{subequations}
This system is coupled with the equations for position and velocity,
\begin{subequations}
\label{Xv-eqs}
   \begin{alignat}{3}
		\pa_t X = &\ v  \circ X  ,  \label{X-equation} \\
		v = &\ \V \n . \label{v-equation}
   \end{alignat}
\end{subequations}
Note that the normal velocity $V$ appears twice: in the velocity equation \eqref{v-equation} and in the evolution equation for mean curvature.

The numerical discretization is based on a weak formulation of \eqref{eq:Willmore evolution equations - system}: we search for $(H,V,\n,z)$ such that %, with the abbreviation $\Q = \tfrac12 H^3 - |A|^2 H$ for the cubic term,
\begin{subequations}
\label{eq:weak form}
	\begin{align}
%		% velocity v
%		\label{eq:weak form - v}
%		& \int_{\Ga[X]} \!\!\! \nb_{\Ga[X]} v \cdot  \nb_{\Ga[X]} \phiv + 
%		\int_{\Ga[X]} \!\!\! v \cdot \phiv 
%		\nonumber \\ &
%		= \int_{\Ga[X]} \!\!\! \nb_{\Ga[X]}( \V \n) \cdot \nb_{\Ga[X]}\phiv + 
%		\int_{\Ga[X]} \!\!\! \V \n \cdot \phiv , \quad \andquad \\[4mm]
		% mean curvature H
		\label{eq:weak form - H}
		&	 \int_{\Ga[X]} \!\!\! \mat H \, \phiH - \int_{\Ga[X]} \!\!\! \nb_{\Ga[X]} \V \cdot  \nb_{\Ga[X]} \phiH = -\int_{\Ga[X]} \!\!\! |A|^2 \,  \V \, \phiH ,
		 \\[1mm]
		% normal velocity V
		\label{eq:weak form - V}
		& \int_{\Ga[X]} \!\!\! \V \phiw + \int_{\Ga[X]} \!\!\! \nb_{\Ga[X]} H \cdot \nb_{\Ga[X]} \phiw = \int_{\Ga[X]} \!\!\!  \Q \, \phiw ; 
		\\[4mm]
		% normal vector \nu
		\label{eq:weak form - nu}
		&	 \int_{\Ga[X]} \!\!\! \mat \n \cdot \phin - \int_{\Ga[X]} \!\!\! \nb_{\Ga[X]} z \cdot \nb_{\Ga[X]} \phin 
		=  \int_{\Ga[X]} \! (H A - A^2) z \cdot \phin 
		\nonumber \\
		& \qquad\qquad\quad		
		 +  \int_{\Ga[X]} \big( |\nb_{\Ga[X]} H|^2 \nu + A^2 \nb_{\Ga[X]} H \big) \cdot \phin 
		\nonumber \\
%		& \qquad\qquad\quad		
%		+ 2 \int_{\Ga[X]} \!\!\! (A\nbg H) \cdot \nbg(\n \cdot \phin) 
%		- 2 \int_{\Ga[X]} \!\!\! (\n \cdot \phin) H \,(\nu \cdot A\nbg H) % kb: this term is zero
		& \qquad\qquad\quad		
		 + 2 \int_{\Ga[X]} \!\!\! (A\nb_{\Ga[X]} H) \cdot (\nb_{\Ga[X]} \phin \n ) 
		\nonumber \\
		& \qquad\qquad\quad		
		+  \int_{\Ga[X]} \!\!\! \Q \,\nb_{\Ga[X]} \cdot \phin - \int_{\Ga[X]} \!\!\! \Q H \,\nu \cdot \phin,
		\\
		% auxiliary variable z
		\label{eq:weak form - z}
		& \int_{\Ga[X]} \!\!\! z \cdot \phiwn + \int_{\Ga[X]} \!\!\! \nb_{\Ga[X]} \n \cdot \nb_{\Ga[X]} \phiwn = \int_{\Ga[X]} \!\!\! |A|^2 \n \cdot \phiwn ,
	\end{align}
\end{subequations}
for all test functions $\phiv \in H^1(\Ga[X])^3$, $\phiH \in H^1(\Ga[X])$, $\phiw \in H^1(\Ga[X])$, and $\phin \in H^1(\Ga[X])^3$, $\phiwn \in H^1(\Ga[X])^3$. Here,
we use the Sobolev space 
$H^1(\Ga)=\{ u \in L^2(\Gamma)\,:\, \nb_\Gamma u \in L^2(\Gamma) \}$. The system \eqref{Xv-eqs}--\eqref{eq:weak form} is complemented with the initial data $X^0$, $\n^0$ and $H^0$. 

We note that the  last two terms on the right-hand side of \eqref{eq:weak form - nu} result from the integration by parts formula 
(cf.~\cite[Section~2.3]{DziukElliott_acta})
$$
\int_\Ga \nb_\Ga f \cdot \varphi = - \int_\Ga  f \,\nb_\Ga \cdot \varphi + \int_\Ga f\, H\n \cdot \varphi,
$$
which we use here with $f=\Q$. The terms in the second and third lines of \eqref{eq:weak form - nu} result from this
integration by parts (used with $\n\cdot\varphi$ and $A\nb_{\Ga} H$ in the roles of $f$ and $\varphi$, respectively) and the product rule: 
\begin{align*}
	&\int_{\Gamma}
	-  \nabla_{\Gamma}\cdot (A\nb_{\Ga} H) (\nu \cdot \varphi )
	=  \int_{\Gamma}
	  (A\nb_{\Ga} H) \cdot \nabla_{\Gamma}(\nu \cdot \varphi) 
	  - \int_\Ga (A\nb_{\Ga} H) \cdot (H\n) \, (\n\cdot\varphi)
	  \\
	 &\quad = \int_{\Gamma} 
	 (A\nb_{\Ga} H) \cdot (A\varphi) 
	 + 
	 \int_{\Gamma}  (A\nb_{\Ga} H) \cdot (\nabla_{\Gamma}\varphi\,\nu) - \int_\Ga (\nb_{\Ga} H) \cdot (HA\n) \, (\n\cdot\varphi) \\
	 &\quad = \int_{\Gamma} (A^2 \nb_{\Ga} H) \cdot\varphi + 
	 \int_{\Gamma}  (A\nb_{\Ga} H) \cdot (\nabla_{\Gamma}\varphi\,\nu),
 \end{align*}
where we used $\nabla_{\Gamma}\nu=A = A^T$ and $A\n=0$. 
%
%[kb: The non-linear terms start to get really long. Maybe we should introduce an abbreviated notation for the non-linear terms, e.g.~analogous to the ones in \eqref{eq:matrix--vector form}.] 
%
%
%The integration by parts formula used in \eqref{eq:weak form - nu} should contain an additional term:
%\begin{align*}
%&
%\int_{\Ga[X]} \!\!\! \nb_{\Ga[X]} \big(\tfrac12 H^3 - |A|^2 H\big) \cdot \phin \\ 
%&=
%- \int_{\Ga[X]} \!\!\! \big(\tfrac12 H^3 - |A|^2 H\big) (\nb_{\Ga[X]} \cdot \phin)
%+\int_{\Ga[X]} \!\!\! \big(\tfrac12 H^3 - |A|^2 H\big) H (\phin\cdot\nu ) . 
%\end{align*}
%This is just like \eqref{eq:integration by parts formula - tricky}.
%
%
%
%[kb: The comparisons with other algorithms are not that useful now as for mean curvature flow! Nevertheless, here they are.] 
%
%[kb: According to Dziuk the algorithm of Rusu \cite{Rusu} was the first parametric FEM based method. We could include that as well.] 

It is instructive to compare our system \eqref{Xv-eqs}--\eqref{eq:weak form}
with the equations on which the algorithms of Dziuk \cite{Dziuk_Willmore} and of Barrett, Garcke and N\"urnberg in \cite{BGN2007,BGN2008Willmore} are based. They both use the fact that $-H\nu =  \Delta_\Gamma x_\Gamma$, with $x_\Gamma$ denoting the identity map on $\Gamma$, and also employ the above integration by parts formula (for $f=1$).
%, cf.~\cite[equation~(2.2)]{DziukElliott_ESFEM}, 
%\begin{equation}
%\label{eq:integration by parts formula - tricky}
%	\int_\Ga \nbg \cdot \vphi = - \int_\Ga H\nu \cdot \vphi .
%\end{equation}

Dziuk's algorithm is based on the following weak formulation \cite[Problem~2]{Dziuk_Willmore}: Find $w \in H^1(\Ga[X])^3$ and the velocity $v \in H^1(\Ga[X])^3$ such that
\begin{equation}
\label{eq:Dziuk alg - weak form}
	\begin{aligned}
		\int_{\Ga[X]} \!\! w \cdot \vphi^w = &\ - \int_{\Ga[X]} \!\! \nbgx x_{\Ga} \cdot \nbgx \vphi^w , \\
		\int_{\Ga[X]} \!\! v \cdot \vphi^v = &\ \half \int_{\Ga[X]} \!\! |w|^2 \nbgx \cdot \vphi^v + \int_{\Ga[X]} \!\! \nbgx w \cdot \nbgx \vphi^v \\
	 + &\ \int_{\Ga[X]} \!\!\! (\nbgx \cdot w) (\nbgx \cdot \vphi^v) - \int_{\Ga[X]} \!\!\!\! D(\vphi^v) \nbgx x_{\Ga} \cdot  \nbgx w , \\
%		\pa_t X = &\ v ,
	\end{aligned}
\end{equation}
for all $\vphi^w \in H^1(\Ga[X])^3$ and $\vphi^v \in H^1(\Ga[X])^3$, and where $D(\vphi)$ denotes the symmetric tensor $D(\vphi) = \nbgx \vphi + (\nbgx \vphi)^T$. 
The above formulation is based on the idea that if $w=-H\n=\Delta_\Ga x_\Ga$, % is the $L^2$-projection of the first variation of the area, 
then the first variation of the Willmore functional (i.e.~the negative surface velocity $-v$) can be expressed using the second expression in \eqref{eq:Dziuk alg - weak form}, see \cite[Lemma~3]{Dziuk_Willmore}.

The algorithm of Barrett, Garcke and N\"urnberg \cite{BGN2008Willmore}
%\footnote{Again, note the opposite sign convention for the normal vector. 
%In (1.16) of \cite{BGN2008Willmore}, their $w$ and $\nu$ are our $-A$ and $-\nu$, respectively, but their $\kappa$ is our $H$. The current sign in \eqref{eq:BGN alg - weak form} should be correct. 
%} 
is based on the following weak formulation: Find the velocity $v \in H^1(\Ga[X])^3$ and $H \in H^1(\Ga[X])$, $A \in H^1(\Ga[X])^{3 \times 3}$ such that
\begin{equation}
\label{eq:BGN alg - weak form}
	\begin{aligned}
		\int_{\Ga[X]} \!\! v \cdot (\nu \vphi^v) = &\ -\int_{\Ga[X]} \!\! \nbgx H \cdot \nbgx \vphi^v  + \int_{\Ga[X]} \!\! Q \vphi^v , \\
		\int_{\Ga[X]} \!\! H \nu \cdot \varphi^H = &\ \int_{\Ga[X]} \!\! \nbgx x_\Ga \cdot \nbgx \vphi^H , \\
		\int_{\Ga[X]} \!\! A \cdot \vphi^A = &\ -  \int_{\Ga[X]} \!\! \nu \cdot (\nbgx \cdot \vphi^A) + \int_{\Ga[X]} \!\! H \nu \cdot (\vphi^A \nu) 
		,
		% , \\
%		\pa_t X = &\ v ,
	\end{aligned}
\end{equation}
for all $\vphi^v \in H^1(\Ga[X])^3$, $\vphi^H \in H^1(\Ga[X])$ and $\vphi^A \in H^1(\Ga[X])^{3 \times 3}$.
The first equation in the above weak formulation is obtained by taking the scalar product of \eqref{eq:Willmore-velocity} with $\nu$, the second equation is the weak form of the identity $-H\nu =  \Delta_\Gamma x_\Gamma$, while the last equation is obtained by choosing $\vphi = \vphi^A \nu$ in the integration by parts formula
%\eqref{eq:integration by parts formula - tricky} 
and using $A = \nbgx \nu$.

The key idea for all these approaches is to formulate Willmore flow using additional variables which satisfy differential equations, derived using geometric identities for surfaces or from the properties of the Willmore flow. 
The characteristic feature of our approach via the evolution equations of Lemma~\ref{lemma:evolution equations} is that it has only parabolic equations for the natural geometric variables that already appear in the velocity law for Willmore flow \eqref{eq:Willmore-velocity}. This will allow us to prove convergence, which is not known for the previously proposed algorithms.

\section{Evolving finite element semi-discretization}
\label{section:ESFEM}

\subsection{Evolving surface finite elements}
{\it (The text of this preparatory subsection is taken verbatim from \cite[Section~3.1]{MCF}.)}
We formulate the evolving surface finite element (ESFEM) discretization for the velocity law coupled with evolution equations on the evolving surface, following the description in \cite{KLLP2017,MCF}, which is based on \cite{Dziuk88} and \cite{Demlow2009}. We use triangular finite elements on the surface and continuous piecewise polynomial basis functions of degree~$k$, as defined in \cite[Section 2.5]{Demlow2009}.

We triangulate the given smooth initial surface $\Gamma^0$ by an admissible family of triangulations $\mathcal{T}_h$ of decreasing maximal element diameter $h$; see \cite{DziukElliott_ESFEM} for the notion of an admissible triangulation, which includes quasi-uniformity and shape regularity. 
For a given triangulation of the initial surface $\Gamma^0$, 
%For a momentarily fixed $h$, 
we denote by $\bfx^0 $  the vector in $\R^{3\dof}$ that collects all nodes $p_j$ $(j=1,\dots,\dof)$ of the initial triangulation. By piecewise polynomial interpolation of degree $k$, the nodal vector defines an approximate surface $\Gamma_h^0$ that interpolates $\Gamma^0$ in the nodes $p_j$. We will evolve the $j$th node in time, denoted $x_j(t)$ with initial condition $x_j(0)=p_j$, and collect the nodes at time $t$ in a column vector% in $\R^{3\dof}$,
$$
\bfx(t) \in \R^{3\dof}. %= (x_1(t),\dots,x_\dof(t)) ^T 
$$
We just write $\bfx$ for $\bfx(t)$ when the dependence on $t$ is not important.

By piecewise polynomial interpolation on the  plane reference triangle that corresponds to every
curved triangle of the triangulation, the nodal vector $\bfx$ defines a closed surface denoted by $\Gamma_h[\bfx]$. We can then define globally continuous finite element {\it basis functions}
$$
\phi_i[\bfx]:\Gamma_h[\bfx]\to\R, \qquad i=1,\dotsc,\dof,
$$
which have the property that on every triangle their pullback to the reference triangle are polynomials of degree $k$, and which satisfy at the node $x_j$
\begin{equation*}
	\phi_i[\bfx](x_j) = \delta_{ij} \quad  \text{ for all } i,j = 1,  \dotsc, \dof .
\end{equation*}
These functions span the finite element space on $\Gamma_h[\bfx]$,
\begin{equation*}
	S_h[\bfx] = S_h(\Gamma_h[\bfx])=\spn\big\{ \phi_1[\bfx], \phi_2[\bfx], \dotsc, \phi_\dof[\bfx] \big\} .
\end{equation*}
For a finite element function $u_h\in S_h[\bfx]$, the tangential gradient $\nabla_{\Gamma_h[\bfx]}u_h$ is defined piecewise on each element.

The discrete surface at time $t$ is parametrized by the initial discrete surface via the map $X_h(\cdot,t):\Gamma_h^0\to\Gamma_h[\bfx(t)]$ defined by
$$
	X_h(p_h,t) = \sum_{j=1}^\dof x_j(t) \, \phi_j[\bfx(0)](p_h), \qquad p_h \in \Gamma_h^0,
$$
which has the properties that $X_h(p_j,t)=x_j(t)$ for $j=1,\dots,\dof$, that  $X_h(p_h,0) = p_h$ for all $p_h\in\Gamma_h^0$, and
$$
	\Gamma_h[\bfx(t)]=\Gamma[X_h(\cdot,t)],
$$
where the right-hand side equals $\{ X_h(p_h,t) \,:\, p_h \in \Ga_h^0 \}$ like in Section~\ref{subsection: basic notions}.

The {\it discrete velocity} $v_h(x,t)\in\R^3$ at a point $x=X_h(p_h,t) \in \Gamma[X_h(\cdot,t)]$ is given by
$$
	\partial_t X_h(p_h,t) = v_h(X_h(p_h,t),t).
$$
In view of the transport property of the basis functions  \cite{DziukElliott_ESFEM},
$$
	\frac\d{\d t} \Bigl( \phi_j[\bfx(t)](X_h(p_h,t)) \Bigr) =0 ,
$$
%which by integration from $0$ to $t$ yields
%$$
%	\phi_j[\bfx(t)](X_h(p_h,t)) = \phi_j[\bfx(0)](p_h).
%$$
the discrete velocity equals, for $x \in \Gamma_h[\bfx(t)]$,
$$
	v_h(x,t) = \sum_{j=1}^\dof v_j(t) \, \phi_j[\bfx(t)](x) \qquad \hbox{with } \ v_j(t)=\dot x_j(t),
$$
where the dot denotes the time derivative $\d/\d t$. 
Hence, the discrete velocity $v_h(\cdot,t)$ is in the finite element space $S_h[\bfx(t)]$, with nodal vector $\bfv(t)=\dot\bfx(t)$.
%Although it would be possible to define the basis functions using the formula $\phi_j[\bfx\t](X_h(p_h,t)) = \phi_j[\bfx(0)](p_h)$ (obtained by integrating the transport property), such a formulation would involve pullbacks, which are avoided here.
%
%??? $X_h$ ist ja mittels der Basisfunktionen definiert. Ich wuerde das weglassen. ???
%

The {\it discrete material derivative} of a finite element function $u_h(x,t)$ with nodal values $u_j(t)$ is
$$
	\mat_h u_h(x,t) = \frac{\d}{\d t} u_h(X_h(p_h,t)) = \sum_{j=1}^\dof \dot u_j(t)  \phi_j[\bfx(t)](x)  \quad\text{at}\quad x=X_h(p_h,t).
$$

%[kb: Would be nice to shorten up a little.] 

%We will determine an approximate normal vector $\n_h$ and an approximate mean curvature $H_h$ as finite element functions: for $x \in \Gamma_h[\bfx(t)]$,
%\begin{align*}
%	\n_h(x,t) &= \sum_{j=1}^\dof \n_j(t) \, \phi_j[\bfx(t)](x) , \\
%	H_h(x,t) &= \sum_{j=1}^\dof H_j(t) \, \phi_j[\bfx(t)](x).
%\end{align*}
%Note that these finite element functions are {\it not} the normal vector and the mean curvature of the discrete surface $\Gamma_h[\bfx(t)]$. 
%For a curved triangle $E_h\subset\Gamma_h[\bfx]$, let $\varphi_{_{E_h}}: E\rightarrow E_h$ denote the one-to-one and onto map (a polynomial of degree $k$) defined on the reference triangle $\widetilde E$. 

\subsection{ESFEM spatial semi-discretization}
\label{subsection:semi-discretization}

A {\it preliminary} finite element spatial semi-discretization of the fourth-order parabolic system \eqref{eq:weak form} coupled with
the velocity and position equations \eqref{Xv-eqs}
reads as follows: Find the unknown nodal vector $\bfx(t)\in \R^{3\dof}$ of the finite element surface parametrization $X_h(\cdot,t)\in S_h[\bfx^0]^3$ and the unknown finite element velocity $v_h(\cdot,t)\in S_h[\bfx(t)]^3$, and the finite element functions $H_h(\cdot,t)\in S_h[\bfx(t)]$, $\V_h(\cdot,t)\in S_h[\bfx(t)]$, and $\n_h(\cdot,t)\in S_h[\bfx(t)]^3$, $z_h(\cdot,t)\in S_h[\bfx(t)]^3$ such that
\begin{subequations}
\label{eq:xh-vh}
\begin{align}
\label{eq:xh}
	\partial_t X_h(p_h,t) = v_h(X_h(p_h,t),t), \qquad p_h\in\Ga_h^0,
\end{align}
with
\begin{align}\label{eq:vh}
v_h = \widetilde I_h (V_h \n_h),
\end{align}
\end{subequations}
where $\widetilde I_h=\widetilde I_h[\bfx]:C(\Ga_h[\bfx])\to S_h(\Ga_h[\bfx])$ denotes the  finite element interpolation operator on the discrete surface $\Ga_h[\bfx]$.

The functions $H_h$, $V_h$, $\n_h$, $z_h$ are determined by the ESFEM semi-discretization of \eqref{eq:weak form},
denoting by $A_h = \half (\nb_{\Ga_h[\bfx]} \n_h + (\nb_{\Ga_h[\bfx]} \n_h)^T)$ the symmetric part of $\nb_{\Ga_h[\bfx]} \n_h$, 
%by $K_h = \half ( H_h^2 - |A_h|^2 )$ the approximate Gaussian curvature 
and by $\Qh = -\tfrac12 H_h^3 + |A_h|^2 H_h$ the cubic term, 
\begin{subequations}
\label{eq:semidiscrete weak form}
	\begin{align}
%		% velocity v_h
%		& \int_{\Ga_h[\bfx]} \!\!\!\! \nb_{\Ga_h[\bfx]} v_h \cdot  \nb_{\Ga_h[\bfx]} \phiv_h + 
%		\int_{\Ga_h[\bfx]} \!\!\!\! v_h \cdot \phiv_h 
%		\nonumber \\ &
%		= \int_{\Ga_h[\bfx]} \!\!\!\! \nb_{\Ga_h[\bfx]}( \V_h \n_h) \cdot \nb_{\Ga_h[\bfx]}\phiv_h + 
%		\int_{\Ga_h[\bfx]} \!\!\!\! \V_h \n_h \cdot \phiv_h , \quad \\[3mm]
		% mean curvature H_h
		&	 \int_{\Ga_h[\bfx]} \!\!\!\! \mat_h H_h \, \phiH_h - \int_{\Ga_h[\bfx]} \!\!\!\! \nb_{\Ga_h[\bfx]} \V_h \cdot  \nb_{\Ga_h[\bfx]} \phiH_h  = - \int_{\Ga_h[\bfx]} \!\!\!\! | A_h |^2 \,  \V_h \, \phiH_h , \\[8pt]
		% normal velocity V_h
		\label{Vh-prelim}
		& \int_{\Ga_h[\bfx]} \!\!\!\! \V_h \phiw_h + \int_{\Ga_h[\bfx]} \!\!\!\! \nb_{\Ga_h[\bfx]} H_h \cdot \nb_{\Ga_h[\bfx]} \phiw_h =  \int_{\Ga_h[\bfx]} \!\!\! \Qh \phiw_h ;
		\\[16pt]
		% normal vector \nu_h 
		\label{nuh-prelim}
		&	 \int_{\Ga_h[\bfx]} \!\!\!\! \mat_h \n_h \cdot \phin_h 
		- \int_{\Ga_h[\bfx]} \!\!\!\!  \nb_{\Ga_h[\bfx]} z_h \cdot \nb_{\Ga_h[\bfx]} \phin_h  
		%\nonumber 
%		\\
%		&\hspace{72pt} 
		 = \int_{\Ga_h[\bfx]} \!  (H_h A_h - A_h^2) z_h \cdot \phin_h 
\nonumber 
		\\
				&\hspace{72pt}  		
		+  \int_{\Ga_h[\bfx]} \big( |\nb_{\Ga_h[\bfx]} H_h|^2 \nu_h + A_h^2 \nb_{\Ga_h[\bfx]} H_h \big) \cdot \phin_h 
		\nonumber \\
		& \hspace{72pt}  		
		+ 2 \int_{\Ga_h[\bfx]} \!\!\! (A_h\nb_{\Ga_h[\bfx]} H_h) \cdot (\nb_{\Ga_h[\bfx]} \phin_h\, \n_h ) 
		\nonumber \\
		&\hspace{72pt}  
		+ \int_{\Ga_h[\bfx]} \!\!\!  \Qh \,\nb_{\Ga_h[\bfx]} \cdot \phin_h - \int_{\Ga_h[\bfx]} \!\!\! \Qh H_h \, \nu_h\cdot \phin_h, \\[8pt]
%		 \nonumber \\
%		&\hspace{72pt} 
		% auxiliary variable z_h
		\label{zh-prelim}
		& \int_{\Ga_h[\bfx]} \!\!\!\! z_h \cdot \phiwn_h + \int_{\Ga_h[\bfx]} \!\!\!\! \nb_{\Ga_h[\bfx]} \n_h \cdot \nb_{\Ga_h[\bfx]} \phiwn_h = \int_{\Ga_h[\bfx]} \!\!\!\! | A_h |^2 \n_h \cdot \phiwn_h ,
	\end{align}
\end{subequations}
for all $\phiv_h\in S_h[\bfx(t)]^3$, $\phiH_h \in S_h[\bfx(t)]$, $\phiw_h \in S_h[\bfx(t)]$, $\phin_h \in S_h[\bfx(t)]^3$, and $\phiwn_h \in S_h[\bfx(t)]^3$. 

The initial values for the nodal vector $\bfx$ are taken as the positions of the nodes of the triangulation of the given initial surface $\Gamma^0$.
The initial data for $H_h$ and $\n_h$ are determined by Lagrange interpolation of $H^0$ and $\n^0$, respectively. 

We note that the finite element functions $\n_h$ and $H_h$ are {\it not} the normal vector and mean curvature of the discrete surface $\Ga_h[\bfx]$.

Our error analysis indicates that the above semi-discretization is not yet convergent of optimal order. To achieve this, we need to add correction terms to \eqref{Vh-prelim} and \eqref{zh-prelim}. These will be constructed in Section~\ref{subsection:semi-discretization-modified}, using the matrix--vector formulation of the next subsection.

\begin{remark} In the evolving finite element algorithm for mean curvature flow of \cite{MCF}, the velocity approximation $v_h$ was chosen as the Ritz projection of $V_h\n_h$ (with $V_h=-H_h$ for mean curvature flow), as opposed to the simpler nodal interpolation of $V_h\n_h$ proposed here. Meanwhile, we have realized that the same $H^1$ error bounds as in \cite{MCF} can be proved for mean curvature flow also with interpolation instead of the Ritz projection. Here we use interpolation for the velocity approximation from the outset, and in our proof of stability we will provide the argument that could have been used for interpolation also in the case of mean curvature flow.
\end{remark}

\subsection{Matrix--vector formulation}
\label{subsection:DAE}

We collect the nodal values in column vectors  $\bfv=(v_j) \in \R^{3N}$, $\bfH=(H_j)\in\R^\dof$, $\bfn=(\n_j) \in \R^{3N}$ and $\boldsymbol{\V} = ({\V}_j) \in \R^\dof$, $\bfz = (z_j) \in \R^{3N}$, and denote 
\begin{equation*}
	\bfu=\left(\begin{array}{c}
	\bfH\\ 
	\bfn
	\end{array}\right) \in \R^{4N} \andquad
	\bfw=\left(\begin{array}{c}
	\boldsymbol{\V}\\ 
	\bfz
	\end{array}\right) \in \R^{4N} .
\end{equation*}

We define the surface-dependent mass matrix $\bfM(\bfx) \in \R^{N \times N}$ and stiffness matrix $\bfA(\bfx) \in \R^{N \times N}$ %(cf.~\cite[Section~2.5]{KLLP2017}) 
on the surface determined by the nodal vector $\bfx$:
\begin{equation*}
	\begin{aligned}
		\bfM(\bfx)\vert_{ij} =&\ \int_{\Ga_h[\bfx]} \! \phi_i[\bfx] \phi_j[\bfx] , \\
		\bfA(\bfx)\vert_{ij} =&\ \int_{\Ga_h[\bfx]} \! \nb_{\Ga_h[\bfx]} \phi_i[\bfx] \cdot \nb_{\Ga_h[\bfx]} \phi_j[\bfx] , 
	\end{aligned}
	\qquad i,j = 1,  \dotsc,\dof ,
\end{equation*}
with the finite element nodal basis functions $\phi_j[\bfx] \in S_h[\bfx]$.
We further let, for $d=3$ or $4$ (with the identity matrices $I_d \in \R^{d \times d}$) 
$$
	\bbM^{[d]}(\bfx)= I_d \otimes \bfM(\bfx), \quad
	\bbA^{[d]}(\bfx)= I_d \otimes \bfA(\bfx).
%	\quad
%	\bfK^{[d]}(\bfx) = I_d \otimes \bigl( \bfM(\bfx) + \bfA(\bfx) \bigr).
$$
When no confusion can arise, we write $\bfM(\bfx)$ for $\bfM^{[d]}(\bfx)$ and $\bfA(\bfx)$ for $\bfA^{[d]}(\bfx)$.

We recall that we denote by $|A_h|^2$ the squared Frobenius norm of $A_h = \tfrac12 \big(\nb_{\Ga_h[\bfx]} \n_h + (\nb_{\Ga_h[\bfx]} \n_h)^T\big)$. We define the surface- and $| A_h |^2$-dependent matrix $\bfF_1(\bfx,\bfu) \in \R^{N \times N}$ by
\begin{subequations}
\label{eq:alpha mass matrix - F}
\begin{equation}
	\label{eq:alpha mass matrix - F - A_h}
	\bfF_1(\bfx,\bfu)\vert_{ij} = - \int_{\Ga_h[\bfx]} \! | A_h |^2 \,\phi_i[\bfx] \phi_j[\bfx] , \\
	\qquad i,j = 1,  \dotsc,\dof ,
\end{equation}
and similarly   the matrix $\bfF_2(\bfx,\bfu) \in \R^{3N \times 3N}$ by 
\begin{equation}
\label{eq:alpha mass matrix - F - K_h}
	\bfF_2(\bfx,\bfu)\vert_{ij} = \int_{\Ga_h[\bfx]} \! (H_h A_h - A_h^2) \,\phi_i[\bfx] \cdot \phi_j[\bfx] , \\
	\qquad i,j = 1, \dotsc,3\dof .
\end{equation}
\end{subequations} 
We then define the block diagonal matrix
\begin{equation*}
	\bfF(\bfx,\bfu) = \left(\begin{array}{cc}
	\bfF_1(\bfx,\bfu) & 0\\ 
	0 &   \bfF_2(\bfx,\bfu) 
	\end{array}\right) \in \R^{4N \times 4N} .
\end{equation*}
We define the non-linear functions $\bff(\bfx,\bfu), \bfg(\bfx,\bfu) \in\R^{4N}$ by
\begin{equation*}
	\bff(\bfx,\bfu) =
	\left(\begin{array}{c}
	0  \\ 
	\bff_{2}(\bfx,\bfu)  
	\end{array}\right) \andquad
	\bfg(\bfx,\bfu) =
	\left(\begin{array}{c}
	\bfg_{1}(\bfx,\bfu) \\ 
	\bfg_{2}(\bfx,\bfu)  
	\end{array}\right) ,
\end{equation*}
with $\bff_{2}(\bfx,\bfu)\in \R^{3N}$, $\bfg_{1}(\bfx,\bfu)\in \R^{N}$ and  $\bfg_{2}(\bfx,\bfu)\in \R^{3N}$, given by  
\begin{equation*}
	\begin{aligned}
		\bff_{2}(\bfx,\bfu)\vert_{j+(\ell - 1)N} \: &= 
		 \int_{\Ga_h[\bfx]} \big( |\nb_{\Ga_h[\bfx]} H_h|^2 \nu_h + A_h^2 \nb_{\Ga_h[\bfx]} H_h \big)_\ell \, \phi_j[\bfx] 
		\nonumber \\
		& \hspace{12pt}  		
	+ 2 \int_{\Ga_h[\bfx]} \!\!\! (A_h\nb_{\Ga_h[\bfx]} H_h) \cdot (\nb_{\Ga_h[\bfx]} \phi_j[\bfx]\, (\n_h )_\ell )
		\\
		& \hspace{12pt} 
		 +\int_{\Gamma_h[\bfx]} \!\! \Qh \, (\nb_{\Ga_h[\bfx]})_\ell \phi_j[\bfx] 
%		\\
%		 &\quad\,
		-\int_{\Ga_h[\bfx]} \!\!\! \Qh H_h \, (\nu_h)_\ell \, \phi_j[\bfx] , \\
		\bfg_{1}(\bfx,\bfu)\vert_j \qquad\quad\ &=  \int_{\Gamma_h[\bfx]}\!\! \Qh \, \phi_j[\bfx] ,
		\\
		\bfg_{2}(\bfx,\bfu)\vert_{j+(\ell-1)N} &=\int_{\Gamma_h[\bfx]}\!\! | A_h |^2 \,(\n_h)_\ell \,  \, \phi_j[\bfx] ,
		\\
%		\bfg(\bfx,\bfu,\bfw)\vert_{j+(\ell-1)N} &= -\int_{\Gamma_h[\bfx]} \!\!\!\!\!   \V_h  (\n_h  )_\ell \, \phi_j[\bfx]
%		-\int_{\Gamma_h[\bfx]} \!\!\!\!\!  \nb_{\Ga_h[\bfx]} (\V_h  (\n_h  )_\ell)\cdot  \nb_{\Ga_h[\bfx]}\phi_j[\bfx] ,
	\end{aligned}
%	\qquad (j = 1,  \dotsc, N,\ \ell=1,2,3) .
\end{equation*}
for $j = 1, \dotsc, N$ and $\ell=1,2,3$.

%, and $\bfK(\bfx)$ for $\bfK^{[d]}(\bfx)$.
%, and
%$\| \cdot \|_{H^1(\Gamma)}$ for $\| \cdot \|_{H^1(\Gamma)^4}$, etc.

The position and velocity equations \eqref{eq:xh-vh} are equivalent to 
\begin{subequations}
\label{eq:matrix-form-X-v}
\begin{align}
\label{eq:matrix-form-X}
	\dot\bfx &= \bfv,
	\\
\label{eq:matrix-form-v}
	\bfv &= \bfV \bullet \bfn,
\end{align}
\end{subequations}
where $\bullet$ denotes the componentwise product of vectors, $\bfV \bullet \bfn = (V_j\n_j)$.
With the notation introduced above, the equations \eqref{eq:semidiscrete weak form} can be written in the following matrix--vector form, where we recall that $\bfu = (\bfH;\bfn)$ and $\bfw = (\bfV;\bfz)$: 
\begin{subequations}
\label{eq:matrix--vector form}
	\begin{align}
%		\label{eq:matrix-form-v}
%		\bfK^{[3]}(\bfx)\bfv = &\ \bfg(\bfx,\bfu,\bfw) , 
%		\\
		\label{eq:matrix-form-u}
		\bfM^{[4]}(\bfx)\dot\bfu - \bfA^{[4]}(\bfx)\bfw = &\ \bfF(\bfx,\bfu)\bfw + \bff(\bfx,\bfu) , \\
		\label{eq:matrix-form-w}
		\bfM^{[4]}(\bfx)\bfw + \bfA^{[4]}(\bfx)\bfu = &\ \bfg(\bfx,\bfu) .
	\end{align}
\end{subequations}

\subsection{Modified ESFEM spatial semi-discretization}
\label{subsection:semi-discretization-modified}
The initial values of $V_h$ and $z_h$ in \eqref{Vh-prelim} and \eqref{zh-prelim}, or equivalently of $\bfw$ in the algebraic equations \eqref{eq:matrix-form-w}, are not at our disposal, but are determined by the equations. On the other hand, our error analysis yields convergence in the $H^1$-norm of optimal order $k$ only if these initial values are $O(h^k)$ close to the Ritz projection of the exact initial values in the $H^1$-norm. We can verify this only in the weaker $H^{-1}$-norm but not in the $H^1$-norm. For the method as stated, we can then show convergence in the $H^1$-norm only of reduced order $k-2$ (and only for polynomial degree $k\ge 4$). 

To obtain the full order $k$ 
(for polynomial degree $k\ge 2$), we modify equation \eqref{eq:matrix-form-w} by adding a time-independent correction term such that the initial value of the nodal vector $\bfw$ equals the nodal vector $\bar \bfw^*(0)$ of exact initial values: Let $\bar\bfw(0)$ be the vector obtained from \eqref{eq:matrix-form-w} at time $t=0$, and set
\begin{equation}
\label{vartheta}
	\bfvartheta =   \bfM^{[4]}(\bfx^0)( \bar \bfw^*(0)- \bar\bfw(0) ).
\end{equation}
We then modify \eqref{eq:matrix-form-w} to
\begin{equation}
\label{eq:matrix-form-w-mod}
		\bfM^{[4]}(\bfx)\bfw + \bfA^{[4]}(\bfx)\bfu = \bfg(\bfx,\bfu) + \bfvartheta,
\end{equation}
so that the initial value becomes $\bfw(0)=\bar \bfw^*(0)$. Note that the time-differenti\-ated equation \eqref{eq:matrix-form-w-mod} is still the same as the time-differentiated equation \eqref{eq:matrix-form-w}, since $\bfvartheta$ does not depend on time.

As an aside,
to interpret this modification in the weak form \eqref{eq:semidiscrete weak form} (which is not necessary for the algorithm or its error analysis), we introduce
$$
\bfeta(t) = (\eta_j(t))= \bfM^{[4]}(\bfx(t))^{-1}\bfvartheta
$$
with $\eta_j=(\eta^V_j;\eta^z_j)\in \R^{1+3}$ and define the finite element functions
\begin{align*}
\eta_h^V(x,t) &= \sum_{j=1}^N \eta_j^V(t) \varphi_j [\bfx(t)](x) \qquad\ \text{ for } \quad x \in \Gamma_h[\bfx(t)], 
\\
\eta_h^z(x,t) &= \sum_{j=1}^N \eta_j^z(t) \,\varphi_j [\bfx(t)](x) \qquad\ \text{ for } \quad x \in \Gamma_h[\bfx(t)]. 
\end{align*}
Then, the modification \eqref{eq:matrix-form-w-mod} corresponds to modifying 
\eqref{Vh-prelim} and \eqref{zh-prelim} to
\begin{subequations}
\label{eq:semidiscrete weak form - mod}
	\begin{align}
		% normal velocity V_h
		\label{Vh-prelim-mod}
		 \int_{\Ga_h[\bfx]} \!\!\!\! \V_h \phiw_h + \int_{\Ga_h[\bfx]} \!\!\!\! \nb_{\Ga_h[\bfx]} H_h \cdot \nb_{\Ga_h[\bfx]} \phiw_h = &  \int_{\Ga_h[\bfx]} \!\!\! \Qh \phiw_h 
%		 \nonumber
%		\\
%		& 
		+ \ \int_{\Ga_h[\bfx]} \!\!\! \eta_h^V \phiw_h ,
		\\[2mm]
		% auxiliary variable z_h
		\label{zh-prelim-mod}
		 \int_{\Ga_h[\bfx]} \!\!\!\! z_h \cdot \phiwn_h + \int_{\Ga_h[\bfx]} \!\!\!\! \nb_{\Ga_h[\bfx]} \n_h \cdot \nb_{\Ga_h[\bfx]} \phiwn_h = & \int_{\Ga_h[\bfx]} \!\!\!\! | A_h |^2 \n_h \cdot \phiwn_h 
%		 \nonumber
%				\\
%						& 
						+ \ \int_{\Ga_h[\bfx]} \!\!\! \eta_h^z \,\phiwn_h .
	\end{align}
\end{subequations}

\subsection{Lifts}
\label{subsec:lifts}
 
{\it (The text of this preparatory subsection is taken verbatim from \cite[Section~3.4]{MCF}.)} 
We need to compare functions on the {\it exact surface} $\Gamma(t)=\Gamma[X(\cdot,t)]$ with functions on the {\it discrete surface} $\Gamma_h(t)=\Gamma_h[\bfx(t)]$. To this end, we further work with functions on the {\it interpolated surface} $\Gamma_h^*(t)=\Gamma_h[\xs(t)]$, where
$\xs(t)$ denotes the nodal vector collecting the grid points $x_j^*(t)=X(p_j,t)$ on the exact surface.

Any finite element function $w_h:\Gamma_h(t)\to\R^m$ ($m=1$ or 3) on the discrete surface, with nodal values $w_j$, is associated with a finite element function $\widehat w_h$ on the interpolated surface $\Gamma_h^*(t)$ that is defined by 
$$
\widehat w_h  = \sum_{j=1}^N w_j \phi_j[\xs(t)].
$$
This can be further lifted to a function on the exact surface by using the \emph{lift operator} $l$, which was introduced for linear and higher-order surface approximations in \cite{Dziuk88} and \cite{Demlow2009}, respectively. 
The lift operator $l$ maps a function on the interpolated surface $\Gamma_h^*$ to a function on the exact surface $\Gamma$, provided that $\Gamma_h^*$ is sufficiently close to $\Gamma$.
The exact regular surface $\Gamma$ can be represented,  in some neighbourhood of the surface,  by a smooth signed distance function $d : \R^3 \times [0,T] \to \R$, cf.~\cite[Section~2.1]{DziukElliott_ESFEM}, such that $\Gamma$ is the zero level set of~$d$  (i.e., $x\in \Gamma$ if and only if $d(x)=0$), with negative values of $d$ in the interior.  
Using this distance function,  the lift of a continuous function $\eta_h \colon \Ga_h^* \to \R^m$ is defined as
$%\begin{equation*}
    \eta_{h}^{l}(y) := \eta_h(x)$ for $x\in\Ga_h^*$,
%\end{equation*}
where for every $x\in \Ga_h^*$ the point $y=y(x)\in\Ga$ is uniquely defined via
$%\begin{equation*}
%\label{eq: lift defining equation}
    y = x - \nu(y) d(x).
$%\end{equation*}
 
The composed lift $L$ from finite element functions on the discrete surface $\Gamma_h(t)$ to functions on the exact surface $\Gamma(t)$ via the interpolated surface $\Gamma_h^*(t)$ is denoted by 
$$
w_h^L = (\widehat w_h)^l.
$$

\section{Convergence of the modified ESFEM semi-discretization}
\label{section: main result}

We are now in the position to formulate the first main result of this paper, which yields optimal-order error bounds for the finite element semi-discretization \eqref{eq:xh-vh}--\eqref{eq:semidiscrete weak form} with the modification \eqref{eq:semidiscrete weak form - mod}
using finite elements of polynomial degree $k \geq 2$. We denote by $\Gamma(t)=\Gamma[X(\cdot,t)]$ the exact surface and by $\Gamma_h(t)=\Gamma[X_h(\cdot,t)]=\Gamma_h[\bfx(t)]$ the discrete surface at time $t$, and introduce the notation
$$
	x_h^L(x,t) =  X_h^L(p,t) \in \Gamma_h(t) \qquad\hbox{for}\quad x=X(p,t)\in\Gamma(t).
$$

\begin{theorem}
\label{MainTHM} 
	Consider the ESFEM semi-discretization \eqref{eq:xh-vh}--\eqref{eq:semidiscrete weak form} with the modification \eqref{eq:semidiscrete weak form - mod}
(i.e., \eqref{eq:matrix-form-X-v}--\eqref{eq:matrix--vector form} with the modification  \eqref{eq:matrix-form-w-mod} in matrix--vector form) of the Willmore flow problem \eqref{Xv-eqs}--\eqref{eq:weak form}, using evolving surface finite elements of polynomial degree $k \geq 2$. 
	Suppose that the problem admits an exact solution $(X,v,\n,H, z,V)$ that is sufficiently differentiable on the time interval $t\in[0,T]$, and that the flow map $X(\cdot,t):\Gamma^0\to \Gamma(t)\subset\R^3$ is non-degenerate so that $\Gamma(t)$ is a regular surface on the time interval $t\in[0,T]$. 
	
	Then, there exists a constant $h_0 > 0$ such that for all mesh sizes $h \leq h_0$ the following error bounds for the lifts of the discrete position, velocity, normal vector, mean curvature and further geometric quantities hold over the exact surface $\Ga(t)=\Ga[X(\cdot,t)]$ for $0 \leq t \leq T$:
	\begin{align*}
		\|x_h^L(\cdot,t) - \mathrm{id}_{\Gamma(t)}\|_{H^1(\Ga(t))^3} \leq &\ Ch^k, \qquad
		\|v_h^L(\cdot,t) - v(\cdot,t)\|_{H^1(\Ga(t))^3} \leq  C h^k, \\
		\|H_h^L(\cdot,t) - H(\cdot,t)\|_{H^1(\Ga(t))} \leq &\ C h^k, \qquad
	        \|\n_h^L(\cdot,t) - \n(\cdot,t)\|_{H^1(\Ga(t))^3} \leq C h^k, \\
	         \|V_h^L(\cdot,t) - V(\cdot,t)\|_{H^1(\Ga(t))} \leq &\ C h^k, \qquad
	        \|z_h^L(\cdot,t) - \nabla_{\Ga(t)}H(\cdot,t)\|_{H^1(\Ga(t))^3} \leq C h^k, \\
		\intertext{and also}
		\|X_h^l(\cdot,t) - X(\cdot,t)\|_{H^1(\Ga^0)^3} \leq &\ Ch^k ,
	\end{align*}
	where the constant $C$ is independent of $h$ and $t$, but depends on bounds of higher derivatives of the solution $(X,v,\n,H,z,V)$ of Willmore flow and on the length $T$ of the time interval.
\end{theorem}

Sufficient regularity assumptions are the following: with bounds that are uniform in $t\in[0,T]$, we assume $X(\cdot,t) \in  H^{k+1}(\Ga^0)$,
$v(\cdot,t) \in H^{k+1}(\Ga(X(\cdot,t)))^3$, and for $u=(\nu,H)$ we have $\ u(\cdot,t), \mat u(\cdot,t), \pa^{(2)} u \in W^{k+1,\infty}(\Ga(X(\cdot,t)))^4$, for $w=(V,z)$ we have $w, \mat w \in W^{k+1,\infty}(\Ga(X(\cdot,t)))^4$.

We note that the remarks made after the convergence result in \cite{MCF} apply also here, in particular on the preservation of admissibility of the triangulation over the whole time interval $[0,T]$ for sufficiently small $h$.

\section{Stability}
\label{section:stability}

\subsection{Preparation: Estimates relating different finite element surfaces}
\label{section:aux}

In our previous work \cite[Section~4]{KLLP2017} and \cite[Section~7.1]{MCF} we proved technical results relating different finite element surfaces, which we recapitulate here, {\it taking verbatim the text of \cite[Section~7.1]{MCF} in this preparatory subsection.}
%We use the following setting.

The  finite element matrices of Section~\ref{subsection:DAE} induce discrete versions of Sobolev norms. Let $\bfx \in \R^{3 N}$ be a nodal vector defining the discrete surface $\Gamma_h[\bfx]$. For any nodal vector $\bfw=(w_j) \in \R^{N}$, with the corresponding finite element function $w_h= \sum_{j=1}^\dof w_j \phi_j[\bfx] \in S_h[\bfx]$, we define the following norms, where
$ \bfK(\bfx) = \bfM(\bfx) + \bfA(\bfx)$ in the third line:
\begin{align} 
\label{M-L2}
	&  \|\bfw\|_{\bfM(\bfx)}^{2} = \bfw^T \bfM(\bfx) \bfw = \|w_h\|_{L^2(\Ga_h[\bfx])}^2 , \\
\label{A-H1}
	&  \|\bfw\|_{\bfA(\bfx)}^{2} = \bfw^T \bfA(\bfx) \bfw = \|\nb_{\Ga_h[\bfx]} 	w_h\|_{L^2(\Ga_h[\bfx])}^2 , \\
\label{K-H1}
	&  \|\bfw\|_{\bfK(\bfx)}^{2} = \bfw^T \bfK(\bfx) \bfw = \|w_h\|_{H^1(\Ga_h[\bfx])}^2 .
\end{align}
When $\bfw\in \R^{dN}$, so that the corresponding finite element function $w_h$ maps into $\R^d$, we write  in the following $\| w_h \|_{L^2(\Gamma)}$ for $\| w_h \|_{L^2(\Gamma)^d}$ and $\| w_h \|_{H^1(\Gamma)}$ for $\| w_h \|_{H^1(\Gamma)^d}$.

Let now $\bfx,\bfy \in \R^{3 N}$ be two nodal vectors defining discrete surfaces $\Gamma_h[\bfx]$ and $\Gamma_h[\bfy]$, respectively. 
We  denote the difference by $\bfe= (e_j)=\bfx-\bfy \in \R^{  3  N}$. For  $\theta\in[0,1]$, we consider the intermediate surface $\Gamma_h^\theta=\Gamma_h[\bfy+\theta\bfe]$ and the corresponding finite element functions given as
$$
	e_h^\theta=\sum_{j=1}^\dof e_j \phi_j[\bfy+\theta\bfe] ,
$$
and in the same way, for any vectors $\bfw,\bfz \in \R^\dof$,
$$
	w_h^\theta=\sum_{j=1}^\dof w_j \phi_j[\bfy+\theta\bfe] \andquad z_h^\theta=\sum_{j=1}^\dof z_j \phi_j[\bfy+\theta\bfe] .
$$
Figure~\ref{figure:relating different surfaces} illustrates the described construction.

\begin{figure}[htbp]
	\begin{center}
		\includegraphics[scale=1]{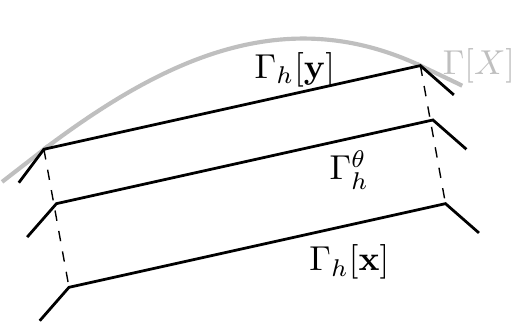}
		\caption{The construction of the intermediate surfaces $\Gamma_h^\theta$}
		\label{figure:relating different surfaces}
	\end{center}
\end{figure}

The following formulae relate the mass and stiffness matrices for the discrete surfaces $\Gamma_h[\bfx]$ and $\Gamma_h[\bfy]$. They result from the Leibniz rule and are given in Lemma 4.1 of \cite{KLLP2017}.
\begin{lemma}
\label{lemma:matrix differences}
	In the above setting, the following identities hold true:
	\begin{align}
		\label{matrix difference M}
		\bfw^T (\bfM(\bfx)-\bfM(\bfy)) \bfz =&\ \int_0^1 \int_{\Ga_h^\theta} w_h^\theta (\nabla_{\Ga_h^\theta} \cdot e_h^\theta) z_h^\theta \; \d\theta, \\
		\label{matrix difference A}
		\bfw^T (\bfA(\bfx)-\bfA(\bfy)) \bfz =&\ \int_0^1 \int_{\Ga_h^\theta} \nb_{\Ga_h^\theta} w_h^\theta \cdot (D_{\Ga_h^\theta} e_h^\theta)\nb_{\Ga_h^\theta}  z_h^\theta \; \d\theta ,
	\end{align}
	where
	$D_{\Ga_h^\theta} e_h^\theta =  \textnormal{tr}(E^\theta) I_3 - (E^\theta+(E^\theta)^T)$ with $E^\theta=\nabla_{\Ga_h^\theta} e_h^\theta \in \R^{3\times 3}$.
\end{lemma}	
The following lemma combines Lemmas 4.2 and 4.3 of \cite{KLLP2017}.
\begin{lemma}
\label{lemma:theta-independence}
	In the above setting, if 
	$$%\label{e-inf-bound}
		\| \nabla_{\Gamma_h[\bfy]} e_h^0 \|_{L^\infty(\Gamma_h[\bfy])} \le \tfrac12,
	$$
	then, for $0\le\theta\le 1$ and $1\le p \le \infty$, the finite element function \\ $w_h^\theta=\sum_{j=1}^\dof w_j \phi_j[\bfy+\theta\bfe]$ on $\Gamma_h^\theta=\Gamma_h[\bfy+\theta\bfe]$ is bounded by
	\begin{align*}
		&\| w_h^\theta \|_{L^p(\Gamma_h^\theta)} \leq c_p \, \|w_h^0 \|_{L^p(\Gamma_h^0)} , 
		\\
		&\| \nabla_{\Gamma_h^\theta} w_h^\theta \|_{L^p(\Gamma_h^\theta)} \le c_p \, \| \nabla_{\Gamma_h^0} w_h^0 \|_{L^p(\Gamma_h^0)} , 
		%\for 1\le p \le \infty,
	\end{align*}
	where $c_p$ is an absolute constant (in particular, independent of $0\le \theta\le 1$ and~$h$).	Moreover, $c_\infty=2$.
\end{lemma}

The first estimate is not stated explicitly in \cite{KLLP2017}, but follows immediately with the proof of Lemma~4.3 in \cite{KLLP2017}.

If $\| \nabla_{\Gamma_h[\bfy]} e_h^0 \|_{L^\infty(\Gamma_h[\bfy])}\le \frac14$, using the lemma for $w_h^\theta=e_h^\theta$ shows that
\begin{equation}
\label{e-theta}
	\| \nabla_{\Gamma_h^\theta} e_h^\theta \|_{L^\infty(\Gamma_h^\theta)} \le \tfrac12, \qquad 0\le\theta\le 1,
\end{equation}
and then the lemma with $p=2$ and the definition of the norms \eqref{M-L2} and \eqref{A-H1} (and interchanging the roles of $\bfy$ and $\bfy+\theta\bfe$) show that 
\begin{equation}
\label{norm-equiv}
	\begin{aligned}
		&\text{the norms $\|\cdot\|_{\bfM(\bfy+\theta\bfe)}$ are $h$-uniformly equivalent for $0\le\theta\le 1$,}
		\\
		&\text{and so are the norms $\|\cdot\|_{\bfA(\bfy+\theta\bfe)}$.}
	\end{aligned}
\end{equation}

Under the condition that 
$\eps := \| \nabla_{\Gamma_h[\bfy]} e_h^0 \|_{L^\infty(\Gamma_h[\bfy])}\le \tfrac14$, using \eqref{e-theta} in Lemma~\ref{lemma:matrix differences}  and applying the Cauchy-Schwarz inequality yields the bounds, with $c=c_\infty c_2^2$,
\begin{equation}
\label{matrix difference bounds}
	\begin{aligned}
		\bfw^T (\bfM(\bfx)-\bfM(\bfy)) \bfz \leq &\ c \eps \, \|\bfw\|_{\bfM(\bfy)} \|\bfz\|_{\bfM(\bfy)} , \\[1mm]
		\bfw^T (\bfA(\bfx)-\bfA(\bfy)) \bfz \leq &\ c \eps \, \|\bfw\|_{\bfA(\bfy)} \|\bfz\|_{\bfA(\bfy)} .
	\end{aligned}
\end{equation}
We will also use similar bounds where we use the $L^\infty$ norm of $z_h$ or its gradient and the $L^2$ norm of the gradient of $e_h$:
\begin{equation}
\label{matrix difference bounds e_x}
	\begin{aligned}
		\bfw^T (\bfM(\bfx)-\bfM(\bfy)) \bfz \leq &\ c \, \|\bfw\|_{\bfM(\bfy)} \|\bfe\|_{\bfA(\bfy)} , \\[1mm]
		\bfw^T (\bfA(\bfx)-\bfA(\bfy)) \bfz \leq &\ c \, \|\bfw\|_{\bfA(\bfy)} \|\bfe\|_{\bfA(\bfy)} .
	\end{aligned}
\end{equation}

Consider now a continuously differentiable function $\bfx:[0,T]\to\R^{3N}$ that defines a finite element surface $\Gamma_h[\bfx(t)]$ for every $t\in[0,T]$, and assume that its time derivative $\bfv(t)=\dot\bfx(t)$ is the nodal vector of a finite element function $v_h(\cdot,t)$ that satisfies
\begin{equation}
\label{vh-bound}
	\| \nabla_{\Gamma_h[\bfx(t)]}v_h(\cdot,t) \|_{L^{\infty}(\Gamma_h[\bfx(t)])} \le K, \qquad 0\le t \le T.
\end{equation}
With $\bfe=\bfx(t)-\bfx(s)=\int_s^t \bfv(r)\,dr$, %Lemmas~\ref{lemma:matrix differences} and~\ref{lemma:theta-independence} 
the bounds \eqref{matrix difference bounds}
then yield the following bounds, which were first shown in Lemma~4.1 of \cite{DziukLubichMansour_rksurf}: 

for $0\le s, t \le T$ with $K|t-s| \le \tfrac14$, 
we have with $C=c K$
\begin{equation}
\label{matrix difference bounds-t}
	\begin{aligned}
		\bfw^T \bigl(\bfM(\bfx(t))  - \bfM(\bfx(s))\bigr)\bfz \leq&\ C \, |t-s| \, \|\bfw\|_{\bfM(\bfx(t))}\|\bfz\|_{\bfM(\bfx(t))} , \\[1mm]
		\bfw^T \bigl(\bfA(\bfx(t))  - \bfA(\bfx(s))\bigr)\bfz \leq&\ C\,  |t-s| \, \|\bfw\|_{\bfA(\bfx(t))}\|\bfz\|_{\bfA(\bfx(t))}.   
	\end{aligned}
\end{equation}
Letting $s\to t$, this implies the bounds stated in Lemma~4.6 of~\cite{KLLP2017}:
\begin{equation}
\label{matrix derivatives}
	\begin{aligned}
		\bfw^T \frac\d{\d t}\bfM(\bfx(t))  \bfz \leq&\ C  \,\|\bfw\|_{\bfM(\bfx(t))}\|\bfz\|_{\bfM(\bfx(t))} , \\[1mm]
		\bfw^T \frac\d{\d t}\bfA(\bfx(t))  \bfz \leq&\ C \, \|\bfw\|_{\bfA(\bfx(t))}\|\bfz\|_{\bfA(\bfx(t))} .
	\end{aligned}
\end{equation}
Moreover, by patching together finitely many intervals over which $K|t-s| \le \tfrac14$, we obtain that
\begin{equation}
\label{norm-equiv-t}
	\begin{aligned}
		&\text{the norms $\|\cdot\|_{\bfM(\bfx(t))}$ are $h$-uniformly equivalent for $0\le t \le T$,}
		\\
		&\text{and so are the norms $\|\cdot\|_{\bfA(\bfx(t))}$.}
	\end{aligned}
\end{equation}

\subsection{Preparation: Interpolation of products of finite element functions}

For the stability of the velocity approximation
%, \bbk which uses here simply the finite element interpolation, in contrast to \cite{MCF} enforcing it via the Ritz map, \ebk
 we will need the following result.

\begin{lemma} \label{lemma:interpolation} 
For an admissible triangulation of a smooth surface  $\Gamma$, let $\Gamma_h^*$ be the interpolated surface with finite elements of polynomial degree $k\ge 1$.
Let  $\widetilde I_h^*:C(\Gamma_h^*)\to S_h(\Gamma_h^*)$ denote the finite element interpolation operator on $\Ga_h^*$. Then, the interpolation of the product of two finite element functions $a_h, b_h$ on $\Gamma_h^*$ is bounded by
$$
\| \widetilde I_h^* (a_h b_h) \|_{H^1(\Gamma_h^*)}  \le C\, \| a_h \|_{H^1(\Gamma_h^*)} \, \| b_h \|_{W^{1,\infty}(\Gamma_h^*)},
$$
where $C$ depends only on $\Gamma$ (more precisely, on bounds of higher derivatives of a parametrization of $\Gamma$), on shape-regularity and quasi-uniformity of the triangulation, and on the degree $k$.
\end{lemma}

%We will show the following estimate:
%\begin{align*}
%	\| \Ih (e_h \nu_h) - e_h \nu_h \|_{L^2(\Ga_h[\xs])} \leq &\ c \|e_h\|_{H^1(\Ga_h[\xs])} \|\nu_h\|_{L^\infty(\Ga_h[\xs])} \\
%	&\ + c \| e_h \|_{L^2(\Ga_h[\xs])} \| \nbgh \nu_h \|_{L^\infty(\Ga_h[\xs])}.
%\end{align*}

\begin{proof} Let $K$ be a curved triangle of the triangulation $\Ga_h^*$. 
Using the interpolation error estimate of \cite[Proposition~2.7]{Demlow2009}, on the element $K$ we obtain
(with different constants $c$)
\begin{align*}
	&\| \Ih^* (a_h b_h) - a_h b_h \|_{H^1(K)} \leq c h^k \|a_h b_h \|_{H^{k+1}(K)} 
	\leq c h^k \sum_{j = 0}^{k+1} |a_h b_h |_{H^{j}(K)} \\
	&\leq  c h^k \sum_{j = 0}^{k+1} \sum_{i = 0}^{j} \| \nbgh^{i} a_h \|_{L^2(K)} \| \nbgh^{j-i} b_h \|_{L^\infty(K)} \\
	& =  c \sum_{j = 0}^{k} h^{k-j} \sum_{i = 0}^{j}  h^i \| \nbgh^{i} a_h \|_{L^2(K)} \, h^{j-i}\| \nbgh^{j-i} b_h \|_{L^\infty(K)} \\
	&\ \ + c \! \sum_{i = 0}^{k} \! h^i \| \nbgh^{i} a_h \|_{L^2(K)} \, h^{k-i} \| \nbgh^{k+1-i} b_h \|_{L^\infty(K)} 
%	\\ &\ 
	+ c h^k \| \nbgh^{k+1} a_h \|_{L^2(K)} \| b_h \|_{L^\infty(K)}  .
\end{align*}
By inverse estimates, this is further bounded by
\begin{align*}
	&\| \Ih^* (a_h b_h) - a_h b_h \|_{H^1(K)} 
%	&=  c \sum_{j = 0}^{k} h^{k-j} \sum_{i = 0}^{j}  h^j \| \nbgh^{i} a_h \|_{L^2(K)} \| \nbgh^{j-i} b_h \|_{L^\infty(K)} \\
%	&\quad\ + c \sum_{i = 1}^{k+1} h^k \| \nbgh^{i} a_h \|_{L^2(K)} \| \nbgh^{k+1-i} b_h \|_{L^\infty(K)} \\
%	&\quad\ + c h^k \| a_h \|_{L^2(K)} \| \nbgh^{k+1} b_h \|_{L^\infty(K)} \\
	\leq  c  \| a_h \|_{L^2(K)} \| b_h \|_{L^\infty(K)} \\
	&\quad\ +  c \| a_h \|_{L^2(K)} \| \nbgh b_h \|_{L^\infty(K)}+
	c \| \nbgh a_h \|_{L^2(K)} \| b_h \|_{L^\infty(K)} 
	\\
	&\leq c \,\| a_h \|_{H^1(K)} \| b_h \|_{W^{1,\infty}(K)}
%	\\
%	&\ + c \| a_h \|_{L^2(K)} \| \nbgh b_h \|_{L^\infty(K)} \\
%	\leq &\ c \big( \| \nbgh a_h \|_{L^2(K)} + \| a_h \|_{L^2(K)} \big) \| b_h \|_{L^\infty(K)} \\
%	&\ + c \| a_h \|_{L^2(K)} \| \nbgh b_h \|_{L^\infty(K)} ,
\end{align*}
Squaring and summing up over the triangles then shows that
$$
\| \Ih^* (a_h b_h) - a_h b_h \|_{H^1(\Gamma_h^*)} \le c \| a_h \|_{H^1(\Gamma_h^*)} \, \| b_h \|_{W^{1,\infty}(\Gamma_h^*)}.
$$
The stated result then follows with the triangle inequality.
\qed
\end{proof}

\subsection{Defects and errors}

We choose reference finite element functions $x_h^*(\cdot,t)$, $v_h^*(\cdot,t)$, $u_h^*(\cdot,t),w_h^*(\cdot,t)$ on the interpolated surface $\Gamma_h[\xs(t)]$ with nodal vectors 
\begin{align*}
&\xs(t)\in\R^{3N},\ \ \vs(t)\in\R^{3N},\\
&\us(t)=\begin{pmatrix} \bfH^\ast(t) \\ \mathbf{n}^\ast(t) \end{pmatrix} \in\R^{4N}, \quad
\ws(t)=\begin{pmatrix} \bfV^\ast(t) \\ \mathbf{z}^\ast(t) \end{pmatrix} \in\R^{4N},
\end{align*}
which are related to the exact solution $X$, $v$ and $u=(H,\nu)$, $w=(V,z)$ as follows: $\xs(t)$ and $\vs(t)$ collect the values at the finite element nodes of $X(\cdot,t)$ and $v(\cdot,t)$, respectively. The vectors $\us(t)$ and $\ws(t)$ need to be chosen in a different way.
The vector $\us(t)$ contains the nodal values of the finite element function $u_h^*(\cdot,t)\in S_h[\xs(t)]^4$ that is determined on the interpolated surface $\Gamma_h[\xs(t)]$ by a \emph{modified} Ritz map of the exact solution component $u$ that will be defined in Section~\ref{section:Defect}.
%
%omitting the ubiquitous argument $t$,
%\begin{equation}\label{uhs-ritz}
%\int_{\Gamma_h[\xs]} \!\! \!\!\!\! \nb_{\Gamma_h[\xs]} u_h^* \cdot \nabla_{\Gamma_h[\xs]} \varphi_h + \int_{\Gamma_h[\xs]} \!\!\!\!  u_h^* \cdot \varphi_h =
%\int_{\Gamma[X]} \!\! \!\!\!\! \nabla_{\Gamma[X]} u \cdot \nabla_{\Gamma[X]} \varphi_h^l + \int_{\Gamma[X]} \!\!\!\! u \cdot  \varphi_h^l
%\end{equation} 
%for all $\varphi_h \in S_h[\xs]^4$, where again $\varphi_h^l$ denotes the lift to a function on $\Gamma[X]$. In the same way, 
Similarly, the vector $\ws(t)$ contains the nodal values of the finite element function $w_h^*(\cdot,t)\in S_h[\xs(t)]^4$ that is defined on the interpolated surface $\Gamma_h[\xs(t)]$ by the Ritz map of the solution component $w$.

The nodal vectors $\xs(t)$, $\vs(t)$ then satisfy the equations \eqref{eq:matrix-form-X-v} up to some defect~$\dv$ that will be studied in Section~\ref{section:Defect},
\begin{subequations}
\label{eq:matrix-form-X-v-star}
\begin{align}
\label{eq:matrix-form-X-star}
	\dotxs &= \vs,
	\\
\label{dv}
	\vs &= \mathbf{V}^\ast \bullet \mathbf{n}^\ast + \dv, 
\end{align}
\end{subequations}
and $\us(t)$, $\ws(t)$ satisfy the equations  \eqref{eq:matrix--vector form} up to some defects  $\bfd_{\bf u}$ and $\bfd_{\bf w}$ that will also be studied in Section~\ref{section:Defect}: 
\begin{subequations}
\label{defect vectors}
	\begin{align}
                		\label{du}
		\bbM^{[4]}(\xs)\dotus - \bbA^{[4]}(\xs)\ws = &\ \bfF(\xs,\us)\ws + \bff(\xs,\us) 
%		\nonumber \\ &\qquad \qquad \ \ 
		+ \bfM^{[4]}(\xs) \du , \\
		\label{dw}
		\bbM^{[4]}(\xs)\ws + \bbA^{[4]}(\xs)\us = &\ \bfg(\xs,\us) + \bfM^{[4]}(\xs) \dw.
	\end{align}
\end{subequations}
In the following we omit the superscript $[4]$ on $\bfM$ and $\bfA$. Furthermore, we simplify the notation and abbreviate $\bfM(\bfx(t))$ and $\bfM(\xs(t))$ to $\M(t)$ and $\Ms(t)$, respectively.  Similarly we write  $\A$ and $\As$. 

The errors between the nodal values of the numerical solutions and the nodal values of the interpolated exact values are denoted by $\ex = \bfx - \xs$, $\ev = \bfv - \vs$, $\eu = \bfu - \us$ and $\ew = \bfw - \ws$ and their corresponding finite element functions on the interpolated surface $\Ga_h[\xs]$ are denoted by $e_x$, $e_v$, $e_u$, and $e_w$, respectively.

We obtain the error equations by subtracting \eqref{defect vectors} from \eqref{eq:matrix--vector form}
with \eqref{eq:matrix-form-w-mod} instead of \eqref{eq:matrix-form-w}, 
and \eqref{eq:matrix-form-X-v-star} from \eqref{eq:matrix-form-X-v}:
\begin{subequations} 
	\label{eq:error equations}
	\begin{align}
			\label{eq:error eq - x}
		\dotex =&\ \ev , \\
		\label{eq:error eq - v}
		\ev =&\   \mathbf{V} \bullet \mathbf{n}- \mathbf{V}^\ast \bullet \mathbf{n}^\ast - \dv,
		 \\
		\label{eq:error eq - u}
		\nonumber \M \doteu - \A \ew 
		=&\ -  \big( \M-\Ms \big) \dotus 
		+ \big( \A-\As \big) \ws \\
		\nonumber &\ + \big(\bfF(\bfx,\bfu)\bfw - \bfF(\xs,\us)\ws\big) \\
		&\ + \big(\bff(\bfx,\bfu) - \bff(\xs,\us)\big) - \Ms\du , \\
		\label{eq:error eq - w}
		\nonumber \M \ew + \A \eu 
		=&\ -  \big( \M-\Ms \big) \ws 
		- \big( \A-\As \big) \us \\
		&\ + \big(\bfg(\bfx,\bfu) - \bfg(\xs,\us)\big) - \Ms\dw  + \bfvartheta .
	\end{align}
\end{subequations}
We note further that $\ex(0)=0$,  but in general $\ev(0)$, $\eu(0)$ and $\ew(0)$ are different from $0$. 
%
%\medskip
%\noindent {\bf Remark:}\, 
%$\ev(0)$ is not zero in general. If we use $v_h=I_h(V_h\nu_h)$, then $\ev(0)$ are the nodal values of the function 
%$$I_h(V_h(0)\nu_h(0))-V(0)\nu(0) ,$$
%while $V_h(0) \neq V(0)$ in general. May be we do not need $\ev(0)$ in our stability estimate. 
%
%If we really need $\ev(0)$, it is also OK. For example, if we use $v_h=I_h(V_h\nu_h)$, then we generally have 
%\begin{align*}
%v_h-v_h^*
%&=I_h[V_h(0)\nu_h(0)]- I_h[V(0)\nu(0)] \\
%&=I_h[(V_h(0)-I_hV(0))\nu_h(0)] + I_h[I_hV(0)(\nu_h(0)-I_h\nu(0))]
%\end{align*}
%and therefore
%\begin{align*}
%\| v_h-v_h^* \|_{H^1} 
%&\le
%\| V_h(0)-I_hV(0) \|_{H^1}  + \| \nu_h(0)-I_h\nu(0) \|_{H^1} 
%\le
%Ch^k , 
%\end{align*}
%where the last inequality can be guaranteed if we add a correction term by either \eqref{eq:matrix-form-w-mod} or \eqref{modified-Ah}, and the second to last inequality is due to the $H^1$ stability of $I_h$, as shown in \eqref{H1-stability-Ih}.
%\medskip
%
%
%[kb: Indeed, Buyang is correct at pointing this out. ]
%\begin{itemize}
%	\item This is not good news, I think. In our MCF paper, we have also used that the initial error in $\bfv$ is zero, but as it is not here, I think, it is not zero there either. On the other hand, it does satisfy a bound  
%\end{itemize}
%
%
We recall from \eqref{vartheta} that
$$
\bfvartheta= \M(\bfx(0)) \,{\bar\bfe}_\bfw(0) \quad\text{ with }\quad  {\bar\bfe}_\bfw(0)=  \bar\bfw^*(0)- \bar \bfw(0).
$$
This is not to be confounded with $\ew(0)=\bar\bfw^*(0) -\bfw^*(0)$, where we remind that $\bar\bfw^*(0)$ contains the values of the exact solution $w$ in the nodes, whereas $\bfw^*(0)$ is constructed by a Ritz map from the exact solution $w$ at the initial time $0$.

%%% r.h.s. without testing for copy-paste
%\begin{align*}
%	=&\ -  \big( \K-\Ks \big) \vs \\
%	&\ + \big(\bfg(\bfx,\bfu,\bfw) - \bfg(\xs,\us,\ws)\big) - \Ms \dv , \\
%	=&\ -  \big( \M-\Ms \big) \dotus 
%	+ \big( \A-\As \big) \ws \\
%	&\ + \big(\bfF(\bfx,\bfu)\bfw - \bfF(\xs,\us)\ws\big) \\
%	&\ + \big(\bff(\bfx,\bfu) - \bff(\xs,\us)\big) - \Ms\du , \\
%	=&\ -  \big( \M-\Ms \big) \ws 
%	- \big( \A-\As \big) \us \\
%	&\ + \big(\bfg(\bfx,\bfu) - \bfg(\xs,\us)\big) - \Ms\dw , \\
%\end{align*}

\subsection{Stability estimate}

We need to bound the errors at time $t$ in terms of the defects up to time $t$ and the errors in the initial values.
The errors will be estimated in the $H^1$ norm on the interpolated surface $\Gamma_h[\xs]$: for a nodal vector $\bfe$ corresponding to a finite element function $e\in S_h(\bfx^*)$, we have with the matrix $\bfK(\xs)=\bfM(\xs)+\bfA(\xs)$ that
$$
\normK{\bfe}^2 = \bfe^T  \bfK(\xs) \bfe = \| e \|_{H^1(\Gamma_h[\xs])}^2.
$$
The defect $\dv$  then needs to be sufficiently small in the $H^1$ norm, and the other defects will be required to be small in the $L^2$ norm
$$
\| \bfd \|_{\bfM(\xs)}^2 = \bfd^T \bfM(\xs) \bfd = \| d \|_{L^2(\Gamma_h[\xs])}^2,
$$
and their time derivatives as well as $ {\bar\bfe}_\bfw(0)$ will be required to be small in
 the norm given by
\begin{equation*}
	\|\bfd\|_{\star,\xs}^2 := \bfd^T \bfM(\xs)\bfK(\xs)\inv \bfM(\xs) \bfd .
\end{equation*}
By \cite[Eq.~(5.5)]{KLLP2017},  this equals
the following dual norm for the corresponding finite element function $d\in S_h[\bfx^*]$, which has  the vector of nodal values $\bfd$,
$$
\|\bfd\|_{\star,\xs} = \|d\|_{H_h\inv(\Gamma_h[\bfx^*])} := 
\sup_{0\ne\varphi_h\in S_h[\xs]} \frac{ \int_{\Gamma_h[\xs]} d  \cdot\varphi_h } { \| \varphi_h \|_{H^1(\Gamma_h[\xs])} } \, .
$$
The following result provides the key stability estimate, which bounds the errors in terms of the defects and the initial errors.
\begin{proposition}
	\label{proposition:stability} Assume that the reference finite element functions $x_h^*(\cdot,t)$, $v_h^*(\cdot,t)$, $u_h^*(\cdot,t),w_h^*(\cdot,t)$ on the interpolated surface $\Gamma_h[\xs(t)]$ have $W^{1,\infty}$ norms that are bounded independently of $h$, for all $t\in[0,T]$. 
Assume that, for some $\kappa$ with $1 < \kappa \leq k$, the defects are bounded by 
	\begin{equation}
	\label{eq:assumed defect bounds}
		\begin{aligned}
			\|\dv(t)\|_{\bfK(\xs(t))}  \leq &\ c h^\kappa , \\
			\|\du(t)\|_{\bfM(\xs(t))} + \|\dotdu(t)\|_{\star,\xs(t)}  \leq &\ c h^\kappa , \\
			\|\dw(t)\|_{\bfM(\xs(t))} + \|\dotdw(t)\|_{\star,\xs(t)} \leq &\ c h^\kappa , \\
		%	\|\dx(t)\|_{\bfM(\xs(t))} \leq &\ c h^\kappa ,
		\end{aligned}
		\qquad \hbox{for }\  0 \leq t \leq T ,
	\end{equation}
	and that the errors of the initial values are bounded by
	\begin{equation}\label{eq:assumed initial errors}
	\begin{aligned}
	  %\| \ev(0) \|_{\bfK(\bfx^0)} +
	  \| \eu(0) \|_{\bfK(\bfx^0)} + \| \ew(0) \|_{\bfK(\bfx^0)} &\le ch^\kappa, \\
	  \| {\bar\bfe}_\bfw(0) \|_{\star,\bfx^0} \ \ &\le ch^\kappa.
	  \end{aligned}
	\end{equation}
	Then, there exists  $h_0>0$ such that the following stability estimate holds for all $h\leq h_0$ and $0\le t \le T$:
	\begin{equation}
	\label{eq:stability bound}
		\begin{aligned}
			& 	  \hspace{-10pt}
                              \normKt{\ex(t)}^2 + \normKt{\ev(t)}^2 + \normKt{\eu(t)}^2 + \normKt{\ew(t)}^2 \\[1mm]
			\leq 
			&\  C \max_{0 \leq s \leq t} \big( \|\dv(s)\|_{\K(\xs(s))}^2 + \|\du\s\|_{\star,\xs(s)}^2  \big) \\
			&\  + C \! \int_0^t \! \Big(  \|\du\s\|_{\bfM(\xs\s)}^2 + \|\dotdu\s\|_{\star,\xs(s)}^2  \\
			&\ \phantom{+ C \! \int_0^t \! \Big( \ } + \|\dw\s\|_{\bfM(\xs\s)}^2 + \|\dotdw\s\|_{\star,\xs(s)}^2 \Big) \d s \\
			&\  + C \bigl(\|\eu(0)\|_{\bfK(\bfx^0)}^2 +  \|\ew(0)\|_{\bfK(\bfx^0)}^2 \bigr)  + C \| {\bar\bfe}_\bfw(0) \|_{\star,\bfx^0}^2, 
		\end{aligned}
	\end{equation}
	where $C$ is independent of $h$ and $t$, but depends on the final time $T$. 
\end{proposition}

\begin{proof} 
The proof uses energy estimates for the error equations in the matrix--vector formulation \eqref{eq:error equations} and relies on the lemmas of Section~\ref{section:aux}, which relate the different finite element surfaces.
	While this basic procedure of the proof looks similar to that of \cite{MCF} and \cite{KLLP2017}, there are substantial differences and technical difficulties that are peculiar to the fourth-order system. 
	 
%	
%	Since we need uniform-in-time $H^1$-norm error bounds (in order to control the $W^{1,\infty}$ norm of the errors in $u$ via an inverse inequality), ideally we would test with the time derivative of the error vectors, which is not feasible directly in the case of the purely algebraic error equation \eqref{eq:error eq - w}.
	
We  exploit the skew-symmetric structure of the system of error equations  \eqref{eq:error eq - u}--\eqref{eq:error eq - w}. The uniform-in-time stability estimate follows from the combination of four different sets of auxiliary energy estimates (denoted by (i)--(iv)), divided into two major parts, Parts (A.1) and (A.2), and finally combining them in Part (A.3), which gives the uniform-in-time stability bound for the errors $\eu$ and $\ew$, as illustrated in Figure~\ref{fig:energy estimates}. 
	Part (B) of the proof contains the estimates for the velocity error equation \eqref{eq:error eq - v}, and finally the two are combined in Part (C), to show the stability bound \eqref{eq:stability bound}. 
	
	\begin{figure}[htbp]
		\begin{center}
			\includegraphics[width=\textwidth]{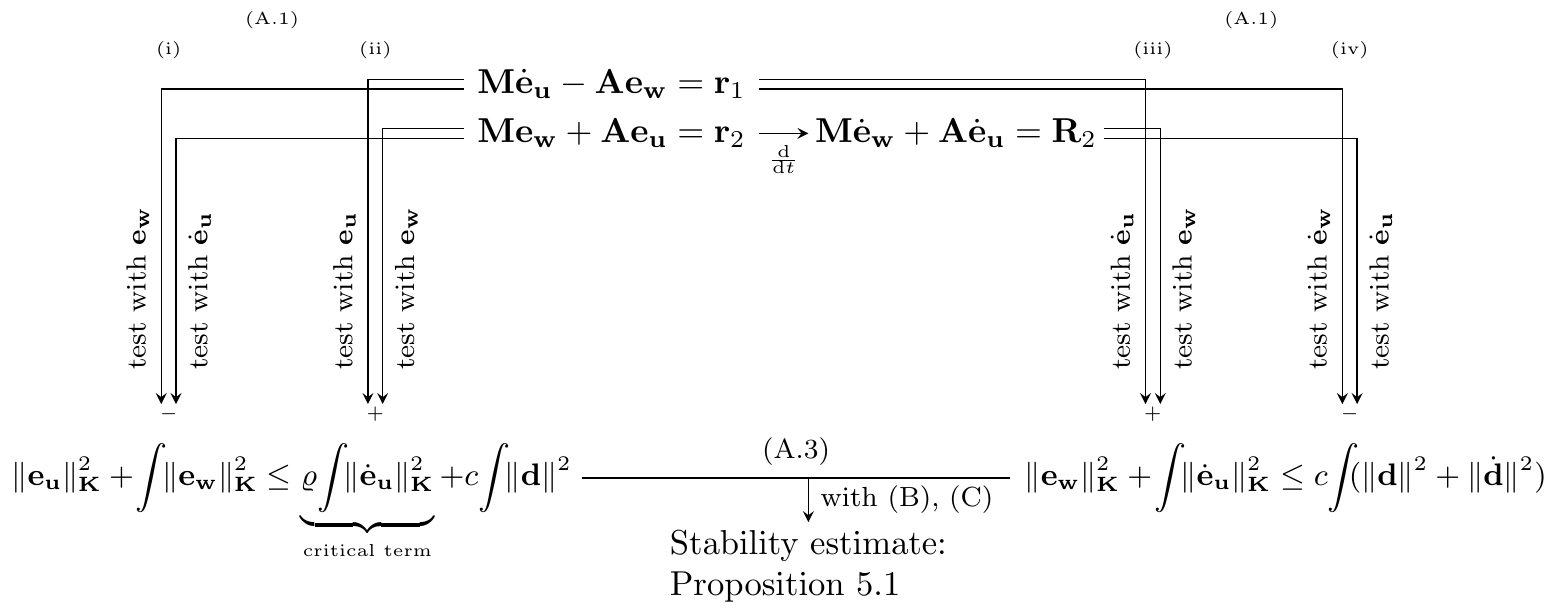}
			\caption{Sketch of the structure of the energy estimates of Part (A) for the stability proof. In the diagram, $\bfr_1$ and $\bfr_2$ denote the right-hand sides of \eqref{eq:error eq - u} and \eqref{eq:error eq - w}, respectively. (Note that, after time differentiation, $\bfR_2$ contains terms other than only the time derivative of $\bfr_2$.)}
			\label{fig:energy estimates}
		\end{center}
	\end{figure}
	
	Let $t^*\in(0,T]$ be the maximal time such that the following inequalities hold:
	\begin{equation}
	\label{eq:assumed bounds}
		\begin{aligned}
			\|e_x(\cdot,t)\|_{W^{1,\infty}(\Ga_h[\xs(t)])} \leq &\ h^{(\kappa-1)/2} , \\
			\|e_v(\cdot,t)\|_{W^{1,\infty}(\Ga_h[\xs(t)])} \leq &\ h^{(\kappa-1)/2} , \\
			\|e_u(\cdot,t)\|_{W^{1,\infty}(\Ga_h[\xs(t)])} \leq &\ h^{(\kappa-1)/2} , \\
			\|e_w(\cdot,t)\|_{W^{1,\infty}(\Ga_h[\xs(t)])} \leq &\ h^{(\kappa-1)/2} , 
		\end{aligned} \qquad \textrm{ for } \quad t\in[0,t^*].
	\end{equation}
	Note that $t^*>0$ since initially $e_x(\cdot,0)=0$ 
	and, by an inverse inequality, we have 
	\begin{align*}
	\|e_u(\cdot,0)\|_{W^{1,\infty}(\Ga_h[\xs(0)])} &\le ch^{-1}\|e_u(\cdot,0)\|_{H^{1}(\Ga_h[\xs(0)])} 
	\\
	&= ch^{-1} \normKt{\eu(0)}
	\le Ch^{\kappa-1},
	\end{align*}
	where the last inequality holds by assumption~\eqref{eq:assumed initial errors}. By the same argument, also 
	$\|e_w(\cdot,0)\|_{W^{1,\infty}(\Ga_h[\xs(0)])} \le Ch^{\kappa-1}$ 
	and, as a result of $v_h = \widetilde I_h (V_h\nu_h)$, 
	\begin{align*}
	&\|e_v(\cdot,0)\|_{W^{1,\infty}(\Ga_h[\xs(0)])} \\
	&\le c\|e_u(\cdot,0)\|_{W^{1,\infty}(\Ga_h[\xs(0)])} +c \|e_w(\cdot,0)\|_{W^{1,\infty}(\Ga_h[\xs(0)])}  
	\le Ch^{\kappa-1} .
	\end{align*}
	By continuity, we then obtain for sufficiently small $h$ that $t^*>0$ (which {\it a priori} might depend on $h$). 
	We first prove the stated error bounds for $0\leq t \leq t^*$. At the end of the proof we will show that in fact $t^*$ coincides with $T$.
%	
%	
%	To estimate the error initial value for $\bfw$, using \eqref{eq:matrix-form-w} and \eqref{dw} at $t=0$, we have
%	\begin{align*}
%		\bfM(\bfx(0)) \bfw(0) = &\ -\bfA(\bfx(0))\bfu(0) +\bfg(\bfx(0),\bfu(0)) \\
%		\bbM(\xs(0)) \ws(0) = &\ -\bbA(\xs(0))\us(0) + \bfg(\xs(0),\us(0)) + \bfM(\xs(0)) \dw(0) .
%	\end{align*}
%	Using that $\bfx(0) = \xs(0)$ and $\bfu(0) = \us(0)$ and since $\bfM(\xs(0))$ is invertible, we obtain \begin{equation*}
%		\ew(0) = \bfw(0) - \ws(0) = - \dw(0) .
%	\end{equation*}
%	Nevertheless, we only have the estimates
%	\begin{align*}
%		\|\ew(0)\|_{\Ks} = \|\dw(0)\|_{\Ks} = &\ O(h) , \\
%		\|\ew(0)\|_{\Ms} = \|\dw(0)\|_{\Ms} = &\ O(h^2) .
%	\end{align*}
%	[kb: I think the first estimate is true, but not enough.]
%	
%	
	
	% W^{1,$\infty$} bounds
	Since the reference finite element functions $x_h^*(\cdot,t),v_h^*(\cdot,t),u_h^*(\cdot,t),w_h^*(\cdot,t)$ on the interpolated surface $\Gamma_h[\xs(t)]$ have $W^{1,\infty}$ norms that are bounded independently of $h$ for all $t\in[0,T]$,
	the bounds~\eqref{eq:assumed bounds} together with Lemma~\ref{lemma:theta-independence}  imply that the $W^{1,\infty}$ norms of the ESFEM functions $x_h(\cdot,t),v_h(\cdot,t),u_h(\cdot,t)$, $w_h(\cdot,t)$ on the discrete surface $\Gamma_h[\bfx(t)]$ are also bounded independently of $h$ and $t\in[0,t^*]$, and so are their lifts to the interpolated surface $\Gamma_h[\xs(t)]$. 
%	By a similar argument, see (A.iii) in the proof of \cite[Proposition~7.1]{MCF}, the assumed smoothness of the time derivatives of the exact surface functions (e.g.~$\mat u, \pa^{(2)} u$, $\mat w$) and interpolation estimates, yield the boundedness of the time derivatives of their interpolations ($\mat_h u_h^*, \pa^{(2)}_h u_h^*$, $\mat_h w_h^*$).
	In particular,  it will be important that the discrete velocity $v_h$ with nodal vector $\bfv$ satisfies \eqref{vh-bound}.
	
	The estimate on the position errors $e_x$ in \eqref{eq:assumed bounds} and the $W^{1,\infty}$ bound on $v_h$ immediately imply that the results of Section~\ref{section:aux} apply with $\bfx$ and $\xs$ in the roles of $\bfx$ and $\bfy$, respectively. In particular, due to the bounds in \eqref{eq:assumed bounds} (for a sufficiently small $h \leq h_0$), the main condition of Lemma~\ref{lemma:theta-independence} and also \eqref{e-theta} is satisfied (with $e_h^\theta=e_x^\theta$), hence the $h$-uniform norm equivalences in \eqref{norm-equiv} and the estimates in \eqref{matrix difference bounds} and \eqref{matrix difference bounds e_x} hold between the surfaces defined by $\bfx$ and $\xs$. Similarly, again due to \eqref{eq:assumed bounds} the bound \eqref{vh-bound} also holds, and hence the estimates in \eqref{matrix difference bounds-t}, \eqref{matrix derivatives} and the $h$-uniform norm equivalences in time \eqref{norm-equiv-t} also hold. 
	When referring to these results from Section~\ref{section:aux} (\eqref{norm-equiv}--\eqref{matrix difference bounds e_x}, \eqref{matrix difference bounds-t}--\eqref{norm-equiv-t}) within the stability proof below, following the above argument, we always mean that their respective smallness assumptions are satisfied via \eqref{eq:assumed bounds}, but we will not repeat this argument at each instance.

	In the following $c$ and $C$ will denote generic constants that might take different values on different occurrences. In contrast, constants with a subscript (such as $c_0$) will play a distinctive role in the proof, and will not change their value between appearances.

	(A) \emph{Estimates for the surface PDEs:}
	
	(A.1)  We start by showing an error estimate for $\eu$ in the $\K$-norm uniformly in time for $0\le t \le t^*$. A critical term will appear, which will be controlled in Part (A.2) and eliminated in (A.3). 
	
	\emph{Energy estimate (i):} We test \eqref{eq:error eq - u} with $\ew$, and test \eqref{eq:error eq - w} with $\doteu$, so that we obtain the two equations
	\begin{align*}
		\ew^T \M \doteu - \ew^T \A \ew 
		=&\ - \ew^T  \big( \M-\Ms \big) \dotus 
		+ \ew^T \big( \A-\As \big) \ws \\
		&\ + \ew^T \big(\bfF(\bfx,\bfu)\bfw - \bfF(\xs,\us)\ws\big) \\
		&\ + \ew^T \big(\bff(\bfx,\bfu) - \bff(\xs,\us)\big) - \ew^T \Ms\du , \\
		\doteu^T \M \ew + \doteu^T \A \eu 
		=&\ - \doteu^T \big( \M-\Ms \big) \ws 
		- \doteu^T \big( \A-\As \big) \us \\
		&\ + \doteu^T \big(\bfg(\bfx,\bfu) - \bfg(\xs,\us)\big) - \doteu^T \Ms\dw  + \doteu^T \bfvartheta,
	\end{align*}
	In order to eliminate the mixed term $\ew^T \M \doteu$ (using the symmetry of the mass matrix $\bfM$), we then subtract the former equation from the latter, and obtain
	\begin{equation}
	\label{eq:i - pre estiamtes}
		\begin{aligned}
			\|\ew\|_{\A}^2 + \doteu^T \A \eu 
			= &\ + \ew^T \big( \M-\Ms \big) \dotus 
			%	\\ &\ 
			- \ew^T \big( \A-\As \big) \ws \\
			&\ - \ew^T \big(\bfF(\bfx,\bfu)\bfw - \bfF(\xs,\us)\ws\big) \\
			&\ - \ew^T \big(\bff(\bfx,\bfu) - \bff(\xs,\us)\big) 
			+ \ew^T \Ms\du , \\
			&\ - \doteu^T \big( \M-\Ms \big) \ws 
			%	\\ &\ 
			- \doteu^T \big( \A-\As \big) \us \\
			&\ + \doteu^T \big(\bfg(\bfx,\bfu) - \bfg(\xs,\us)\big) 
			- \doteu^T \Ms\dw   + \doteu^T \bfvartheta .
		\end{aligned}
	\end{equation}
	
	The terms in equation \eqref{eq:i - pre estiamtes} will now be bounded separately, using the results of Section~\ref{section:aux}. 
	
	For the second term on the left-hand side of \eqref{eq:i - pre estiamtes}, using the product rule, the symmetry of $\A$ and the bound \eqref{matrix derivatives}, we have 
	\begin{equation}
	\label{eq:i estimate lhs derivative}
		\begin{aligned}
			\doteu^T \A \eu 
			= &\ \Half \diff \Big(\eu^T \bfA \eu \Big) - \Half \eu^T \dot \A \eu \\
			\geq &\ \Half \diff \|\eu\|_{\A}^2 -c \|\eu\|_{\A}^2 .
		\end{aligned}
	\end{equation}
	
	On the right-hand side, for the matrix difference terms in the first and fourth line of \eqref{eq:i - pre estiamtes} we use \eqref{matrix difference bounds e_x}, and, recalling that $\Ms+\As=\Ks$, we obtain
	\begin{equation}
	\label{eq:i estimates matrix difference terms}
		\begin{aligned}
			\ew^T \big( \M-\Ms \big) \dotus 
			- \ew^T \big( \A-\As \big) \ws \leq &\ c \|\ew\|_{\K} \|\ex\|_{\K} , \\
			- \doteu^T \big( \M-\Ms \big) \ws 
			- \doteu^T \big( \A-\As \big) \us \leq &\ c \|\doteu\|_{\K} \|\ex\|_{\K} .
		\end{aligned}
	\end{equation}
%
%In \cite[below (7.27)]{MCF}, we have shown that 
%\begin{align*}
%&-\dot{\mathbf{e}}_{\mathbf{u}}^{T}\left(\mathbf{A}-\mathbf{A}^{*}\right) \mathbf{u}^{*} \\ 
%&\leq-\frac{\mathrm{d}}{\mathrm{d} t}\left(\mathbf{e}_{\mathbf{u}}^{T}\left(\mathbf{A}-\mathbf{A}^{*}\right) \mathbf{u}^{*}\right)+c\left\|\mathbf{e}_{\mathbf{u}}\right\|_{\mathbf{A}^*}\left(\left\|\mathbf{e}_{\mathbf{v}}\right\|_{\mathbf{K}^*}+\left\|\mathbf{e}_{\mathbf{x}}\right\|_{\mathbf{K}^*} \right)
%\end{align*}
%
For the defect terms in \eqref{eq:i - pre estiamtes}, %using the Cauchy--Schwarz inequality and the norm equivalence \eqref{norm-equiv}, 
	we obtain
\begin{align*}
\ew^T \Ms\du &= \int_{\Ga_h[\xs]} \!\! e_w \, d_u 
\\
&\le \| e_w \|_{H^1(\Ga_h[\xs])} \, \| d_u \|_{H_h^{-1}(\Ga_h[\xs])} 
= \| \ew \|_{\Ks}  \| \du \|_{\star,\xs}
\end{align*}
and in the same way
$$
- \doteu^T \Ms\dw \leq  \| \doteu \|_{\Ks}  \| \dw \|_{\star,\xs},
$$
%
%Since the defect is $O(h^k)$ in the $L^2$ norm, see \eqref{eq:stability bound}, we have a better estimate:
%$$
%- \doteu^T \Ms\dw \leq  \| \doteu \|_{\Ms}  \| \dw \|_{\Ms} , 
%$$
%which requires a weaker norm of $\dot\eu$. 
%
and similarly, using the norm equivalence \eqref{norm-equiv-t},
\begin{align*}
\doteu(t)^T \bfvartheta &= \doteu(t)^T \M(\bfx^0) {\bar\bfe}_\bfw^0 
\\
&\le \| \doteu(t) \|_{\bfK(\bfx^0)} \, \| {\bar\bfe}_\bfw^0 \|_{\star,\bfx^0}
\le c \| \doteu(t) \|_{\bfK(\bfx(t))} \, \| {\bar\bfe}_\bfw^0 \|_{\star,\bfx^0}.
\end{align*}
%
%If we use \eqref{modified-Ah}, then we do not have $\bfvartheta$ and thus do not have $\| \doteu(t) \|_{\bfK(\bfx(t))}$ on the right-hand side.  
%
We thus obtain the following bounds for the defect terms:
	\begin{equation}
	\label{eq:i estimates defect terms}
		\begin{aligned}
			\ew^T \Ms\du \leq &\  c \|\ew\|_{\K} \|\du\|_{\star,\xs}, \\
			- \doteu^T \Ms\dw \leq &\ c  \|\doteu\|_{\K}\, \|\dw\|_{\star,\xs} , \\
		        \doteu^T \bfvartheta \le &\ c \| \doteu \|_{\K} \, \| {\bar\bfe}_\bfw^0 \|_{\star,\bfx^0}.
		\end{aligned}
	\end{equation}
	For the non-linear terms we use the following bounds: The two terms involving $\bff$ and $\bfg$ are bounded exactly as the (general, locally Lipschitz) non-linear term in Part~(A.v) in the proof of Proposition~7.1 in \cite{MCF} (using the fact that by the $W^{1,\infty}$ bound for the exact solution $u_h^\ast$ and the error $e_u$ in \eqref{eq:assumed bounds}, the numerical solution $u_h$ is also bounded in the $W^{1,\infty}$ norm) and the norm equivalence \eqref{norm-equiv}. We altogether have
\begin{equation}
\label{eq:i estimates for nonlinarities}
	\begin{aligned}
		\ew^T \big(\bff(\bfx,\bfu) - \bff(\xs,\us)\big) \leq &\ c  \|\ew\|_{\K}  \big( \|\ex\|_{\K} + \|\eu\|_{\K} \big) \\
		\doteu^T \big(\bfg(\bfx,\bfu) - \bfg(\xs,\us)\big) \leq &\ c \|\doteu\|_{\M} \big( \|\ex\|_{\K} + \|\eu\|_{\K} \big) .
	\end{aligned}
\end{equation}
	
%
%\begin{equation}
%\begin{aligned}
%			\ew^T \big(\bff(\bfx,\bfu) - \bff(\xs,\us)\big) \leq &\ c \|\ew\|_{\M} \big( \|\ex\|_{\K} + \|\eu\|_{\K} \big) .
%		\end{aligned}
%	\end{equation}
%
	
%	We remark here, that for the term with $\doteu$ it seems tempting to express it using the full derivative of the product $\eu^T \big(\bfg(\bfx,\bfu) - \bfg(\xs,\us)\big)$ (similarly as we will do later in \eqref{eq:iv estimates matrix differences dotew}, or as in Part~(A.iv) in the proof of Proposition~7.1 in \cite{MCF}), however, this would still lead to a bound depending on $\|\doteu\|_{\M}$.
%	
	
	For the term involving the state-dependent mass matrix $\bfF(\bfx,\bfu)$ we estimate by
	\begin{equation}
	\label{eq:i estimate for F - negative sign}
		\begin{aligned}
			&\ - \ew^T \big(\bfF(\bfx,\bfu)\bfw - \bfF(\xs,\us)\ws\big) \\
			= &\ - \ew^T \bfF(\bfx,\bfu)\ew - \ew^T \big(\bfF(\bfx,\bfu) - \bfF(\xs,\us)\big)\ws \\
			\leq &\ c_0  \|\ew\|_{\M}^2 + c \|\ew\|_{\M} \big( \|\ex\|_{\K} + \|\eu\|_{\K} \big) ,
		\end{aligned}
	\end{equation}
	where the first term in the middle line is estimated using \eqref{eq:alpha mass matrix - F} together with the $W^{1,\infty}$ boundedness of the numerical solutions $u_h$ (due to \eqref{eq:assumed bounds}), while the second term is again estimated using the above argument from \cite[Proposition~7.1, Part (A.v)]{MCF}, and using the norm equivalence \eqref{norm-equiv}. Here, $c_0$ is a positive constant.

	%We present this calculation here, starting by writing this term as
	%\begin{align*}
	%	\ew^T \big(\bfF(\bfx,\bfu) - \bfF(\xs,\us)\big)\ws 
	%	= &\ \int_{\Ga_h^1} F(u_h^1) w_h^{\ast,1} e_w^1 - \int_{\Ga_h^0} F(u_h^0) w_h^{\ast,0} e_w^0 \\
	%	= &\ \int_0^1 \frac{\d}{\d \theta} \int_{\Ga_h^\theta} F(u_h^\theta) w_h^{\ast,\theta} e_w^\theta \d \theta \\
	%	%%
	%%	= &\ \int_0^1 \int_{\Ga_h^\theta} (\pa_\theta^\bullet F(u_h^\theta)) w_h^{\ast,\theta} e_w^\theta \d \theta \\
	%%	&\ + \int_0^1 \int_{\Ga_h^\theta} F(u_h^\theta) (\pa_\theta^\bullet w_h^{\ast,\theta}) e_w^\theta \d \theta \\
	%%	&\ + \int_0^1 \int_{\Ga_h^\theta} F(u_h^\theta) w_h^{\ast,\theta e_w^\theta \Big( \nb_{\Ga_h^\theta} \cdot e_x^\theta \Big) \d \theta \\
	%	%%
	%%	= &\ \int_0^1 \int_{\Ga_h^\theta} F'(u_h^\theta)(\pa_\theta^\bullet u_h^\theta) w_h^{\ast,\theta e_w^\theta \d \theta \\
	%%	&\ + \int_0^1 \int_{\Ga_h^\theta} F(u_h^\theta) (\pa_\theta^\bullet w_h^{\ast,\theta}) e_w^\theta \d \theta \\
	%%	&\ + \int_0^1 \int_{\Ga_h^\theta} F(u_h^\theta) w_h^{\ast,\theta e_w^\theta \Big( \nb_{\Ga_h^\theta} \cdot e_x^\theta \Big) \d \theta \\
	%	%%
	%	= &\ \int_0^1 \int_{\Ga_h^\theta} F'(u_h^\theta)(e_u^\theta) w_h^{\ast,\theta} e_w^\theta \d \theta \\
	%	&\ + \int_0^1 \int_{\Ga_h^\theta} F(u_h^\theta) (\pa_\theta^\bullet w_h^{\ast,\theta}) e_w^\theta \d \theta \\
	%	&\ + \int_0^1 \int_{\Ga_h^\theta} F(u_h^\theta) w_h^{\ast,\theta} e_w^\theta \Big( \nb_{\Ga_h^\theta} \cdot e_x^\theta \Big) \d \theta
	%\end{align*}
	
	Altogether, the combination of the above bounds yields the first energy estimate
        \begin{align}
		\label{eq:energy estimate - i}
			\|\ew\|_{\A}^2 + \diff \|\eu\|_{\A}^2 
			%		\\
			\leq &\ c \|\eu\|_{\A}^2 + 2 c_0 \|\ew\|_{\M}^2 \\
		&\ + c \|\ew\|_{\K} \big( \|\ex\|_{\K} + \|\eu\|_{\K} +  \|\du\|_{\star,\xs} \big)  \nonumber \\	
			&\ + c \|\doteu\|_{\K} \big( \|\ex\|_{\K} + \|\eu\|_{\K} + \|\dw\|_{\star,\xs} + \| {\bar\bfe}_\bfw^0 \|_{\star,\bfx^0} \big) .
			\nonumber
%			\\
%			&\ - \diff \Big( \eu^T \big(\M - \Ms\big) \ws \Big)	
%		 - \diff \Big( \eu^T \big(\A - \As\big) \us \Big) \\
%			&\ - \diff \Big( \eu^T \big(\A - \As\big) \eu \Big) .
	\end{align}
%
%If we use \eqref{modified-Ah}, then we only have $\| \doteu(t) \|_{\bfM(\bfx(t))}$ on the right-hand side of \eqref{eq:energy estimate - i}.  
%

	\emph{Energy estimate (ii):} We test \eqref{eq:error eq - u} with $\eu$ and \eqref{eq:error eq - w} with $\ew$, then sum up to cancel the mixed terms $\eu^T \A \ew$, to obtain
	\begin{align*}
		\eu^T \M \doteu + \|\ew\|_{\M}^2   
		= &\ - \eu^T \big( \M-\Ms \big) \dotus 
		%	\\ &\ 
		+ \eu^T \big( \A-\As \big) \ws \\
		&\ + \eu^T \big(\bfF(\bfx,\bfu)\bfw - \bfF(\xs,\us)\ws\big) \\
		&\ + \eu^T \big(\bff(\bfx,\bfu) - \bff(\xs,\us)\big) 
		-\eu^T  \Ms\du , \\
		&\ - \ew^T \big( \M-\Ms \big) \ws 
		%	\\ &\ 
		- \ew^T \big( \A-\As \big) \us \\
		&\ + \ew^T \big(\bfg(\bfx,\bfu) - \bfg(\xs,\us)\big) 
		- \ew^T \Ms\dw  + \ew^T \bfvartheta . \\
	\end{align*}
	We estimate these terms by applying the same techniques as in (i): using \eqref{eq:i estimate lhs derivative} on the left-hand side, and on the right-hand side using \eqref{eq:i estimates matrix difference terms}, \eqref{eq:i estimates defect terms}, \eqref{eq:i estimates for nonlinarities}, and \eqref{eq:i estimate for F - negative sign}.
	%except for the term involving the state-dependent mass matrix $\bfF(\bfx,\bfu)$, for which we estimate by
	%\begin{equation}
	%\label{eq:ii estimate for F - 2}
	%	\begin{aligned}
	%		&\ \eu^T \big(\bfF(\bfx,\bfu)\bfw - \bfF(\xs,\us)\ws\big) \\
	%		= &\ \eu^T \bfF(\xs,\us)\ew + \eu^T \big(\bfF(\bfx,\bfu) - \bfF(\xs,\us)\big) \ew + \eu^T \big(\bfF(\bfx,\bfu) - \bfF(\xs,\us)\big)\ws \\
	%		\leq &\ c \|\eu\|_{\M}\|\ew\|_{\M} + c \|\eu\|_{\M} \big( \|\ex\|_{\K} + \|\eu\|_{\K} \big) ,
	%	\end{aligned}
	%\end{equation}
	%where for the estimate we used the $L^\infty$ boundedness of the error $e_w$ and the $W^{1,\infty}$ boundedness of the numerical solutions $u_h$ (due to \eqref{eq:assumed bounds}) and the estimates for the nonlinearity (we again refer to \cite[Proposition~7.1, Part (A.v)]{MCF}), and the norm equivalence \eqref{norm-equiv}.

We thus obtain the second energy estimate
	\begin{align}
	 	\label{eq:energy estimate - ii}
		\diff \|\eu\|_{\M}^2 + \|\ew\|_{\M}^2 \nonumber
		\leq &\ 
		c \|\eu\|_{\K} \big( \|\ex\|_{\K} + \|\eu\|_{\K} +  \|\du\|_{\star,\xs} \big) \\
		&\ + c \|\ew\|_{\K} \big( \|\ex\|_{\K} + \|\eu\|_{\K} + \|\dw\|_{\star,\xs} + \| {\bar\bfe}_\bfw^0 \|_{\star,\bfx^0} \big). 
		%		\\
		%		&\ - \frac12 \diff \Big( \eu^T \big(\M - \Ms\big) \eu \Big) .
	\end{align}	
	
	We now take the weighted linear combination of the two energy estimates \eqref{eq:energy estimate - i} and \eqref{eq:energy estimate - ii}, with weights $1$ and $4 c_0$, respectively, to obtain
	\begin{align*}
		&\ \diff \|\eu\|_{\A}^2 + 4 c_0 \diff \|\eu\|_{\M}^2 + \|\ew\|_{\A}^2 + 4 c_0 \|\ew\|_{\M}^2 \\ 
		\leq &\ c \|\eu\|_{\A}^2 + 2 c_0 \|\ew\|_{\M}^2 \\
		&\ + c \|\ew\|_{\K} \big( \|\ex\|_{\K} + \|\eu\|_{\K} +  
		 \|\du\|_{\star,\xs} +\|\dw\|_{\star,\xs} + \| {\bar\bfe}_\bfw^0 \|_{\star,\bfx^0}\big) \\	
		&\ + c \|\doteu\|_{\K} \big( \|\ex\|_{\K} + \|\eu\|_{\K} +  \|\dw\|_{\star,\xs} + \| {\bar\bfe}_\bfw^0 \|_{\star,\bfx^0} \big) \\
		&\ + c \|\eu\|_{\K} \big( \|\ex\|_{\K} + \|\eu\|_{\K} + \|\du\|_{\star,\xs} \big) .
	\end{align*}
	The term $\|\ew\|_{\M}^2$ is absorbed to the left-hand side.
	Then, using Young's inequality (often weighted with a small constant $\varrho > 0$ that can be chosen arbitrarily), by further absorptions to the left-hand side, and by collecting the terms, we obtain the estimate (with constants $c$ that depend on the choice of~$\varrho$)
	\begin{equation}
	\label{eq:energy estimate - i+ii}
		\begin{aligned}
		&	\diff \|\eu\|_{\A}^2 + 8 c_0 \diff \|\eu\|_{\M}^2 + c_0 \|\ew\|_{\K}^2 
			%%
%			\leq &\ c \|\eu\|_{\A}^2 \\
%			&\ + c \|\ew\|_{\K} \big( \|\ex\|_{\K} + \|\eu\|_{\K} + \|\du\|_{\Ms} \big) \\	
%			&\ + c \|\doteu\|_{\K} \big( \|\ex\|_{\K} + \|\eu\|_{\K} + \|\dw\|_{\Ms} \big) \\
%			&\ + c \|\eu\|_{\K} \big( \|\ex\|_{\K} + \|\eu\|_{\K} + \|\du\|_{\Ms} \big) \\
%			&\ + c \|\ew\|_{\K} \big( \|\ex\|_{\K} + \|\eu\|_{\K} + \|\dw\|_{\Ms} \big) \\
			%%
			\\
			&\leq  \varrho \|\doteu\|_{\K}^2 
			%	\\ &\ 
			+ c \big( \|\ex\|_{\K}^2 + \|\eu\|_{\K}^2 \big) 
			+ c \big( \|\du\|_{\star,\xs}^2 +\|\dw\|_{\star,\xs}^2 + \| {\bar\bfe}_\bfw^0 \|_{\star,\bfx^0}^2 \big) .
		\end{aligned}
	\end{equation}
	
	We now integrate the inequality \eqref{eq:energy estimate - i+ii} in time, divide by $\min\{1, c_0\}$, and use the norm equivalence \eqref{norm-equiv}, which altogether yields
	\begin{equation}
	\label{eq:energy estimate - i+ii - integrated}
		\begin{aligned}
			\|\eu\t\|_{\K(\xs\t)}^2 + &\ \int_0^t \|\ew\s\|_{\K(\xs\s)}^2 \d s \\
%			\|\eu\t\|_{\K(\xs\t)}^2 + \int_0^t \|\ew\s\|_{\K(\xs\s)}^2 \d s 
			\leq &\ \|\eu(0)\|_{\K(\xs(0))}^2 + \varrho \int_0^t \|\doteu\s\|_{\K(\xs\s)}^2 \d s \\ 
			&\ + c \int_0^t \big( \|\eu\s\|_{\K(\xs\s)}^2 + \|\ex\s\|_{\K(\xs\s)}^2 \big) \d s \\
			&\ + c \int_0^t \big( \|\du\s\|_{\star,\xs\s}^2 +\|\dw\s\|_{\star,\xs\s}^2 + \| {\bar\bfe}_\bfw^0 \|_{\star,\bfx^0}^2 \big) \d s .
		\end{aligned}
	\end{equation}
	%Any multiplicative constant $c$ is absorbed into the, not yet chosen, factor $\varrho > 0$.

%
%If we use \eqref{modified-Ah}, then we only have $\| \doteu(t) \|_{\bfM(\bfx(t))}$ on the right-hand side of \eqref{eq:energy estimate - i+ii - integrated}.  
%

	(A.2) To establish a bound for the critical term involving $\|\doteu\|_{\K}$ in \eqref{eq:energy estimate - i+ii - integrated}
	and to show a uniform-in-time bound for $\ew$ in the $\K$-norm, we perform a second pair of energy estimates.	
	We take the time derivative of equation \eqref{eq:error eq - w}:
	\begin{equation}
	\label{eq:error eq - w - diff}
		\begin{aligned}
			\M \dotew + \A \doteu 
			=&\ - {\dot \M}\ew - {\dot \A}\eu \\
			&\ - \big( \M-\Ms \big) \dotws - \diff \big( \M-\Ms \big) \ws \\ 
			&\ - \big( \A-\As \big) \dotus - \diff \big( \A-\As \big) \us \\ 
			&\ + \diff \big(\bfg(\bfx,\bfu) - \bfg(\xs,\us)\big)  \\
			&\ - \Ms\dotdw - {\dot \M}^*\dw .
		\end{aligned}
	\end{equation}
	
	\emph{Energy estimate (iii):} We test \eqref{eq:error eq - u} with $\doteu$ and \eqref{eq:error eq - w - diff} with $\ew$, then sum up to cancel the terms $\ew^T \A \doteu$, and obtain
	\begin{align*}
		\ew^T \M \dotew + \|\doteu\|_{\M}^2
		=&\ - \doteu^T \big( \M-\Ms \big) \dotus 
		%	\\ &\
		+ \doteu^T \big( \A-\As \big) \ws \\
		&\ + \doteu^T \big(\bfF(\bfx,\bfu)\bfw - \bfF(\xs,\us)\ws\big) \\
		&\ + \doteu^T \big(\bff(\bfx,\bfu) - \bff(\xs,\us)\big) - \doteu^T \Ms\du  \\
		&\ - \ew^T {\dot \M}\ew - \ew^T {\dot \A}\eu \\
		&\ - \ew^T \big( \M-\Ms \big) \dotws - \ew^T \diff \big( \M-\Ms \big) \ws \\ 
		&\ - \ew^T \big( \A-\As \big) \dotus - \ew^T \diff \big( \A-\As \big) \us \\ 
		&\ + \ew^T \diff \big(\bfg(\bfx,\bfu) - \bfg(\xs,\us)\big)  \\
		&\ - \ew^T \Ms\dotdw - \ew^T {\dot \M}^*\dw .
	\end{align*}
	
	We estimate these terms by applying the same techniques as in (i) and (ii): Using \eqref{eq:i estimate lhs derivative} on the left-hand side, and on the right-hand side using \eqref{eq:i estimates matrix difference terms}, \eqref{eq:i estimates defect terms}, and \eqref{eq:i estimates for nonlinarities} and \eqref{eq:i estimate for F - negative sign}. 
	
	The new terms involving time derivatives of the mass or stiffness matrix or their differences are bounded as follows.
	
	Using \eqref{matrix derivatives} (via the bounds \eqref{eq:assumed bounds}) we obtain
	\begin{equation}
	\label{eq:iii estimates time derivatives of matrices}
		\begin{aligned}
			- \ew^T {\dot \M}\ew \leq &\ c \|\ew\|_{\M}^2 \leq c \|\ew\|_{\K}^2 , \\
			- \ew^T {\dot \A}\eu \leq &\ c \|\ew\|_{\A} \|\eu\|_{\A} \leq c \|\ew\|_{\K} \|\eu\|_{\K} . 
		\end{aligned}
	\end{equation}
	Time derivatives of the differences of mass or stiffness matrices are bounded, exactly as in Part~(A.iv) in the proof of Proposition~7.1 in \cite{MCF}, by
	\begin{equation}
	\label{eq:iii estimates time derivatives of differences of matrices}
		\begin{aligned}
			- \ew^T \diff \big( \M-\Ms \big) \ws \leq &\ c \|\ew\|_{\K} \big( \|\ev\|_{\K} + \|\ex\|_{\K} \big) , \\ 
			- \ew^T \diff \big( \A-\As \big) \us \leq &\ c \|\ew\|_{\K} \big( \|\ev\|_{\K} + \|\ex\|_{\K} \big) . 
		\end{aligned}
	\end{equation}
	The defect term with a time derivative of the mass matrix is bounded using \eqref{matrix derivatives}, by
	\begin{equation}
	\label{eq:iii estimate time derivative of matrix with defect}
		\begin{aligned}
			\ew^T {\dot \M}^*\dw 
			\leq &\ c \|\ew\|_{\Ms} \|\dw\|_{\Ms} 
			\leq c \|\ew\|_{\M} \|\dw\|_{\Ms} .
		\end{aligned}
	\end{equation}
%The term containing the derivative of the nonlinearity $\bfg$ is bounded in Lemma~\ref{lem:f2} below by
The term containing the derivative of the nonlinearity $\bfg$ is estimated by techniques used in Part (A.v) of the proof of Proposition~7.1 in \cite{MCF}. A lengthy calculation yields%(see the Appendix for details):
\begin{equation}\label{est-dot-g}
\ew^T \diff \big(\bfg(\bfx,\bfu) - \bfg(\xs,\us)\big) \le c \|\ew\|_{\M} \Big(  \|\ex\|_{\K} + \|\ev\|_{\K}  + \|\eu\|_{\K}  + \|\doteu\|_{\K} \Big) .
\end{equation}
By combining the above estimates we obtain the third energy estimate
	\begin{equation}
	\label{eq:energy estimate - iii}
		\begin{aligned}
			&\ \Half \diff \|\ew\|_{\M}^2 + \Half \|\doteu\|_{\M}^2 \\
			\leq &\ c \|\doteu\|_{\K} \big( \|\ex\|_{\K} + \|\eu\|_{\K} + \|\ew\|_{\K}  \big) \\
			&\ + c \|\ew\|_{\K} \big(  \|\ex\|_{\K} + \|\ev\|_{\K} + \|\eu\|_{\K} + \|\doteu\|_{\K} + \|\ew\|_{\K} \big) \\
			&\ +c \|\doteu\|_{\K}\, \|\du\|_{\star,\xs}
			+ c \|\ew\|_{\K} \big( \|\dotdw\|_{\star,\xs} + \|\dw\|_{\Ms} \big) . 
		\end{aligned}
	\end{equation}
%
%(The following inequality was used in \eqref{eq:energy estimate - iii+iv - integrated}) \\
Integrating this inequality in time, we obtain with a constant $\varrho$ that is chosen suitably small, and with constants $c$ depending on the choice of $\varrho$,
	\begin{equation}
	\label{eq:energy estimate - iii - 2}
		\begin{aligned}
			&\ \Half \|\ew(t)\|_{\M(\bfx(t))}^2 + \Half \int_0^t \|\doteu(s) \|_{\M(\bfx(s))}^2 \d s \\
			\leq &\  
			\Half \|\ew(0)\|_{\M(\bfx(0))}^2 + \varrho \int_0^t \|\doteu(s)\|_{\K(\bfx(s))}^2 \d s \\
			& 
			+ c \int_0^t \big( \|\ex(s)\|_{\K(\bfx(s))}^2 + \|\ev(s)\|_{\K(\bfx(s))}^2   \big) \d s \\
			& 
			+ c \int_0^t \big(  \|\eu(s)\|_{\K(\bfx(s))}^2 + \|\ew(s)\|_{\K(\bfx(s))}^2 \big) \d s \\
			& 
			+ c \int_0^t \big(\|\du(s)\|_{\star,\xs}^2 + \|\dw(s)\|_{\Ms}^2 \big) \d s .
		\end{aligned}
	\end{equation}
%where $\epsilon$ can be an arbitrary positive constant. 
%

	\emph{Energy estimate (iv):} Finally, we test \eqref{eq:error eq - u} with $\dotew$ and \eqref{eq:error eq - w - diff} with $\doteu$, then subtract the former from the latter to obtain
	\begin{align*}
		\dotew^T \A \ew + \|\doteu\|_{\A}^2
		=&\ + \dotew^T \big( \M-\Ms \big) \dotus 
		- \dotew^T \big( \A-\As \big) \ws \\
		&\ - \dotew^T \big(\bfF(\bfx,\bfu)\bfw - \bfF(\xs,\us)\ws\big) \\
		&\ - \dotew^T \big(\bff(\bfx,\bfu) - \bff(\xs,\us)\big) 
		+ \dotew^T \Ms\du  \\
		&\ - \doteu^T {\dot \M}\ew - \doteu^T {\dot \A}\eu \\
		&\ - \doteu^T \big( \M-\Ms \big) \dotws - \doteu^T \diff \big( \M-\Ms \big) \ws \\ 
		&\ - \doteu^T \big( \A-\As \big) \dotus 
		%	\\	&\
		- \doteu^T \diff \big( \A-\As \big) \us \\ 
		&\ + \doteu^T \diff \big(\bfg(\bfx,\bfu) - \bfg(\xs,\us)\big)  \\
		&\ - \doteu^T \Ms\dotdw - \doteu^T {\dot \M}^*\dw .
	\end{align*}
	
	We bound these terms by applying the same techniques as in (iii) and (i)--(ii): We bound the terms using \eqref{eq:i estimate lhs derivative} on the left-hand side, and on the right-hand side all the terms that do not involve $\dotew$, using \eqref{eq:i estimates matrix difference terms}, \eqref{eq:i estimates defect terms}, and \eqref{eq:i estimates for nonlinarities} and \eqref{eq:i estimate for F - negative sign}. The terms involving time derivatives of $\M$ and $\A$ are bounded using \eqref{eq:iii estimates time derivatives of matrices}--
%	, \eqref{eq:iii estimates time derivatives of differences of matrices}, 
	\eqref{eq:iii estimate time derivative of matrix with defect}. 
	As in \eqref{est-dot-g} we have
	\begin{equation}
	\label{eq:iv estimate for the tricky term}
		\begin{aligned}
			\doteu^T \diff \big(\bfg(\bfx,\bfu) - \bfg(\xs,\us)\big) \leq &\ c \|\doteu\|_{\M} \Big(  \|\ex\|_{\K} + \|\ev\|_{\K}  + \|\eu\|_{\K}  + \|\doteu\|_{\K} \Big) .
		\end{aligned}
	\end{equation}
		The critical term $c \|\doteu\|_{\M} \|\doteu\|_{\K}$  is bounded, using Young's inequality and recalling that $\K=\M+\A$, as follows:
	\begin{equation}
	\label{eq:iv tricky term bound}
		\begin{aligned}
			c \|\doteu\|_{\M} \|\doteu\|_{\K} 
			= &\ c \|\doteu\|_{\M} \|\doteu\|_{\M+\A} \\
			\leq &\ c \|\doteu\|_{\M}^2 + c \|\doteu\|_{\M} \|\doteu\|_{\A} \\
			\leq &\ c_1 \|\doteu\|_{\M}^2 + \half \|\doteu\|_{\A}^2 , \qquad \text{with a constant } c_1 > 0 .
		\end{aligned}
	\end{equation}

	The terms involving $\dotew$ are rewritten and bounded by the following trick, cf.~\cite[equation (7.23)]{MCF}, and then bounded using \eqref{matrix difference bounds e_x}, see the arguments to \eqref{eq:i estimates matrix difference terms} and \eqref{eq:iii estimates time derivatives of differences of matrices} (cf.~Part~(A.iv) in \cite{MCF}):
	\begin{equation}
	\label{eq:iv estimates matrix differences dotew}
		\begin{aligned}
			\dotew^T \big( \M-\Ms \big) \dotus = &\ \diff \Big( \ew^T \big( \M-\Ms \big) \dotus \Big) \\
			&\ - \ew^T \diff \big( \M-\Ms \big) \dotus - \ew^T \big( \M-\Ms \big) \ddotus \\
			\leq &\ \diff \Big( \ew^T \big( \M-\Ms \big) \dotus \Big) + c \|\ew\|_{\K} \big( \|\ex\|_{\K} + \|\ev\|_{\K} \big) , \\
			- \dotew^T \big( \A-\As \big) \ws = &\ - \diff \Big( \ew^T \big( \A-\As \big) \ws \Big) \\
			&\ + \ew^T \diff \big( \A-\As \big) \ws + \ew^T \big( \A-\As \big) \dotws \\
			\leq &\ - \diff \Big( \ew^T \big( \A-\As \big) \ws \Big) + c \|\ew\|_{\K} \big( \|\ex\|_{\K} + \|\ev\|_{\K} \big) .
		\end{aligned}
	\end{equation}
Analogously, for the non-linear terms we obtain
%	\begin{align*}
%		&\ -\dotew^T \big(\bfF(\bfx,\bfu)\bfw - \bfF(\xs,\us)\ws\big) \\
%		= &\ -\diff\Big( \ew^T \big(\bfF(\bfx,\bfu)\bfw - \bfF(\xs,\us)\ws\big) \Big) \\
%		&\ + \ew^T \diff \big(\bfF(\bfx,\bfu)\bfw - \bfF(\xs,\us)\ws\big) \\
%		\leq &\ -\diff\Big( \ew^T \big(\bfF(\bfx,\bfu)\bfw - \bfF(\xs,\us)\ws\big) \Big) \\
%		&\ + c \|\ew\|_{\M} \Big(  \|\ex\|_{\K} + \|\ev\|_{\K}  + \|\eu\|_{\K}  + \|\doteu\|_{\K} + \|\ew\|_{\K} \Big) ,
%	\end{align*}
%
%In the last inequality the term $\|\ew\|_{\K}$ should be replaced by $\|\dotew\|_{\M}$. The term $\|\dotew\|_{\M}$ may be avoided if we change the estimation to the following one: 
\begin{align*}
		&\ -\dotew^T \big(\bfF(\bfx,\bfu)\bfw - \bfF(\xs,\us)\ws\big) \\
		= &\ -\dotew^T  \bfF(\bfx,\bfu)\ew -\dotew^T   \big(\bfF(\bfx,\bfu)- \bfF(\xs,\us)\big)\ws  \\
		= &\ - \diff\Big( \frac12 \ew^T \bfF(\bfx,\bfu)\ew + \ew^T  \big(\bfF(\bfx,\bfu)- \bfF(\xs,\us)\big)\ws \Big) \\
		&\ +\frac12 \ew^T \diff \bfF(\bfx,\bfu)\ew + \ew^T \diff \big(\bfF(\bfx,\bfu)- \bfF(\xs,\us)\big)\ws \big) \\
		\leq &\ -\diff\Big( \frac12 \ew^T \bfF(\bfx,\bfu)\ew + \ew^T  \big(\bfF(\bfx,\bfu)- \bfF(\xs,\us)\big)\ws \Big) \\
		&\ + c \|\ew\|_{\M} \Big(  \|\ex\|_{\K} + \|\ev\|_{\K}  + \|\eu\|_{\K}  + \|\doteu\|_{\K} + \|\ew\|_{\M} \Big) ,
\end{align*}
	where the inequality for the third term is shown  as in \eqref{est-dot-g}. 
	Similarly, we have for the other term
	\begin{align*}
		-\dotew^T \big(\bff(\bfx,\bfu) - \bff(\xs,\us)\big)
		%	= &\ -\diff\Big( \ew^T \big(\bff(\bfx,\bfu) - \bff(\xs,\us)\big) \Big) \\
		%	&\ + \ew^T \diff \big(\bff(\bfx,\bfu) - \bff(\xs,\us)\big) \\
		\leq &\ -\diff\Big( \ew^T \big(\bff(\bfx,\bfu) - \bff(\xs,\us)\big) \Big) \\
		&\ + c \|\ew\|_{\K}  \Big(  \|\ex\|_{\K} + \|\ev\|_{\K} + \|\eu\|_{\K}  + \|\doteu\|_{\K} \Big) .
	\end{align*}

	Finally, the defect term is bounded 	using \eqref{eq:iii estimate time derivative of matrix with defect},
	\begin{align*}
		\dotew^T \Ms\du = &\ \diff \Big( \ew^T \Ms \du  \Big) - \ew^T \Ms \dotdu - \ew^T {\dot \M}^* \du \\
		\leq &\ \diff \Big( \ew^T \Ms \du  \Big) + c \|\ew\|_{\K} \big( \|\dotdu\|_{\star,\xs} + \|\du\|_{\Ms} \big) .
	\end{align*}	
	Then the combination of the bounds above, and an absorption to the left-hand side of the last term in \eqref{eq:iv tricky term bound}, yields the fourth energy estimate
	\begin{equation}
	\label{eq:energy estimate - iv} 
		\begin{aligned}
			&\ \Half \diff \|\ew\|_{\A}^2 + \Half \|\doteu\|_{\A}^2 \\
			\leq &\ c_1 \|\doteu\|_{\M}^2 \\
			&\ + c \|\ew\|_{\K} \Big(   \|\ex\|_{\K} + \|\ev\|_{\K} +  \|\eu\|_{\K} + \|\doteu\|_{\K} + \|\ew\|_{\K} \Big) \\
			&\ + c \|\doteu\|_{\K} \Big(  \|\ex\|_{\K} + \|\ev\|_{\K} + \|\eu\|_{\K} + \|\ew\|_{\K} \Big) \\
			&\ + c \|\ew\|_{\K} \big( \|\du\|_{\Ms} +  \|\dotdu\|_{\star,\xs}  \big) 
			%		\\ &\
			+ c \|\doteu\|_{\K}\big(  \|\dw\|_{\Ms} + \|\dotdw\|_{\star,\xs} \big) \\
			&\ + \diff \Big( \ew^T \big( \M-\Ms \big) \dotus \Big) 
			- \diff \Big( \ew^T \big( \M-\Ms \big) \ws \Big) 
			+ \diff \Big( \ew^T \Ms \du  \Big) \\
			&\  - \diff \Big( \frac12 \ew^T \bfF(\bfx,\bfu)\ew + \ew^T  \big(\bfF(\bfx,\bfu)- \bfF(\xs,\us)\big)\ws \Big) \\
			&\ -\diff\Big( \ew^T \big(\bff(\bfx,\bfu) - \bff(\xs,\us)\big) \Big) .
		\end{aligned}
	\end{equation}
	
	Now we take the weighted sum of the the two energy estimates from (iii) and (iv), i.e.~we sum $4 c_1$ times \eqref{eq:energy estimate - iii} and \eqref{eq:energy estimate - iv}. Then we absorb the term $c_1 \|\eu\|_{\M}^2$ to the left-hand side.
%	, and multiply both sides with the minimum of $c_1$ and $1/2$ (cannot normalize the coefficient, because a nonnegative term inside $\diff$ cannot be dropped) . 
	Using Young's inequality (often with a small parameter), absorptions to the left-hand side, and collecting the terms, we obtain 
	$$%\begin{equation}
	%\label{eq:energy estimate - iii+iv}
		\begin{aligned}
			&\ 
			\diff \big(2 c_1 \|\ew\|_{\M}^2 + \Half  \|\ew\|_{\A}^2\big)  + c_1 \|\doteu\|_{\M}^2  + \Half \|\doteu\|_{\A}^2
		         \\
			\leq &\ c \|\eu\|_{\K}^2 + c \|\ex\|_{\K}^2 + c \|\ev\|_{\K}^2 + c \|\ew\|_{\K}^2 \\
			&\ + c \|\du\|_{\Ms}^2 + c \|\dotdu\|_{\star,\xs}^2 + c \|\dw\|_{\Ms}^2 + c\|\dotdw\|_{\star,\xs}^2 \\
			&\ + c \diff \Big( \ew^T \big( \M-\Ms \big) \dotus \Big) 
			- c \diff \Big( \ew^T \big( \M-\Ms \big) \dotws \Big) 
			+ c \diff \Big( \ew^T \Ms \du  \Big) \\
			&\ - c \diff \Big( \frac12 \ew^T \bfF(\bfx,\bfu)\ew + \ew^T  \big(\bfF(\bfx,\bfu)- \bfF(\xs,\us)\big)\ws \Big) \\
			&\ - c \diff \Big( \ew^T \big(\bff(\bfx,\bfu) - \bff(\xs,\us)\big) \Big) .
		\end{aligned}
	$$%\end{equation}
We integrate this inequality %\eqref{eq:energy estimate - iii+iv} 
in time, to obtain
	\begin{equation*}
		\begin{aligned}
			&\ \|\ew\t\|_{\K(\bfx\t)}^2 + \int_0^t \|\doteu\s\|_{\K(\bfx\s)}^2 \d s \\
			\leq &\ c \,\|\ew(0)\|_{\K(\bfx(0))}^2 \\
			&\ + c \! \int_0^t \!\! \big(  \|\ex\s\|_{\K(\bfx\s)}^2 + \|\ev\s\|_{\K(\bfx\s)}^2 + \|\eu\s\|_{\K(\bfx\s)}^2 + \|\ew\s\|_{\K(\bfx\s)}^2 \big) \d s \\
			&\ + c \! \int_0^t \!\! \big( \|\du\s\|_{\bfM(\bfx\s)}^2 + \|\dotdu\s\|_{\star,\bfx\s}^2 + \|\dw\s\|_{\bfM(\bfx\s)}^2 + c\|\dotdw\s\|_{\star,\bfx\s}^2 \big) \d s \\
			&\ + c \, \ew\t^T \Big( \M(\bfx\t) - \M(\xs\t) \Big) \dotus\t \\
			&\ - c \, \ew\t^T \Big( \M(\bfx\t) - \M(\xs\t) \Big) \dotws\t  + c \, \ew\t^T \M(\bfx\t) \du\t \\
%			&\ - c \, \ew\t^T \Big(\bfF(\bfx\t,\bfu\t)\bfw\t - \bfF(\bfx\t,\us\t)\ws\t\Big) \\
			&\ - c \, \ew\t^T \bfF(\bfx\t,\bfu\t)\ew\t - c \, \ew\t^T  \big(\bfF(\bfx\t,\bfu\t)- \bfF(\xs\t,\us\t)\big)\ws\t \\
			&\ + c \, \ew\t^T \Big(\bff(\bfx\t,\bfu\t) - \bff(\bfx\t,\us\t)\Big) .
		\end{aligned}
	\end{equation*}
We estimate those terms that are not integrated in time using \eqref{matrix difference bounds} (with \eqref{eq:assumed bounds}), and for the non-linear terms we use \eqref{eq:i estimate for F - negative sign} and \eqref{eq:i estimates for nonlinarities}. Then using Young's inequality (possibly with a sufficiently small constant $\varrho > 0$) and absorbing terms to the left-hand side, we obtain
	\begin{equation*}
%	\label{eq:energy estimate - iii+iv - integrated}
		\begin{aligned}
			&\ \|\ew\t\|_{\K(\bfx\t)}^2 + \int_0^t \|\doteu\s\|_{\K(\bfx\s)}^2 \d s \\
			%%
%			&\ \Half \|\ew\t\|_{\K(\bfx\t)}^2 + \Half \int_0^t \|\doteu\s\|_{\K(\bfx\s)}^2 \d s \\
%			\leq &\ \|\ew(0)\|_{\K(\bfx(0))}^2 \\
%			&\ + c \! \int_0^t \!\! \big( \|\ex\s\|_{\K(\bfx\s)}^2 + \|\eu\s\|_{\K(\bfx\s)}^2 + \|\ev\s\|_{\K(\bfx\s)}^2 + \|\ew\s\|_{\K(\bfx\s)}^2 \big) \d s \\
%			&\ + c \! \int_0^t \!\! \big( \|\du\s\|_{\bfM(\bfx\s)}^2 + \|\dotdu\s\|_{\bfM(\bfx\s)}^2 + \|\dw\s\|_{\bfM(\bfx\s)}^2 + c\|\dotdw\s\|_{\bfM(\bfx\s)}^2 \big) \d s \\
%			%
%			&\ + c \|\ew\t\|_{\K(\bfx\t)} \|\ex\t\|_{\K(\bfx\t)} \\
%			&\ + c \|\ew\t\|_{\K(\bfx\t)} \|\ex\t\|_{\K(\bfx\t)}  + c \|\ew\t\|_{\K(\bfx\t)} \|\du\t\|_{\star,\xs\t} \\
%			&\ + c \|\ew\t\|_{\K(\bfx\t)} \big(\|\ex\t\|_{\K(\bfx\t)} + \|\eu\t\|_{\K(\bfx\t)} \big) \\
%			&\ + c \|\ew\t\|_{\K(\bfx\t)} \big(\|\ex\t\|_{\K(\bfx\t)} + \|\eu\t\|_{\K(\bfx\t)} \big) \\
			%%
			\leq &\ c_2 \|\eu\t\|_{\K(\bfx\t)}^2 + c \|\ex\t\|_{\K(\bfx\t)}^2 
			 + c \|\ew\t\|_{\M(\bfx\t)}^2  \\
			&\ + c \! \int_0^t \!\! \big( \|\ex\s\|_{\K(\bfx\s)}^2 + \|\ev\s\|_{\K(\bfx\s)}^2 + \|\eu\s\|_{\K(\bfx\s)}^2 + \|\ew\s\|_{\K(\bfx\s)}^2 \big) \d s \\
			&\ + c \! \int_0^t \!\! \big( \|\du\s\|_{\bfM(\bfx\s)}^2 + \|\dotdu\s\|_{\star,\bfx\s}^2 + \|\dw\s\|_{\bfM(\bfx\s)}^2 + c\|\dotdw\s\|_{\star,\bfx\s}^2 \big) \d s \\
			&\  + c \|\du\t\|_{\star,\bfx\t}^2 +  c\|\ew(0)\|_{\K(\bfx(0))}^2,  \qquad \qquad \text{with a constant } c_2 > 0 .
		\end{aligned}
	\end{equation*}
The term $\|\ew\t\|_{\M(\bfx\t)}^2 $ can be estimated by using \eqref{eq:energy estimate - iii - 2}. Then we have 
	\begin{equation}
	\label{eq:energy estimate - iii+iv - integrated}
		\begin{aligned}
			&\ \|\ew\t\|_{\K(\bfx\t)}^2 + \frac12 \int_0^t \|\doteu\s\|_{\K(\bfx\s)}^2 \d s \\
			%%
%			&\ \Half \|\ew\t\|_{\K(\bfx\t)}^2 + \Half \int_0^t \|\doteu\s\|_{\K(\bfx\s)}^2 \d s \\
%			\leq &\ \|\ew(0)\|_{\K(\bfx(0))}^2 \\
%			&\ + c \! \int_0^t \!\! \big( \|\ex\s\|_{\K(\bfx\s)}^2 + \|\eu\s\|_{\K(\bfx\s)}^2 + \|\ev\s\|_{\K(\bfx\s)}^2 + \|\ew\s\|_{\K(\bfx\s)}^2 \big) \d s \\
%			&\ + c \! \int_0^t \!\! \big( \|\du\s\|_{\bfM(\bfx\s)}^2 + \|\dotdu\s\|_{\bfM(\bfx\s)}^2 + \|\dw\s\|_{\bfM(\bfx\s)}^2 + c\|\dotdw\s\|_{\bfM(\bfx\s)}^2 \big) \d s \\
%			%
%			&\ + c \|\ew\t\|_{\K(\bfx\t)} \|\ex\t\|_{\K(\bfx\t)} \\
%			&\ + c \|\ew\t\|_{\K(\bfx\t)} \|\ex\t\|_{\K(\bfx\t)}  + c \|\ew\t\|_{\K(\bfx\t)} \|\du\t\|_{\star,\xs\t} \\
%			&\ + c \|\ew\t\|_{\K(\bfx\t)} \big(\|\ex\t\|_{\K(\bfx\t)} + \|\eu\t\|_{\K(\bfx\t)} \big) \\
%			&\ + c \|\ew\t\|_{\K(\bfx\t)} \big(\|\ex\t\|_{\K(\bfx\t)} + \|\eu\t\|_{\K(\bfx\t)} \big) \\
			%%
			\leq &\ c_2 \|\eu\t\|_{\K(\bfx\t)}^2 + c \|\ex\t\|_{\K(\bfx\t)}^2  \\
			&\ + c \! \int_0^t \!\! \big( \|\ex\s\|_{\K(\bfx\s)}^2 + \|\ev\s\|_{\K(\bfx\s)}^2 + \|\eu\s\|_{\K(\bfx\s)}^2 + \|\ew\s\|_{\K(\bfx\s)}^2 \big) \d s \\
			&\ + c \! \int_0^t \!\! \big( \|\du\s\|_{\bfM(\bfx\s)}^2 + \|\dotdu\s\|_{\star,\bfx\s}^2 + \|\dw\s\|_{\bfM(\bfx\s)}^2 + c\|\dotdw\s\|_{\star,\bfx\s}^2 \big) \d s \\
			&\  + c \|\du\t\|_{\star,\bfx\t}^2 +  c\|\ew(0)\|_{\K(\bfx(0))}^2 .
		\end{aligned}
	\end{equation}
	
	(A.3) \emph{Combining the energy estimates (i)--(iv):} We multiply \eqref{eq:energy estimate - i+ii - integrated} with $2 c_2$ and sum with \eqref{eq:energy estimate - iii+iv - integrated}, absorb the terms (by choosing $\varrho > 0$ in \eqref{eq:energy estimate - i+ii - integrated} sufficiently small) with $\|\doteu\|_{\K}$ and $\|\eu\|_{\K}$ to the left-hand side, to obtain
	\begin{equation}
	\label{eq:energy estimates - final}
		\begin{aligned}
			&\ \|\eu\t\|_{\K(\bfx\t)}^2 + \|\ew\t\|_{\K(\bfx\t)}^2 
			%		\\ &\
			%+ \int_0^t \!\! \|\doteu\s\|_{\K(\bfx\s)}^2 \d s %+ \int_0^t \!\! \|\ew\s\|_{\K(\bfx\s)}^2 \d s  
			 \\
			%%%%
			\leq &\ c \|\ex\t\|_{\K(\bfx\t)}^2  \\
			&\ + c \! \int_0^t \!\! \big( \|\ex\s\|_{\K}^2 + \|\ev\s\|_{\K(\bfx\s)}^2 + \|\eu\s\|_{\K(\bfx\s)}^2  + \|\ew\s\|_{\K}^2 \big) \d s \\
			&\ + c \! \int_0^t \!\! \big( \|\du\s\|_{\bfM(\bfx\s)}^2 + \|\dotdu\s\|_{\star,\bfx\s}^2 + \|\dw\s\|_{\bfM(\bfx\s)}^2 + c\|\dotdw\s\|_{\star,\bfx\s}^2 \big) \d s \\
			&\ + c \|\du\t\|_{\star,\bfx\t}^2  + c \|\eu(0)\|_{\K(\bfx(0))}^2 + c \|\ew(0)\|_{\K(\bfx(0))}^2 +  c\| {\bar\bfe}_\bfw(0) \|_{\star,\bfx(0)}^2. 
		\end{aligned}
	\end{equation}

	(B) {\it Estimates for the velocity equation:} 
	We write $\ev$ as the nodal vector of the finite element function $e_v\in S_h[\xs]$. 
	To obtain an expression for the function $e_v$, we denote by $\widetilde I_h^\ast$ the finite element interpolation operator on $\Ga_h[\xs]$, and we set
	\begin{align*}
	V_h^\theta &= \sum_{j=1}^N V_j \, \phi_j[\bfx + \theta(\xs -\bfx)], \quad\
	\nu_h^\theta = \sum_{j=1}^N \nu_j \, \phi_j[\bfx + \theta(\xs -\bfx)].
	\end{align*}
	In view of \eqref{eq:error eq - v}, we then have
	\begin{align*}
	e_v &= \widetilde I_h^\ast (V_h^1 \nu_h^1) - \widetilde I_h^\ast (V_h^* \nu_h^*) - d_v 
	\\
	&= \widetilde I_h^\ast((V_h^1-V_h^*) \nu_h^1) + \widetilde I_h^\ast (V_h^* (\nu_h^1-\nu_h^*)) -d_v.
	\end{align*}
	Lemma~\ref{lemma:interpolation} gives us the bound
	\begin{align*}
	\| e_v \|_{H^1(\Ga_h[\xs])}  &\le C\, \| V_h^1-V_h^* \|_{H^1(\Ga_h[\xs])} \| \nu_h^1 \|_{W^{1,\infty}(\Ga_h[\xs])} 
	\\
	&\ +
	C\, \| \nu_h^1-\nu_h^* \|_{H^1(\Ga_h[\xs])} \| V_h^* \|_{W^{1,\infty}(\Ga_h[\xs])} + \| d_v \|_{H^1(\Ga_h[\xs])}.
	\end{align*}
	The $W^{1,\infty}$ norms of $V_h^*$ and $\nu_h^*$ are bounded independently of $h$ by assumption, and that of $\nu_h^1=\nu_h^*+e_\nu$ by \eqref{eq:assumed bounds}. Since the nodal vector of $V_h^1-V_h^*$ is the subvector of the first $N$ components of $\ew$ and the nodal vector of 
	$\nu_h^1-\nu_h^*$ is the subvector of the last $3N$ components of $\eu$, the above bound yields
	\begin{equation} \label{eq:velocity estimate}
	\| \ev \|_{\K(\xs)} \le c\, \bigl( \| \eu \|_{\K(\xs)} + \| \ew \|_{\K(\xs)} \bigr) + \| \dv \|_{\K(\xs)}.
	\end{equation}
 	
	(C) {\it Combination:} We use the differential equation $\dotex = \ev$ \eqref{eq:error eq - x} to show the bound
	\begin{align*}
		\|\ex\t\|_{\K(\xs\t)}^2 = &\ \int_0^t \Half \frac{\d}{\d s} \|\ex\s\|_{\K(\xs\s)}^2 \d s \\
		%	=&\ \int_0^t 2 \ex(s)^T \bfK(\bfx) \dotex(s) + \ex(s)^T \diff\big(\bfK(\bfx)\big) \ex(s) \ \d s \\
		\leq &\ c \int_0^t \|\ev\s\|_{\K(\xs\s)}^2 \d s + c \int_0^t \|\ex\s\|_{\K(\xs\s)}^2 \d s  ,
	\end{align*}
	which is substituted into %the estimates \eqref{eq:velocity estimate} and 
	\eqref{eq:energy estimates - final}. This combined estimate is  plugged into \eqref{eq:velocity estimate}, using
	 the equivalence \eqref{norm-equiv} between the norms $\|\cdot\|_{\K(\bfx\t)}$ and $\|\cdot\|_{\K(\xs\t)}$.
	 Then we sum the three inequalities  and obtain
	\begin{equation*} 
	%\label{eq:pre Gronwall}
		\begin{aligned}
			& \|\ex\t\|_{\K(\xs\t)}^2 + \|\ev\t\|_{\K(\xs\t)}^2 + \|\eu\t\|_{\K(\xs\t)}^2  + \|\ew\t\|_{\K(\xs\t)}^2 \\
			\leq &\ + c \! \int_0^t \!\! \big( \|\ex\s\|_{\K(\xs\s)}^2 + \|\ev\s\|_{\K(\xs\s)}^2 + \|\eu\s\|_{\K(\xs\s)}^2  + \|\ew\s\|_{\K(\xs\s)}^2 \big) \d s \\
			&\ + c \! \int_0^t \!\! \big( \|\du\s\|_{\bfM(\xs\s)}^2 + \|\dotdu\s\|_{\star,\xs\s}^2 + \|\dw\s\|_{\bfM(\xs\s)}^2 + c\|\dotdw\s\|_{\star,\xs\s}^2 \big) \d s \\
			&\ +  c \|\dv\t\|_{\K(\xs\t)}^2  + c \|\du\t\|_{\star,\xs\t}^2 
			\\ &\
			+ c \|\eu(0)\|_{\K(\xs(0))}^2 + c \|\ew(0)\|_{\K(\xs(0))}^2   + c\| {\bar\bfe}_\bfw(0) \|_{\star,\xs(0)}^2.
		\end{aligned}
	\end{equation*}
By Gronwall's inequality we obtain the stability bound \eqref{eq:stability bound} for $t \in [0,t^*]$.
	
	\medskip 
	Finally it remains to show that $t^* = T$ for $h$ sufficiently small. 
	To this end we use the assumed defect bounds to obtain the error estimates of order~$\kappa$:
	\begin{align*} 
		\normK{\ex(t)} + \normK{\ev(t)} + \normK{\eu(t)} + \normK{\ew(t)} \leq C h^{\kappa}.
	\end{align*}
	Then, by the inverse inequality, we have for $t\in[0,t^*]$ 
	\begin{equation*} 
		\begin{aligned}
			&\ \|e_x(\cdot,t)\|_{W^{1,\infty}(\Gamma_h[\xs\t])}
			+ \|e_v(\cdot,t)\|_{W^{1,\infty}(\Gamma_h[\xs\t])} \\
			&\ + \|e_u(\cdot,t)\|_{W^{1,\infty}(\Gamma_h[\xs\t])}
			+ \|e_w(\cdot,t)\|_{W^{1,\infty}(\Gamma_h[\xs\t])} \\
			\leq &\ ch^{-1}
			\big(\normK{\ex(t)} + \normK{\ev(t)} + \normK{\eu(t)} + \normK{\ew(t)} \big) \\ 
			\leq &\
			c C h^{\kappa-1} \leq \tfrac12 h^{(\kappa-1)/2} ,
		\end{aligned}
	\end{equation*}
	for sufficiently small $h$.
	Hence we can extend the bounds \eqref{eq:assumed bounds} beyond $t^*$, which contradicts the maximality of $t^*$ unless $t^*=T$. Therefore we have the stability bound \eqref{eq:stability bound} for $t\in[0,T]$. 
	\qed	
\end{proof}

\section{Consistency}
\label{section:Defect}

In this section we show that the defects and the errors in the initial values can be bounded by $O(h^k)$ in the appropriate norms that appear in the stability bounds of Proposition~\ref{proposition:stability}. Together with the argument of \cite[Section~9]{MCF}, this completes the proof of
Theorem~\ref{MainTHM}.
We first state the bounds in Subsection~\ref{subsec:consistency errors} and then turn to the essentials of their proof in the following subsections.

\subsection{Defect bounds and initial value error bounds}
\label{subsec:consistency errors}

As before, we let the vectors $\xs(t)\in\R^{3N}$ and $\vs(t)\in\R^{3N}$ collect the positions $X(p_i,t)$ of the moving finite element nodes on the exact surface $\Ga(t)=\Ga[X(\cdot,t)]$ and their velocities $v(X(p_i,t),t)$, respectively. They are the nodal vectors of finite element functions $X_h^*(\cdot,t)\in S_h[\Ga^0]^3$ and $v_h^*(\cdot,t) \in S_h(\Gamma_h[\xs(t)])^3$ on the interpolated surface $\Ga_h^*(t)=\Ga_h[\xs(t)]=\Ga[X_h^*(\cdot,t)]$, which moves with the velocity $v_h^*(x,t)=\tfrac{\d}{\d  t} X_h^*(p,t)$ for $x=X_h^*(p,t)$. We write
$$
\partial_{h}^\bullet f (x,t) = \partial_x f(x,t) v_h^*(x,t) + \partial_t f(x,t) , \qquad x\in \Ga_h^*(t),
$$
for the material derivative on $\Ga_h^*(t)$ corresponding to the velocity $v_h^*$. (This is not the same as the material derivative on the discrete surface $\Ga_h[\bfx(t)]$ with velocity $v_h$, although we use the same symbol $\partial_h^\bullet$.)

We further need reference finite element functions  
$$
u_h^*(\cdot,t)= \begin{pmatrix} H_h^\ast(t) \\ \nu_h^\ast(t) \end{pmatrix}\in S_h[\xs(t)]^{1+3},
\ \
w_h^*(\cdot,t)= \begin{pmatrix} V_h^\ast(t) \\ z_h^\ast(t) \end{pmatrix}\in S_h[\xs(t)]^{1+3}
$$ 
with their nodal vectors 
\begin{align*}
%&\xs(t)\in\R^{3N},\ \ \vs(t)\in\R^{3N},\\
%&
\us(t)=\begin{pmatrix} \bfH^\ast(t) \\ \mathbf{n}^\ast(t) \end{pmatrix} \in\R^{N+3N}, \quad
\ws(t)=\begin{pmatrix} \bfV^\ast(t) \\ \mathbf{z}^\ast(t) \end{pmatrix} \in\R^{N+3N}.
\end{align*}

\begin{proposition}
\label{prop:defect-bounds}
Under the assumptions of Theorem~\ref{MainTHM}, there exist, for $t\in[0,T]$, finite element functions $u_h^*(\cdot,t),w_h^*(\cdot,t)\in S_h[\xs(t)]^4$ such that the following holds true:

(a) The lift to the exact surface $\Ga(t)$ of these functions is $O(h^k)$ close to the exact solution in the $H^1$-norm:
\begin{equation}\label{us-ws-error bound}
   \begin{aligned}
      \| (u_h^*)^l(\cdot,t) - u(\cdot,t) \|_{H^1(\Ga(t))} &\le C\, h^k, 
      \\
      \| (w_h^*)^l(\cdot,t) - w(\cdot,t) \|_{H^1(\Ga(t))} &\le C\, h^k.
\end{aligned}
\end{equation}

(b) The defect functions $d_v(\cdot,t)\in S_h[\xs(t)]^3$ and $d_u(\cdot,t),d_w(\cdot,t)\in S_h[\xs(t)]^4$ with nodal vectors
$\bfd_\bfv(t)$ and $\bfd_\bfu(t)$, $\bfd_\bfw(t)$, respectively, which are defined by \eqref{eq:matrix-form-X-v-star}--\eqref{defect vectors}, are bounded by
	\begin{equation}
	\label{eq:defect bounds}
		\begin{aligned}
	\| d_v(\cdot,t) \|_{H^1(\Ga_h^*(t))} = \|\dv(t)\|_{\bfK(\xs(t))}  \leq &\ C h^k , \\
	\| d_u(\cdot,t) \|_{L^2(\Ga_h^*(t))}  =
			\|\du(t)\|_{\bfM(\xs(t))}   \leq &\ C h^k , \\
	\| d_w(\cdot,t) \|_{L^2(\Ga_h^*(t))}	=	\|\dw(t)\|_{\bfM(\xs(t))}  \leq &\ C h^k , \\
	\| \partial_{h}^\bullet d_u(\cdot,t) \|_{H^{-1}_h(\Ga_h^*(t))}	=	\|\dotdu(t)\|_{\star,\xs(t)}  \leq &\ C h^k, \\
	\| \partial_{h}^\bullet d_w(\cdot,t) \|_{H^{-1}_h(\Ga_h^*(t))}	=      \|\dotdw(t)\|_{\star,\xs(t)}  \leq &\ C h^k.
		%	\|\dx(t)\|_{\bfM(\xs(t))} \leq &\ c h^\kappa ,
		\end{aligned}
	\end{equation}
The constants $C$ are independent of $h$ and $t\in[0,T]$.
\end{proposition}

Because of the large number of terms, we will not give the full proof of this proposition. We will instead consider the construction and properties of $z_h^*$ as an exemplary and actually the most demanding case among the reference finite element functions. This function will be constructed by a modified Ritz map in the next subsection, whereas $H_h^*$, $\nu_h^*$ and $V_h^*$ can be constructed with the required approximation properties via the standard Ritz map. The construction of $z_h^*$ by a modified Ritz projection is required in order to obtain the $O(h^k)$ estimate for the defect $d_\nu$ (where $d_u=(d_H;d_\nu)\in S_h[\xs]^{1+3}$). We will therefore prove the approximation estimate for $z_h^*$ according to part (a) and the defect estimates for $d_\nu$ and its material derivative according to part (b). All the other bounds of Proposition~\ref{prop:defect-bounds} are obtained by similar or simpler arguments.

\begin{remark} The error bounds \eqref{us-ws-error bound} imply that $u_h^*(\cdot,t)$ and $w_h^*(\cdot,t)$ are bounded in the $W^{1,\infty}$ norm uniformly in $h$ and $t\in[0,T]$. This follows from the equivalent error bound 
$
  \| u_h^*(\cdot,t) - \widehat I_h^*u(\cdot,t) \|_{H^1(\Ga_h^*(t))} \le C\, h^k
$
and the inverse estimate 
$\| \varphi_h \|_{W^{1,\infty}(\Ga_h^*(t))} \le ch^{-1} \, \| \varphi_h \|_{H^1(\Ga_h^*(t))}$
%\quad\text{ 
for $\varphi_h \in S_h(\Ga_h^*(t)$,
%} 
where $c$ depends only on shape regularity and quasi-uniformity of the triangulation and on the regularity and size of $\Ga(t)$.
\end{remark}

\begin{proposition}\label{prop:initial errors}
Under the assumptions of Theorem~\ref{MainTHM}, the initial errors are bounded by
\begin{equation}\label{eq:initial errors}
	\begin{aligned}
	\| \widetilde I_h	u(\cdot,0) - u_h^*(\cdot,0) \|_{H^1(\Ga_h^*(0))} = \|\eu(0)\|_{\bfK(\bfx^0)} \leq&\ C h^k, \\
	\| \widetilde I_h	w(\cdot,0) - w_h^*(\cdot,0) \|_{H^1(\Ga_h^*(0))} = \|\ew(0)\|_{\bfK(\bfx^0)} \leq&\ C h^k, \\		 
	\| \widetilde I_h	w(\cdot,0) - \bar w_h(\cdot,0) \|_{H^{-1}_h(\Ga_h^*(0))}= \| {\bar\bfe}_\bfw(0) \|_{\star,\bfx^0} \leq&\ C h^k,
	\end{aligned}
\end{equation}
where $\widetilde I_h:C(\Ga^0) \to S_h[\xs(0)]$ is the finite element interpolation operator and 
$\bar w_h(\cdot,0)=(\bar V_h(\cdot,0);\bar z_h(\cdot,0))\in S_h[\xs(t)]^4$ is the finite element function defined by \eqref{Vh-prelim} and \eqref{zh-prelim} at $t=0$, that is, it has the nodal vector $\bar \bfw(0)$ defined in Section~\ref{subsection:semi-discretization-modified}.
\end{proposition}

Here we note that the first two bounds (for $\eu(0)$ and $\ew(0)$) follow directly from the error bounds \eqref{us-ws-error bound} at $t=0$, the known error bounds for interpolation and the equivalence of the $H^1$-norms of a function on the interpolated surface and its lift on the exact surface \cite{DziukElliott_ESFEM}. The third bound will be proved for the $z$-component of $\bar w_h=(\bar V_h,\bar z_h)$ in Section~\ref{subsec:err-bar-z}. The proof for the $V$-component is completely analogous and is therefore omitted.

After combining Propositions~\ref{prop:defect-bounds} and~\ref{prop:initial errors} with the stability estimates of Proposition~\ref{proposition:stability},  our main result Theorem~\ref{MainTHM} is proved by the same short argument as in \cite[Section~9]{MCF}.

\subsection{Construction of $z_h^*$  by a modified Ritz map}

The defect $d_\nu(\cdot,t)\in S_h(\Ga_h^*(t))^3$ is defined as the finite element function on $\Ga_h^*(t)$ with nodal vector $\bfd_{\mathbf{\nu}}(t)$, which comprises the last $3N$ components of the defect vector $\bfd_\bfu(t)=(\bfd_\bfH(t);\bfd_{\mathbf{\nu}}(t))\in \R^{4N}$ defined in \eqref{du}. Translated back into a function setting, $d_\nu$ is determined as the defect in \eqref{nuh-prelim} on inserting reference finite element functions $H_h^*$, $\nu_h^*$ and $V_h^*$, $z_h^*$ on $\Ga_h^*(t)=\Ga_h[\xs(t)]$ in place of the numerical solution $H_h$, $\nu_h$ and $V_h$, $z_h$ on $\Ga_h(t)=\Ga_h[\bfx(t)]$: 
with $A_h^*=\half(\nb_{\Ga_h^*} \nu_h^*+ (\nb_{\Ga_h^*} \nu_h^*)^T)$ %, $K_h^*=\half((H_h^*)^2 - |A_h^*|^2)$ 
and $Q_h^* = -\tfrac12 (H_h^*)^3 + |A_h^*|^2 H_h^*$ (and omitting the omnipresent argument $t$),
\begin{align}
	\label{d_nu}
	\int_{\Ga_h[\xs]} \!\!\! d_\nu \vphi_h 
	= &\ \int_{\Ga_h[\xs]} \!\!\! \mat_{h}\nu_h^* \cdot \vphi_h 
		\\ 
	&\  - \int_{\Ga_h[\xs]} \!\!\! \nb_{\Ga_h[\xs]}z_h^* \cdot\nb_{\Ga_h[\xs]}\vphi_h  \nonumber
	\\    
	&\  	- 2 \int_{\Ga_h[\xs]} \!\!\! (A_h^*\nb_{\Ga_h[\xs]} H_h^*) \cdot (\nb_{\Ga_h[\xs]} \vphi_h\, \n_h^* ) 		
	- \int_{\Ga_h[\xs]} \!\!\!  Q_h^* (\nb_{\Ga_h[\xs]} \cdot \vphi_h) 
	\nonumber \\	
	       		&\ -  \int_{\Ga_h[\xs]} \big( |\nb_{\Ga_h[\xs]} H_h^*|^2 \nu_h^* + (A_h^*)^2 \nb_{\Ga_h[\xs]} H_h^* \big) \cdot \vphi_h 
		\nonumber \\
	&\  + \int_{\Ga_h[\xs]} \!\!\! \Qh^* H_h^* \, \nu_h^*\cdot \vphi_h
	 - \int_{\Ga_h[\xs]} \!\! (H_h^* A_h^* - (A_h^*)^2) \ z_h^* \cdot \vphi_h  
	\nonumber
%	\\
%			&	 \int_{\Ga_h[\bfx]} \!\!\!\! \mat_h \n_h \cdot \phin_h 
%		- \int_{\Ga_h[\bfx]} \!\!\!\!  \nb_{\Ga_h[\bfx]} z_h \cdot \nb_{\Ga_h[\bfx]} \phin_h  
%		%\nonumber 
%%		\\
%%		&\hspace{72pt} 
%		 = -  \int_{\Ga_h[\bfx]} \!\!\!  \Qh \,\nb_{\Ga_h[\bfx]} \cdot \phin_h \nonumber 
%		\\
%		&\hspace{36pt}  
%		+\int_{\Ga_h[\bfx]} \!\!\! \Qh H_h \, \nu_h\cdot \phin_h
%%		 \nonumber \\
%%		&\hspace{72pt} 
%		+ \int_{\Ga_h[\bfx]} \!  \bigl( H_h A_h - A_h^2 \bigr) \nb_{\Ga_h[\bfx]} H_h \cdot \phin_h , \\[8pt] 
\end{align}
for all $\vphi_h\in S_h[\xs]^3$. As we want to bound the $L^2$ norm of $d_\nu$, we can expect difficulties from the terms in the second and third line of \eqref{d_nu}, which contain the surface gradient and divergence of the test function $\vphi_h$. We subtract equation \eqref{eq:weak form - nu} for the exact solution $\nu$ from \eqref{d_nu}:
\begin{align}
\label{d_nu-reformulated}
	&\int_{\Ga_h[\xs]} \!\!\! d_\nu \vphi_h 
	=  \int_{\Ga_h[\xs]} \!\!\! \mat_{h}\nu_h^* \cdot \vphi_h  - \int_{\Ga[X]} \!\!\! \mat \n \cdot \vphi_h^l 
		\\ 
	&\  + \left( - \int_{\Ga_h[\xs]} \!\!\! \nb_{\Ga_h[\xs]}z_h^* \cdot\nb_{\Ga_h[\xs]}\vphi_h 
	- 2 \int_{\Ga_h[\xs]} \!\!\! (A_h^*\nb_{\Ga_h[\xs]} H_h^*) \cdot (\nb_{\Ga_h[\xs]} \vphi_h\, \n_h^* )  \right.
	\nonumber \\	
	&\qquad\  \left.
	- \int_{\Ga_h[\xs]} \!\!\!  Q_h^* (\nb_{\Ga_h[\xs]} \cdot \vphi_h) \right.
	\nonumber \\	
	&\qquad\  \left. + \int_{\Ga[X]} \!\!\! \nb_{\Ga[X]} z \cdot \nb_{\Ga[X]} \vphi_h^l  
	+ 2 \int_{\Ga[X]} \!\!\! (A\nb_{\Ga[X]} H) \cdot (\nb_{\Ga[X]} \vphi_h^l \n ) \right.
	\nonumber \\	
	&\qquad\  \left.	 + \int_{\Ga[X]} \!\!\! \Q \,\nb_{\Ga[X]} \cdot \vphi_h^l  \right)
        	\nonumber \\	
		       		&\ -  \int_{\Ga_h[\xs]} \big( |\nb_{\Ga_h[\xs]} H_h^*|^2 \nu_h^* + (A_h^*)^2 \nb_{\Ga_h[\xs]} H_h^* \big) \cdot \vphi_h 
        \nonumber \\
			&\qquad\quad +  \int_{\Ga[X]} \big( |\nb_{\Ga[X]} H|^2 \nu + A^2 \nb_{\Ga[X]} H \big) \cdot \vphi_h^l
       \nonumber \\
	& \  + \int_{\Ga_h[\xs]} \!\!\! \Qh^* H_h^* \, \nu_h^*\cdot \vphi_h - \int_{\Ga[X]} \!\!\! \Q H \,\nu \cdot \vphi_h^l
        	\nonumber \\		
	&\  - \int_{\Ga_h[\xs]} \!\!  (H_h^* A_h^* - (A_h^*)^2) \ z_h^* \cdot \vphi_h   
	+ \int_{\Ga[X]} \!  (H A - A^2) \ z \cdot \vphi_h^l
	\nonumber
\end{align}
for all $\vphi_h\in S_h[\xs]^3$. The critical terms in the big bracket from the second to the fifth line are replaced by harmless terms, which do not contain a gradient or divergence of the test function, when we define $z_h^*\in S_h[\xs]^3$ by a 
\emph{modified Ritz map}: $z_h^*$ is determined from $z$ (and $Q_h^*$ and $Q$) by
\begin{align}
\nonumber
&\int_{\Ga_h[\xs]} \!\!\! \bigl( \nb_{\Ga_h[\xs]}z_h^* \cdot\nb_{\Ga_h[\xs]}\vphi_h + z_h^*\vphi_h \bigr)
%\\
%\nonumber
%&   
=  \int_{\Ga[X]} \!\!\! \bigl( \nb_{\Ga[X]} z \cdot \nb_{\Ga[X]} \vphi_h^l + z \vphi_h^l \bigr)
\\
\nonumber
& + \left(  2 \int_{\Ga_h[\xs]} \!\!\! (A_h^*\nb_{\Ga_h[\xs]} H_h^*) \cdot (\nb_{\Ga_h[\xs]} \vphi_h\, \n_h^* ) -
2 \int_{\Ga[X]} \!\!\! (A\nb_{\Ga[X]} H) \cdot (\nb_{\Ga[X]} \vphi_h^l \n ) \right)
\\
\label{mod-ritz}
&	+ \left( \int_{\Ga_h[\xs]} \!\!\!  Q_h^* (\nb_{\Ga_h[\xs]} \cdot \vphi_h) - \int_{\Ga[X]} \!\!\! \Q \,\nb_{\Ga[X]} \cdot \vphi_h^l  \right)
\end{align}
for all $\vphi_h\in S_h[\xs]^3$. A similar modified Ritz map, also motivated by a perturbation term containing the gradient of the test function, was previously used and analysed in \cite{LubichMansour_wave}. For later use we can now restate the simplified version of \eqref{d_nu-reformulated},
\begin{align}
\label{d_nu after mod ritz}
	&\int_{\Ga_h[\xs]} \!\!\! d_\nu \vphi_h 
	=  \int_{\Ga_h[\xs]} \!\!\! \mat_{h}\nu_h^* \cdot \vphi_h  - \int_{\Ga[X]} \!\!\! \mat \n \cdot \vphi_h^l 
		\\ 
	&\  + \int_{\Gamma_h[\xs]} \! z_h^* \cdot  \varphi_h 
- \int_{\Gamma[X]} z \cdot  \varphi_h^l 
        	\nonumber \\	
			 &\ -  \int_{\Ga_h[\xs]} \big( |\nb_{\Ga_h[\xs]} H_h^*|^2 \nu_h^* + (A_h^*)^2 \nb_{\Ga_h[\xs]} H_h^* \big) \cdot \vphi_h 
        \nonumber \\
			&\qquad\quad +  \int_{\Ga[X]} \big( |\nb_{\Ga[X]} H|^2 \nu + A^2 \nb_{\Ga[X]} H \big) \cdot \vphi_h^l
       \nonumber \\
	& \  + \int_{\Ga_h[\xs]} \!\!\! \Qh^* H_h^* \, \nu_h^*\cdot \vphi_h - \int_{\Ga[X]} \!\!\! \Q H \,\nu \cdot \vphi_h^l
        	\nonumber \\		
	&\  - \int_{\Ga_h[\xs]} \!\! (H_h^* A_h^* - (A_h^*)^2) \ z_h^* \cdot \vphi_h   
	+ \int_{\Ga[X]} \!  (H A - A^2) \ z \cdot \vphi_h^l
	\nonumber
\end{align}
for all $\vphi_h\in S_h[\xs]^3$, where, thanks to the modified Ritz map \eqref{mod-ritz}, no gradients of the test function appear.

We note that the equation for $\nu$ is the only equation in \eqref{eq:weak form} for which the right-hand side contains the gradient or the divergence of the test function. The other reference finite element functions $H_h^*,\nu_h^*,V_h^*$ are therefore defined by the standard Ritz map (which only contains the first line in \eqref{mod-ritz}). We further note that $Q_h^*$ is thus defined independently of $z_h^*$, and so all the reference finite element functions $H_h^*,\nu_h^*,V_h^*,z_h^*$ are well defined.

\subsection{Error bounds of the reference finite element functions}
For the functions defined by the standard Ritz map, 
the optimal-order $H^1$ error bounds of \cite[Theorems~6.3--6.4]{Kovacs2017} yield, under the regularity assumptions of Theorem~\ref{MainTHM},
\begin{equation}
\label{ritz-error}
\begin{aligned}
\| (\nu_h^*)^l(\cdot, t) -\nu(\cdot,t) \|_{H^1(\Ga(t))} &\le C\, h^k
\\
\| (\partial_{h}^\bullet \nu_h^*)^l(\cdot, t) - \partial^\bullet \nu(\cdot,t) \|_{H^1(\Ga(t))} &\le C\, h^k
\\
\|(\partial_{h}^\bullet\partial_{h}^\bullet \nu_h^*)^l(\cdot, t) - \partial^\bullet \partial^\bullet \nu(\cdot,t) \|_{H^1(\Ga(t))} &\le C\, h^k
\end{aligned}
\end{equation}
and analogously for $H_h^*$ and $V_h^*$. These error bounds also yield, by an inverse estimate, that these reference finite element functions functions have $W^{1,\infty}$ norms bounded independently of $h$. These facts, together with the equivalence of norms on $\Ga(t)$ and $\Ga_h^*(t)$, give us
 an $L^2$ error bound for $Q_h^*$:
$$
\| (Q_h^*)^l(\cdot, t) -Q(\cdot,t) \|_{L^2(\Ga(t))} \le C\, h^k.
$$
We then rewrite the last term in \eqref{mod-ritz} as
\begin{align*}
&\int_{\Ga_h[\xs]} \!\!\!  Q_h^* (\nb_{\Ga_h[\xs]} \cdot \vphi_h) -  \int_{\Ga[X]} \!\!\! (Q_h^*)^l \,\nb_{\Ga[X]} \cdot \vphi_h^l
\\
& \ + \int_{\Ga[X]} \!\!\! (Q_h^*)^l \,\nb_{\Ga[X]} \cdot \vphi_h^l - \int_{\Ga[X]} \!\!\! \Q \,\nb_{\Ga[X]} \cdot \vphi_h^l,
\end{align*}
where the difference in the second line is bounded by $C\, h^k \| \nb_{\Ga[X]} \cdot \vphi_h^l \|_{L^2(\Ga[X])}$ in view of the $L^2$ error bound of $Q_h^*$. The first term has a bound of the same type by the higher-order version of \cite[Lemma~5.5]{DziukElliott_L2} that follows with the geometric error bound of \cite[Lemma~5.2]{Kovacs2017}.

The difference of the integrals in the second line of \eqref{mod-ritz} is bounded by $C\, h^k \| \nb_{\Ga[X]}  \vphi_h^l \|_{L^2(\Ga[X])}$ by the same arguments.

With these error bounds for the extra terms in the modified Ritz map, the proof of \cite[Theorems~8.2 and~8.3]{LubichMansour_wave} together with the higher-order geometric perturbation error bounds of \cite[Lemma~5.6]{Kovacs2017} yields the error bounds
\begin{equation}
\label{mod-ritz-error}
\begin{aligned}
\| (z_h^*)^l(\cdot, t) -z(\cdot,t) \|_{H^1(\Ga(t))} &\le C\, h^k
\\
\| \partial_{h}^\bullet(z_h^*)^l(\cdot, t) - \partial^\bullet z(\cdot,t) \|_{H^1(\Ga(t))} &\le C\, h^k.
\end{aligned}
\end{equation}

\subsection{Bounds for the defect $d_\nu$ and its material derivative}

The techniques of the proof of the defect bounds have already been used for the defect bounds in \cite{KLLP2017} and \cite{MCF}, but we now have the additional difficulty to prove estimates for the time-differentiated defects.

The terms in \eqref{d_nu after mod ritz} do not contain the gradient of the test function, and this allows us obtain estimates of $d_\nu$ in the $L^2$ norm. All the pairs of terms in \eqref{d_nu after mod ritz} can be bounded by the same arguments as for the corresponding terms in \cite[Section~8]{KLLP2017}, using the Ritz map error bounds of the previous subsection instead of the bound for the interpolation error where required. These arguments are based on geometric estimates that were previously proved in \cite{Demlow2009,Dziuk88,DziukElliott_ESFEM,DziukElliott_L2,Kovacs2017}. We refer to \cite{KLLP2017} for the details.

To bound the material derivative $\partial_{h}^\bullet d_\nu$ in the $H^{-1}_h$ norm, we differentiate \eqref{d_nu-reformulated} in time for time-dependent test functions $\vphi_h(\cdot,t)\in S_h[\xs(t)]^3$ with time-independent nodal vector and use the Leibniz formula to obtain
$$
\int_{\Ga_h^*} \partial_{h}^\bullet d_\nu \, \vphi_h = \frac{\d}{\d t} \int_{\Ga_h^*}  d_\nu \,\vphi_h
- \int_{\Ga_h^*}  d_\nu \, (\nb_{\Ga_h[\xs]} \cdot v_h^*) \,\vphi_h,
$$
where we used that $\partial_{h}^\bullet \vphi_h =0$. The last term is readily estimated using the $L^2$ bound of $d_\nu$. For the first term on the right-hand side, we insert \eqref{d_nu after mod ritz} and differentiate each pair of terms using the Leibniz rule. With the Ritz map error bounds \eqref{ritz-error} for the material derivatives of $\nu_h^*$ and $H_h^*$ and again with the arguments of \cite{KLLP2017} we then obtain the stated error bound for $\partial_{h}^\bullet d_\nu$.

\subsection{Error bound of $\bar z_h(\cdot,0)$}
\label{subsec:err-bar-z}

By the construction of $\bar z_h(\cdot,0)$ from \eqref{zh-prelim}, we have (omitting in the following the argument $t=0$)
$$
\int_{\Gamma_h^0} \! \bar z_h \cdot \vphi_h =- \int_{\Gamma_h^0} \! \nb_{\Gamma_h^0} \n_h \cdot \nb_{\Gamma_h^0} \vphi_h + \int_{\Gamma_h^0} \! | A_h |^2 \n_h \cdot \vphi_h
$$
for all $\vphi_h\in S_h(\Gamma_h^0)^3$. Here $\nu_h$ is the finite element interpolation of $\nu$ on $\Gamma_h^0$ and
$A_h = \half (\nb_{\Ga_h^0} \n_h + (\nb_{\Ga_h^0} \n_h)^T)$. This is to be compared with the interpolation of $z(0)$, which satisfies
\eqref{eq:weak form - z}, that is,
$$
\int_{\Ga^0} \!\!\! z \cdot \vphi =- \int_{\Ga^0} \!\!\! \nb_{\Ga[X]} \n \cdot \nb_{\Ga^0} \vphi + \int_{\Ga^0} \!\!\! |A|^2 \n \cdot \vphi
$$
for all $\vphi_h\in H^1(\Gamma^0)$. So we obtain
\begin{align*}
&\int_{\Gamma_h^0} \! (\bar z_h - \widetilde I_h z) \cdot \vphi_h =
- \int_{\Gamma_h^0} \! \nb_{\Gamma_h^0} \n_h \cdot \nb_{\Gamma_h^0} \vphi_h
+\int_{\Ga^0} \!\!\! \nb_{\Ga[X]} \n \cdot \nb_{\Ga^0} \vphi_h^l
\\
&\ + \int_{\Gamma_h^0} \! | A_h |^2 \n_h \cdot \vphi_h - \int_{\Ga^0} \!\!\! |A|^2 \n \cdot \vphi_h^l
- \int_{\Gamma_h^0} \!  \widetilde I_h z \cdot \vphi_h + \int_{\Ga^0} \!\!\! z \cdot \vphi_h^l
\end{align*}
for all $\vphi_h\in S_h(\Gamma_h^0)^3$. The pairs of terms on the right-hand side can be estimated in the same way as similar pairs in the defects, and so we obtain the last bound of \eqref{eq:initial errors} for $\bar z_h(\cdot,0)$ in the $H^{-1}_h(\Ga_h^0)$ norm.

\section{Numerical experiments}
\label{section:numerics}

We performed the following numerical experiments for Willmore flow: 
\begin{itemize}
	\item Convergence tests using stationary solutions of Willmore flow, i.e.~a sphere and a Clifford torus. 
	\item We report on the Willmore energy of a few surfaces from the literature, e.g.~\cite{Dziuk_Willmore} and \cite{BGN2008Willmore}, as they evolve towards the stationary solution, minimising the energy. 
%	\item \bbk ??? \ebk We have implemented Dziuk's algorithm form \cite[Problem~3]{Dziuk_Willmore}, and performed a few comparison tests.
\end{itemize}
All our numerical experiments use quadratic evolving surface finite elements. The numerical computations were carried out in Matlab. The initial meshes for all surfaces were generated using DistMesh \cite{distmesh}, without taking advantage of the symmetries of the surfaces.
For time discretization we use the linearly implicit backward difference formula (BDF) of order 2 applied to the 
system \eqref{eq:matrix-form-X-v}--\eqref{eq:matrix--vector form}
in matrix-vector formulation, in the same way as described in \cite{MCF}.

\subsection{Convergence tests}

We computed approximations to Willmore flow for two surfaces that are \emph{stationary} along Willmore flow (though not along the discretized flow): a sphere of radius $R=1$ ($W(\Gamma) = 8\pi \approx 25.1327$), and a Clifford torus ($W(\Gamma) = 4\pi^2 \approx 39.4784$), i.e.~a torus with the quotient of radii $R/r = \sqrt{2}$, in our case with major radius $R = 1$ and minor radius $r = 1/\sqrt{2}$. The considered time interval is  $[0,T]=[0,1]$ in both cases.

For our convergence tests we used a sequence of time step sizes $\tau_k=\tau_{k-1}/2$ with $\tau_0 = 0.2$, and a sequence of meshes with mesh widths $h_k \approx 2^{-1/2} h_{k-1}$. For the Clifford torus, in order to reasonably resolve the surface, we used graded meshes that are more refined around the hole.

We  started the time integration from the nodal interpolations of the exact initial values $\nu(\cdot,0)$, $H(\cdot,0)$, and $V(\cdot,0) = 0$, $z(\cdot,0) = \nb_{\Ga^0} H(\cdot,0)$, which were computed analytically.

In Figure~\ref{fig:conv_sphere_space_nonprojected} we report on the errors between the numerical and exact solutions for the Willmore flow of a sphere until the final time $T=1$ (the first three plots from left to right: the surface error and the errors of the dynamic variables $\nu$ and $H$), and on the Willmore energy of the discrete surfaces at the final time $T=1$ (rightmost plot). 
The logarithmic error plots show the $L^\infty(H^1)$ norm errors, and the Willmore energy (the plot is semi-logarithmic) against the mesh width $h$. The lines marked with different symbols correspond to different time step sizes.

\begin{figure}[htbp]
	\includegraphics[width=\textwidth]{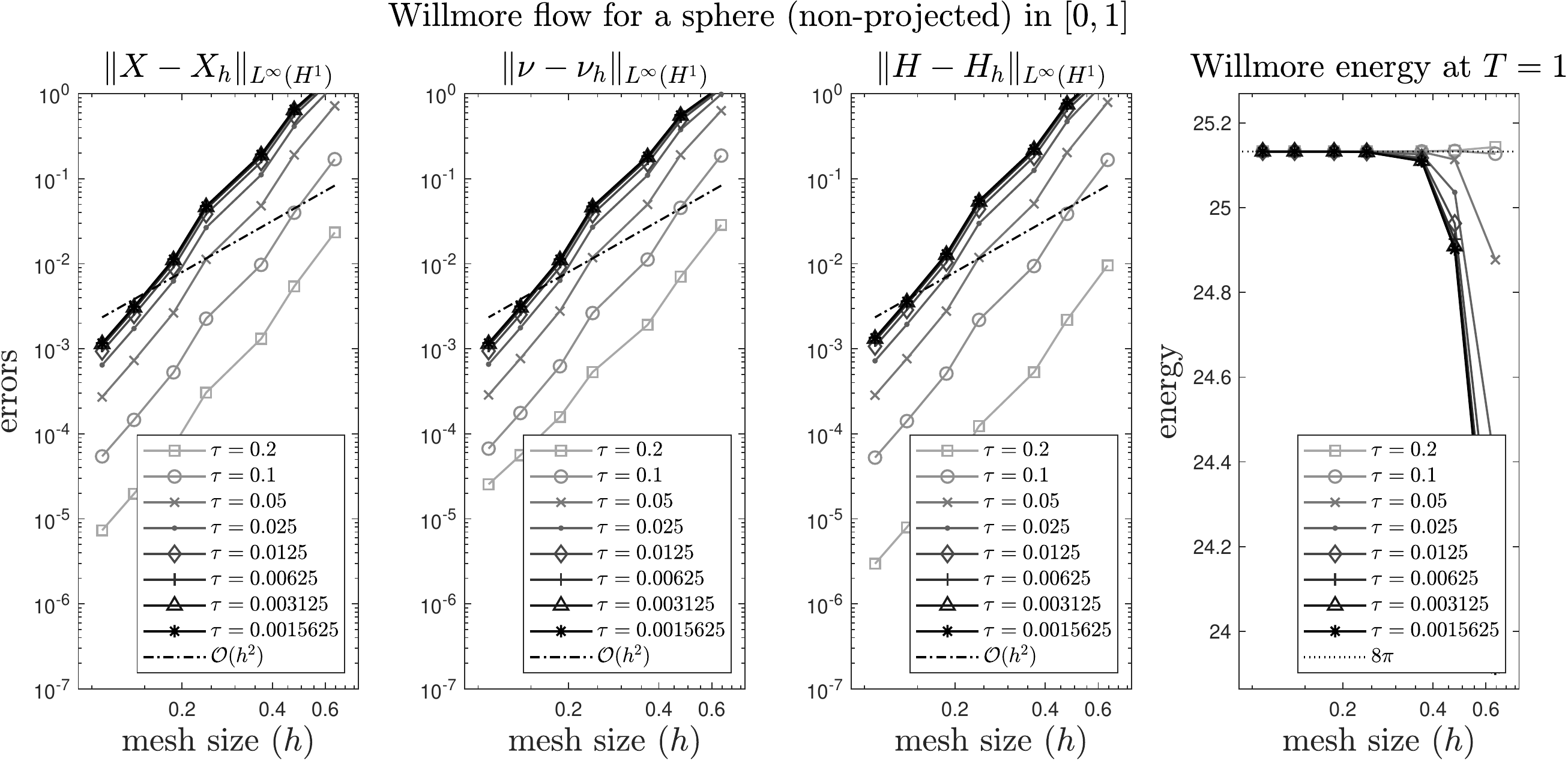}
	\caption{Spatial convergence of the BDF2 / quadratic ESFEM discretization (without projecting $nu_h$ and $z_h$) for the Willmore flow of a sphere for $T = 1$.}
	\label{fig:conv_sphere_space_nonprojected}
\end{figure}

In order to preserve important properties of the functions in the approximation, it is recommended to \emph{project} the computed surface normal $\nu_h\approx \nu$ onto the unit sphere and the function $z_h\approx \nabla_\Gamma H$ onto the tangent space. This is done after every time step. The advantages of a normalising projection for $\nu_h$ were already explored in numerical experiments for mean curvature flow in \cite{MCF}. 
%Note that although the length of $\nu_h$ deviates from $1$ in the computations, it still satisfies the estimate $\|1 - |\nu_h^L|\|_{H^1(\Ga[X])} \leq \|\nu - \nu_h^L\|_{H^1(\Ga[X])^3} = O(h^k)$, in view of Theorem~\ref{MainTHM} and the reversed triangle inequality.

Figure~\ref{fig:normalization} reports on the effect of the projections using a sphere of radius~$1$: the left plots show the minimum ($\circ$) and maximum ($*$) lengths of the normal vectors, the middle plots the roundness of the surface (minimum ($\circ$) and maximum ($*$) distance from origin), and the right plots show the Willmore energy along the flows.

\begin{figure}[htbp]
	\includegraphics[width=\textwidth]{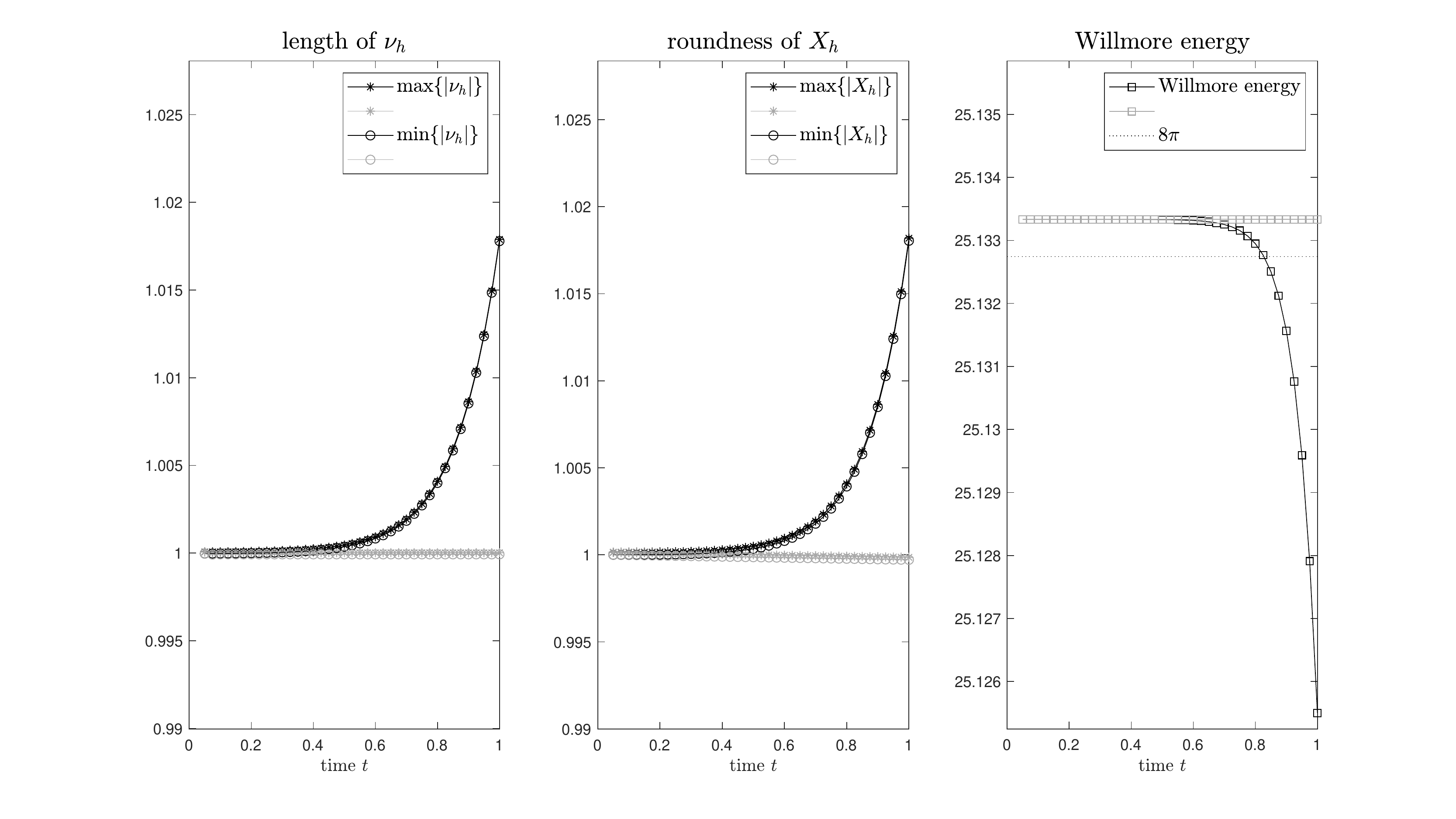}
	\caption{Comparing the effect of projecting $\nu_h$ and $z_h$ onto the unit sphere and the tangent plane, respectively, (black without projections, grey with projections) using a sphere (dof $=1914$, $\tau = 0.025$) reporting on length of normal $\nu_h$, surface roundness (i.e.~radius), and Willmore energy along the flow in $[0,1]$.}
	\label{fig:normalization}
\end{figure}

Figure~\ref{fig:conv_sphere_space} and \ref{fig:conv_Clifford_space} report on the same experiments for the algorithm \emph{with projections} for the sphere and the Clifford torus, respectively. For the Clifford torus we   used the linearly implicit  Euler method (BDF1), whose damping property proved to be advantageous for this delicate experiment. \ebk 

In Figures~\ref{fig:conv_sphere_space_nonprojected}, \ref{fig:conv_sphere_space} and \ref{fig:conv_Clifford_space} the spatial discretization error dominates, matching (exceeding even) the $O(h^2)$ order of convergence (note the reference lines), in agreement with our theoretical results in Theorem~\ref{MainTHM}.

\begin{figure}[htbp]
	\includegraphics[width=\textwidth]{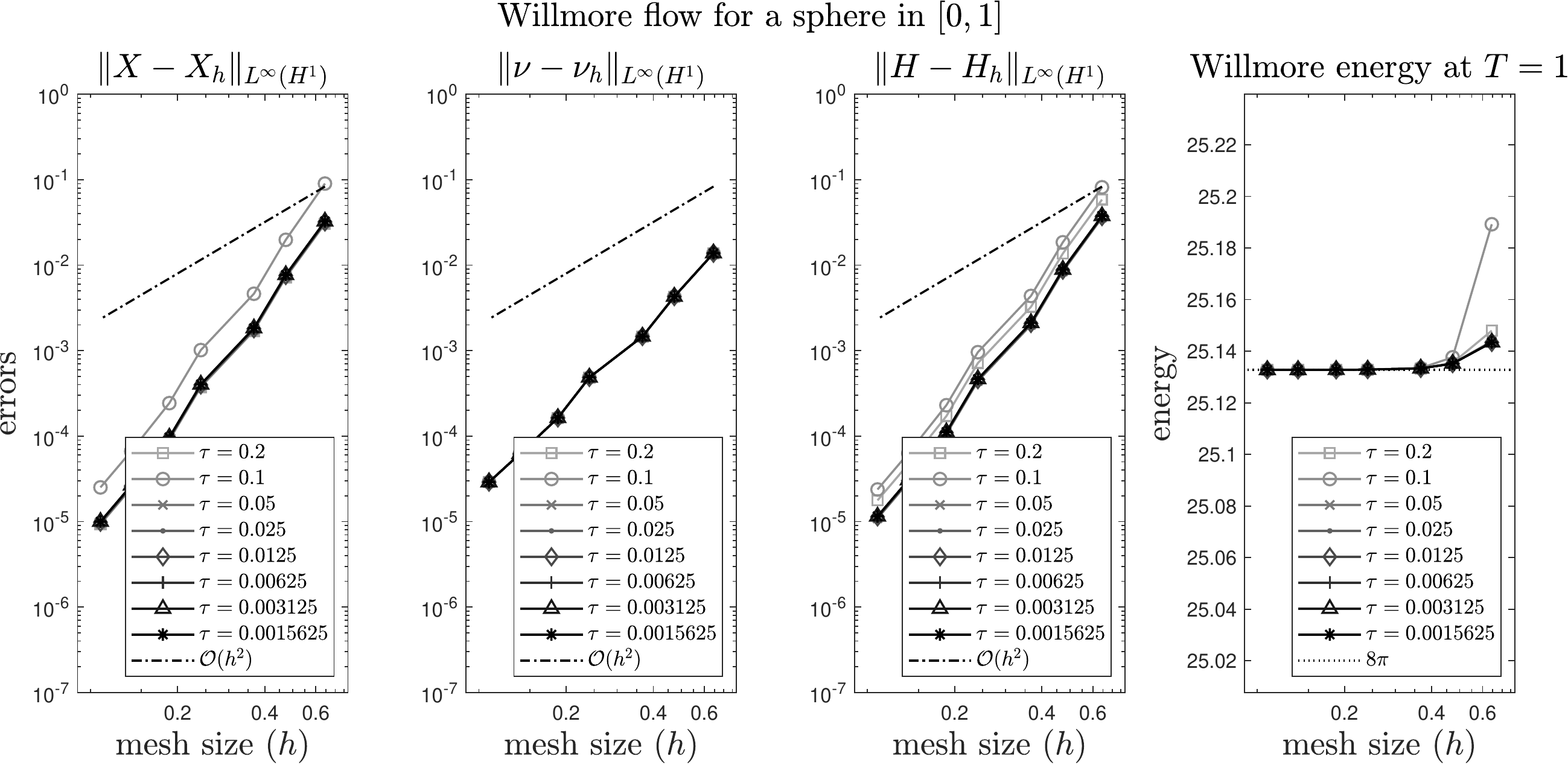}
	\caption{Spatial convergence of the BDF2 / quadratic ESFEM discretization (with projection of $\nu_h$ and $z_h$ onto the unit sphere and the tangent plane, respectively) for the Willmore flow of a sphere.}
	\label{fig:conv_sphere_space}
\end{figure}
\begin{figure}[htbp]
	\includegraphics[width=\textwidth]{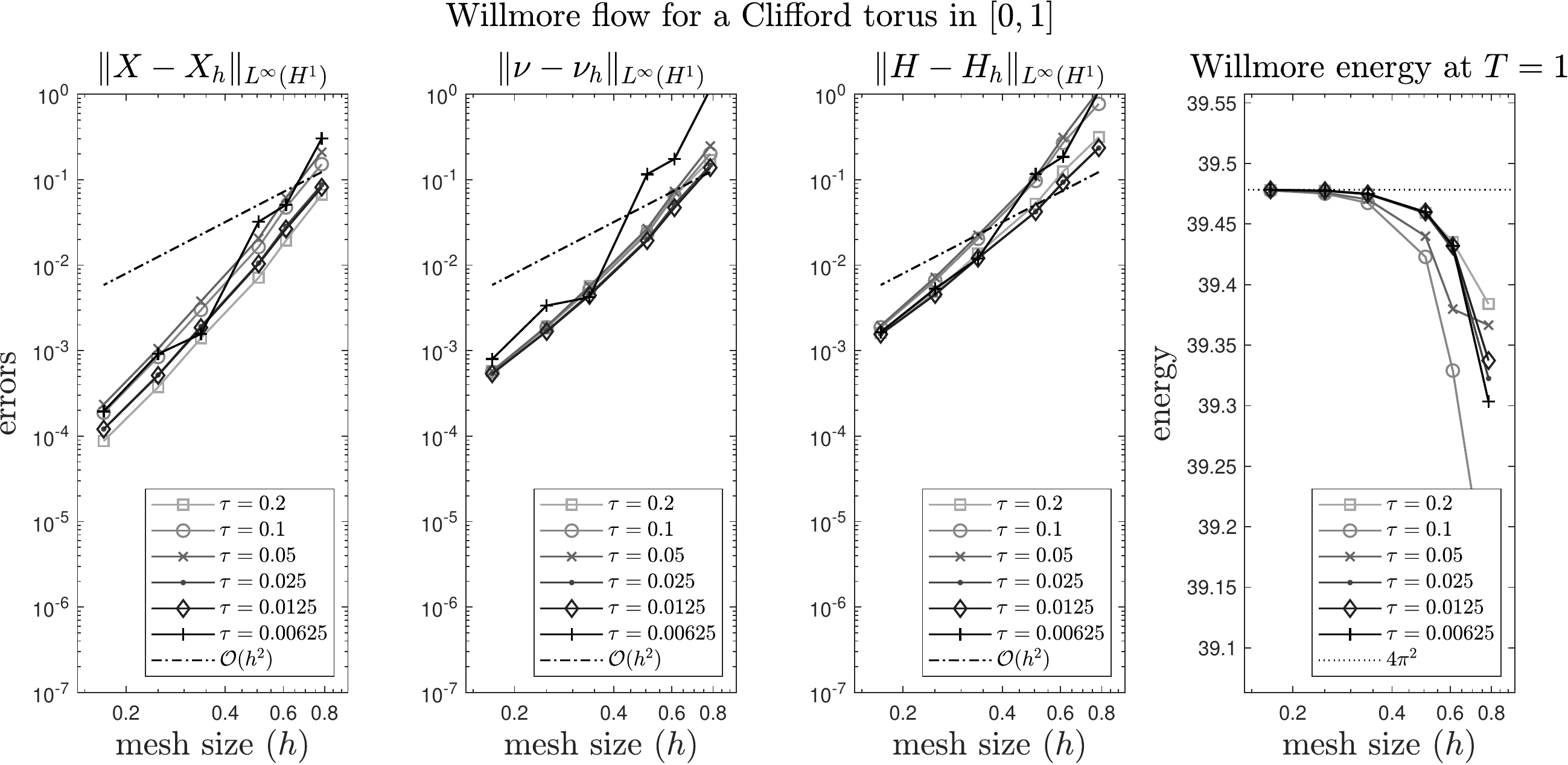}
	\caption{Spatial convergence of the BDF1 / quadratic ESFEM discretization (with projection of $\nu_h$ and $z_h$ onto the unit sphere and the tangent plane, respectively) for the Willmore flow of a Clifford torus.}
	\label{fig:conv_Clifford_space}
\end{figure}

\subsection{Willmore flow towards stationary solutions}

In Figure~\ref{fig:W_E} we report on experiments for the surface evolution and Willmore energy for three surfaces: an ellipsoid, a surface of the  shape of a red blood cell (cf.~\cite{Dziuk_Willmore} and \cite{BGN2008Willmore}), and a periodically perturbed torus (cf.~\cite{Dziuk_Willmore}). 
The meshes have $4978$, $8978$, and $12000$ nodes, respectively, and were integrated in $[0,4]$ with a time step size $\tau = 0.2 \cdot 2^{-8} \approx 7.8 \cdot 10^{-4}$. 
In the last two rows, times $t = 3$ and $4$, it can be observed that the algorithm leaves the surface stationary.
%\bbk (Allow us to note, that in \cite{Dziuk_Willmore} for the same experiments a step size of $\tau = 10^{-5}$ was used, and the corresponding energy plots were only displayed over a shorter time period with $T = 5 \cdot 10^{-3}$.) \ebk 

\begin{figure}[htbp]
	\includegraphics[width=0.32\textwidth]{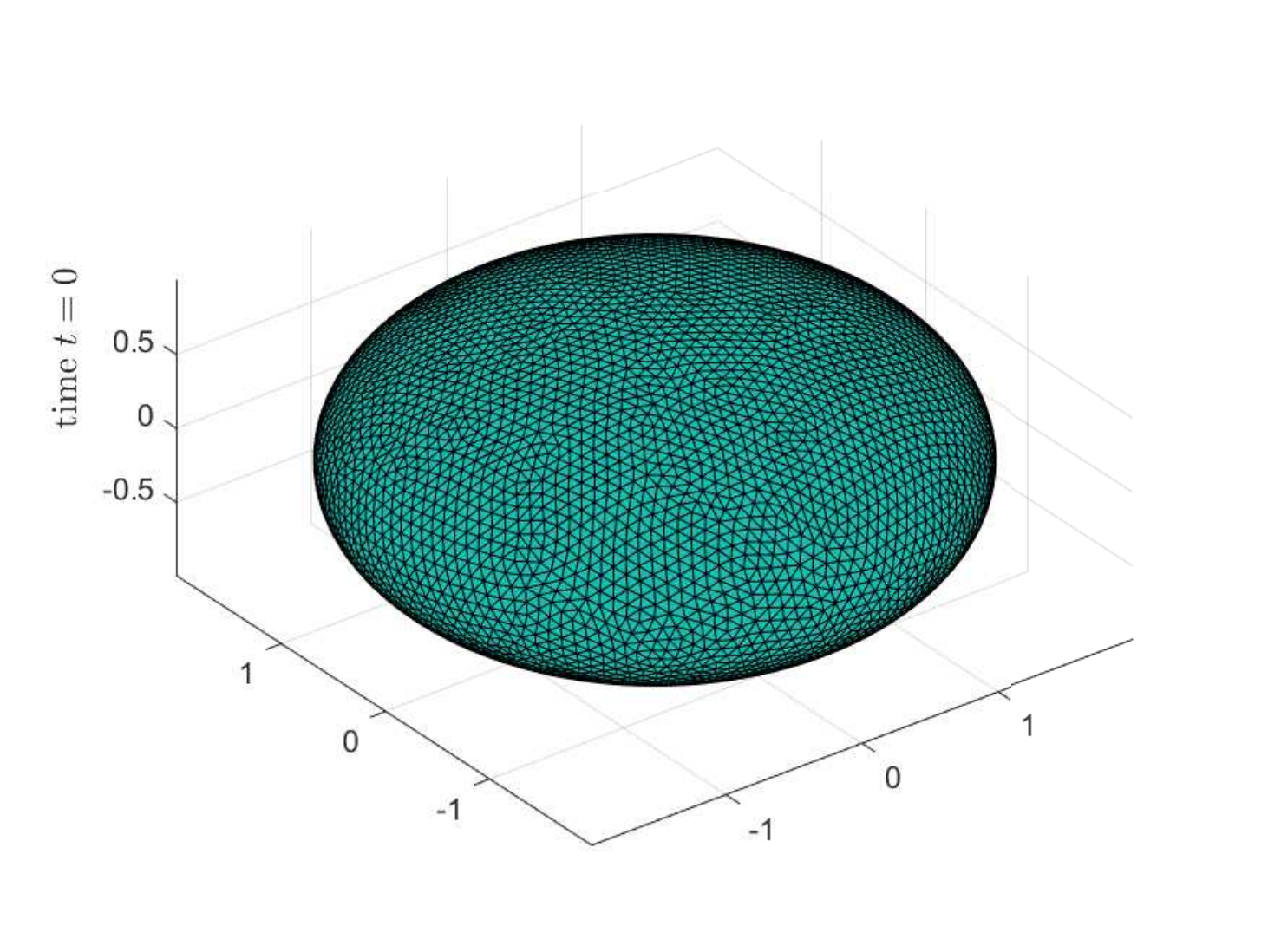} 
	\includegraphics[width=0.32\textwidth]{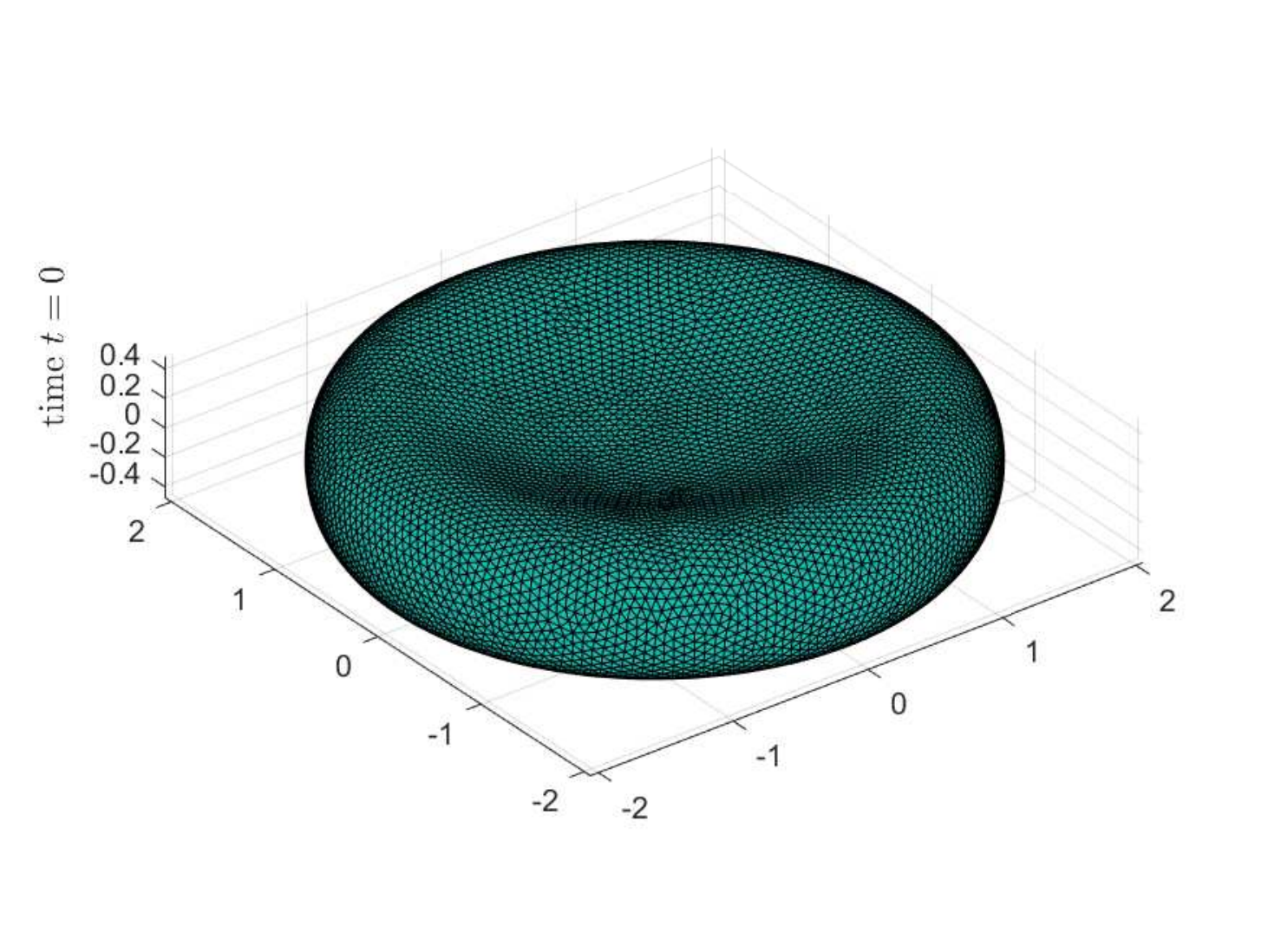} 
	\includegraphics[width=0.32\textwidth]{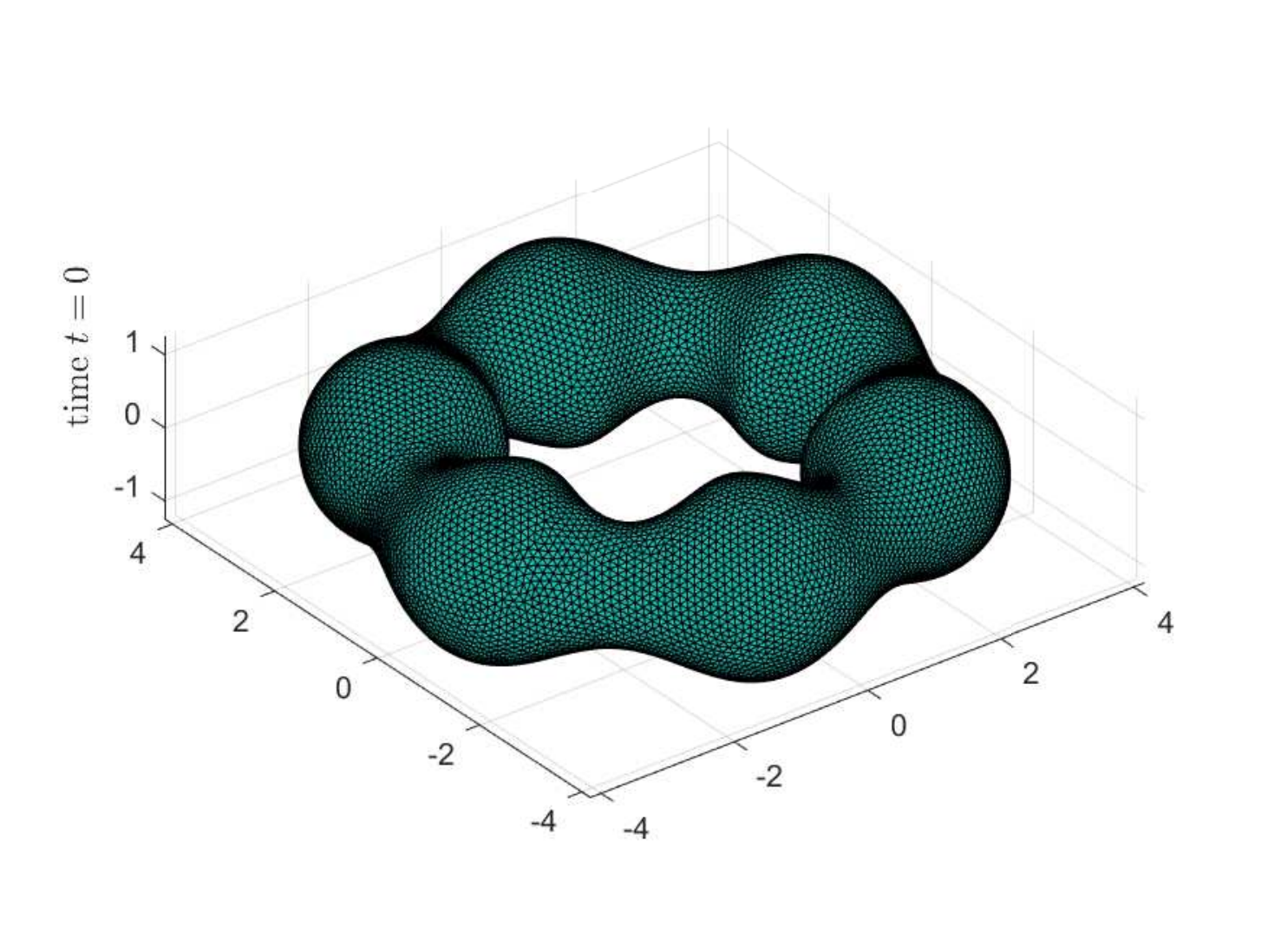} 
	\\
	\includegraphics[width=0.32\textwidth]{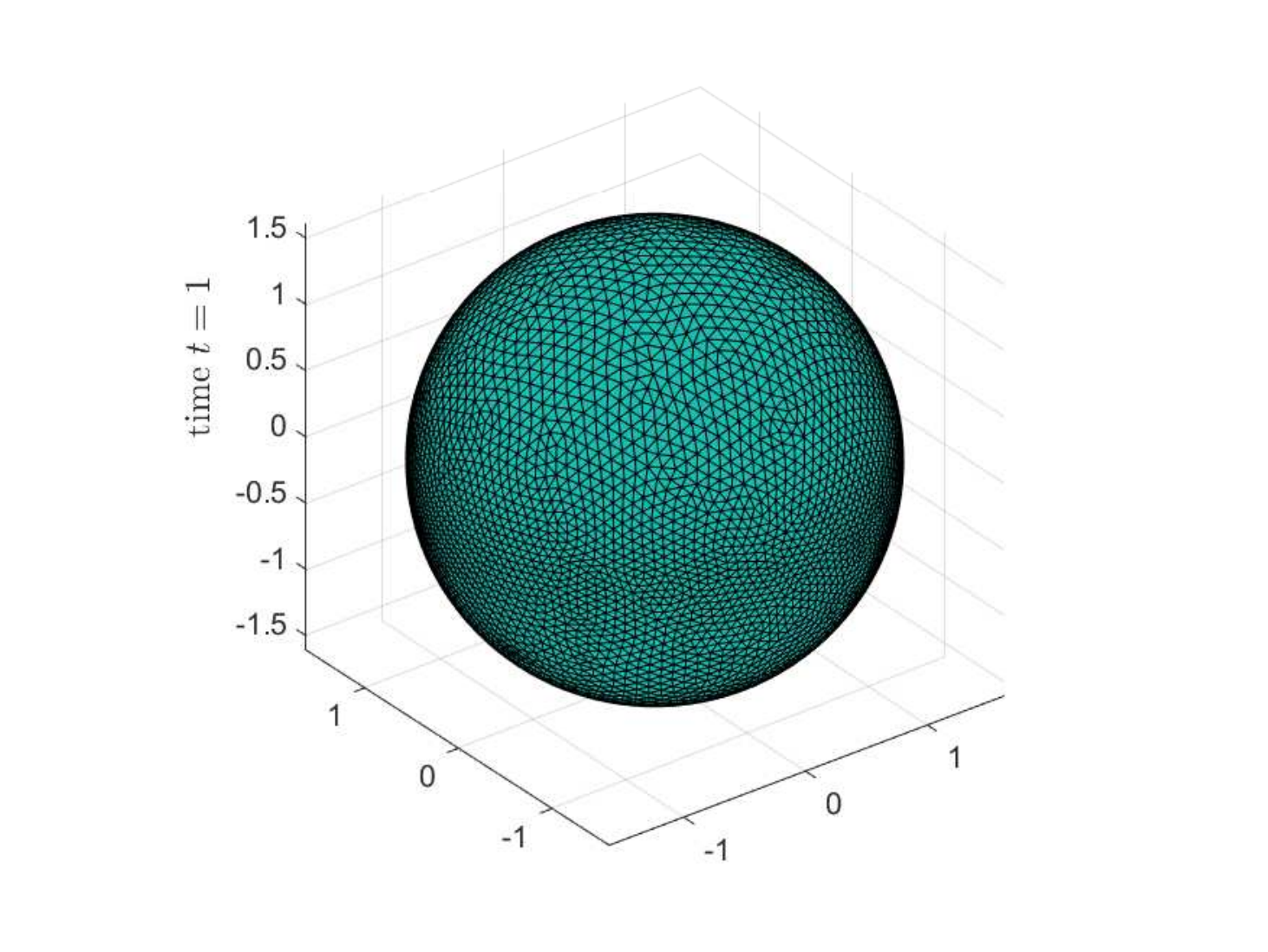} 
	\includegraphics[width=0.32\textwidth]{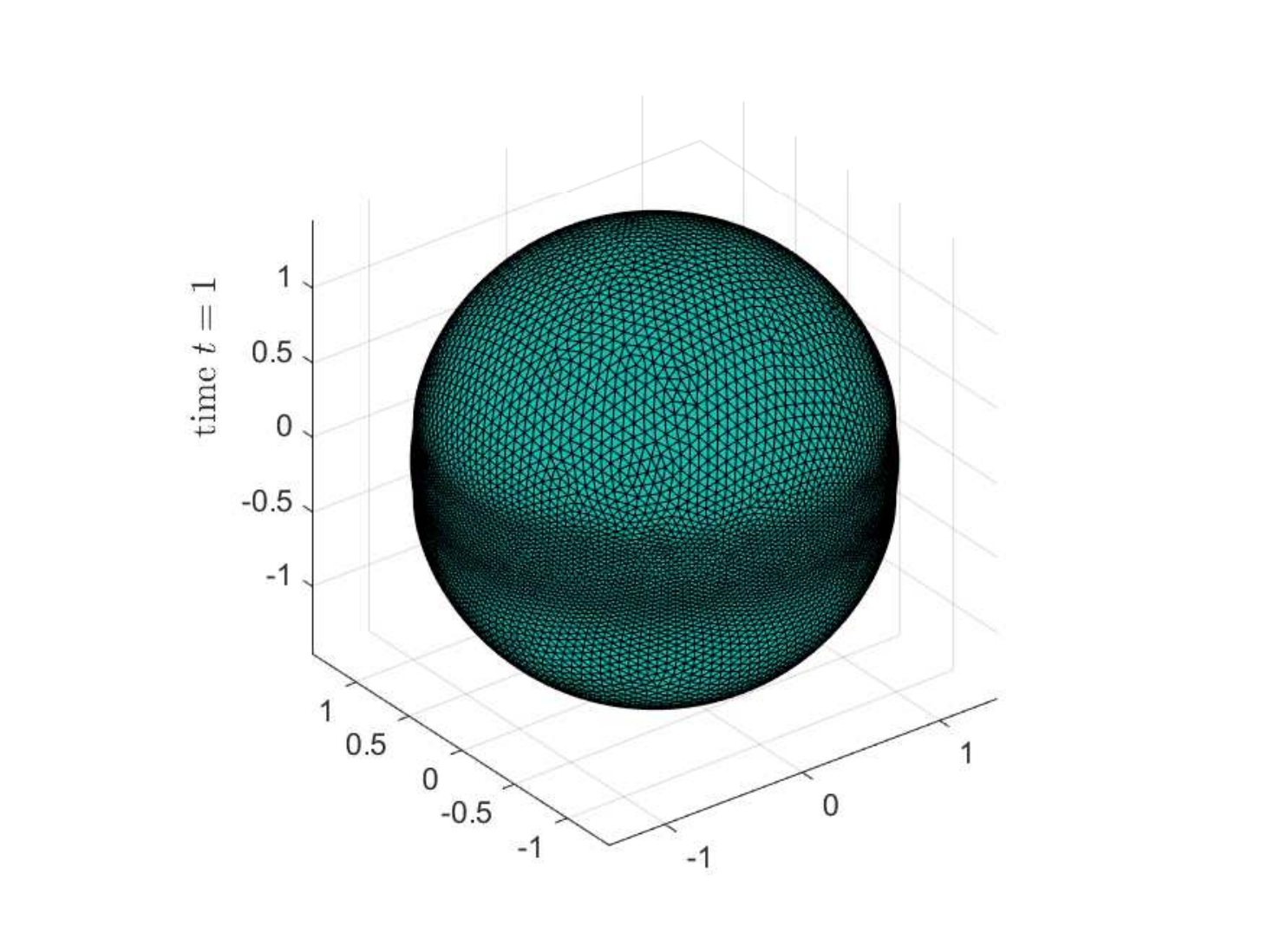} 
	\includegraphics[width=0.32\textwidth]{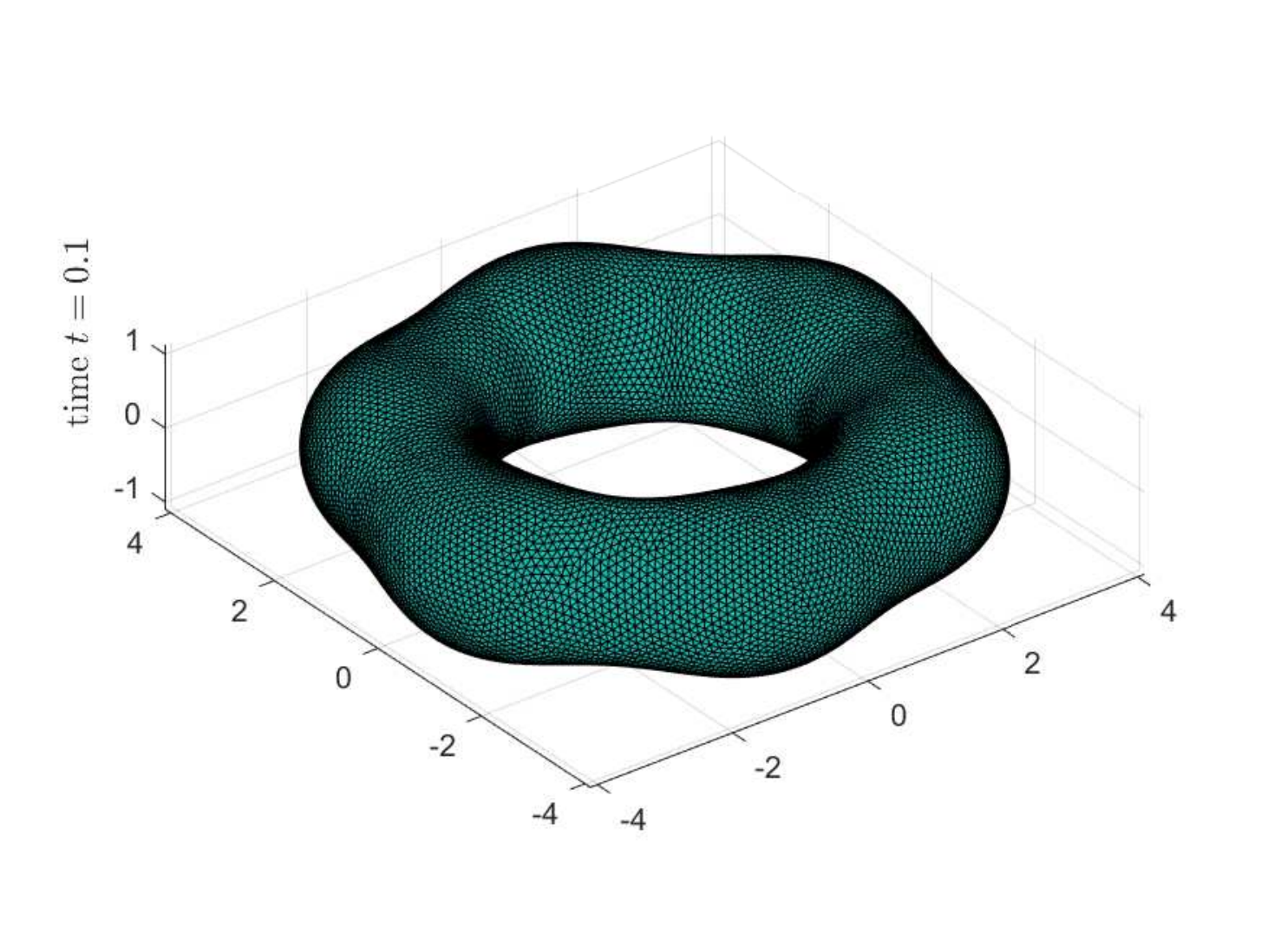} 
	\\
	\includegraphics[width=0.32\textwidth]{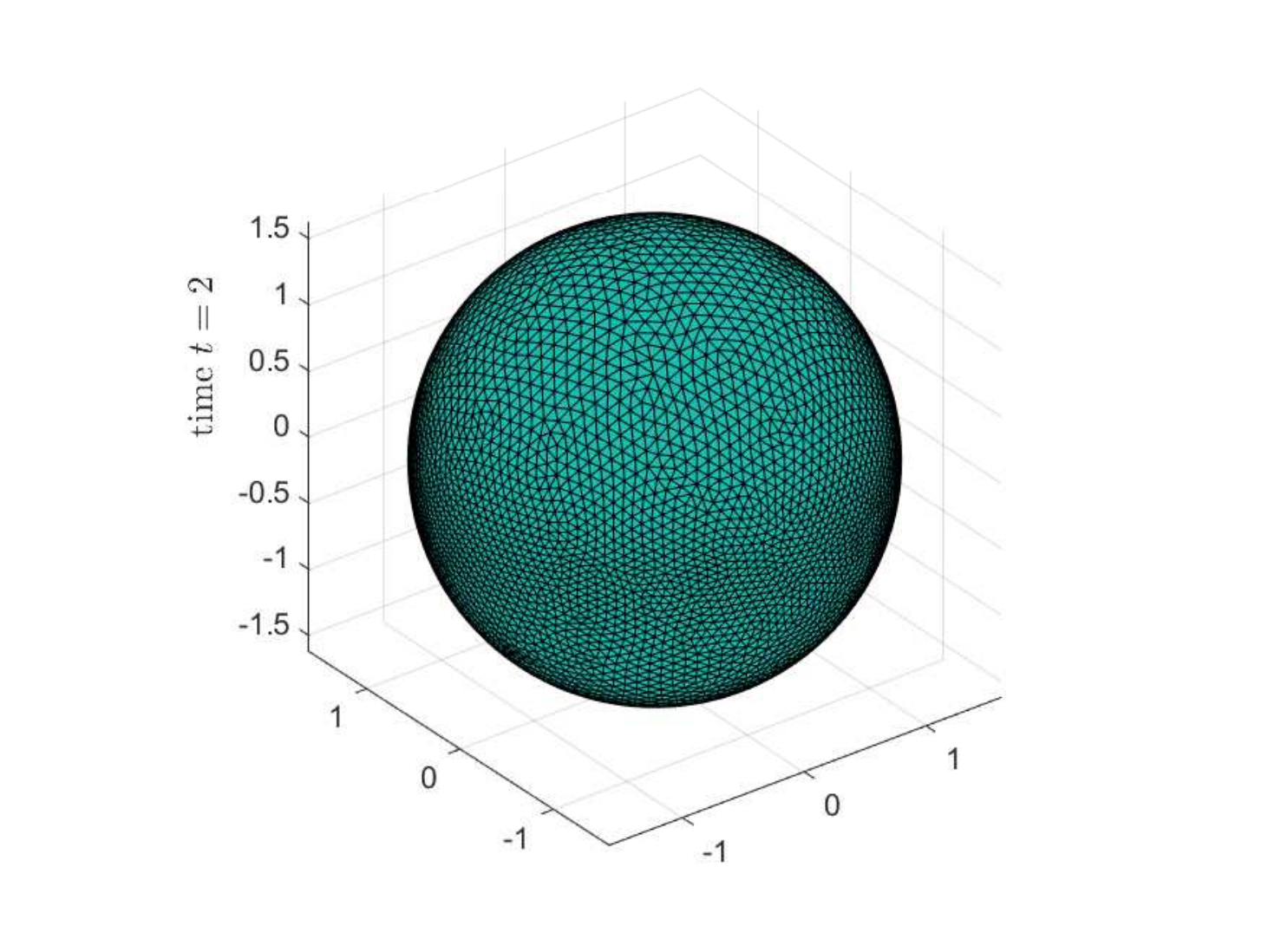} 
	\includegraphics[width=0.32\textwidth]{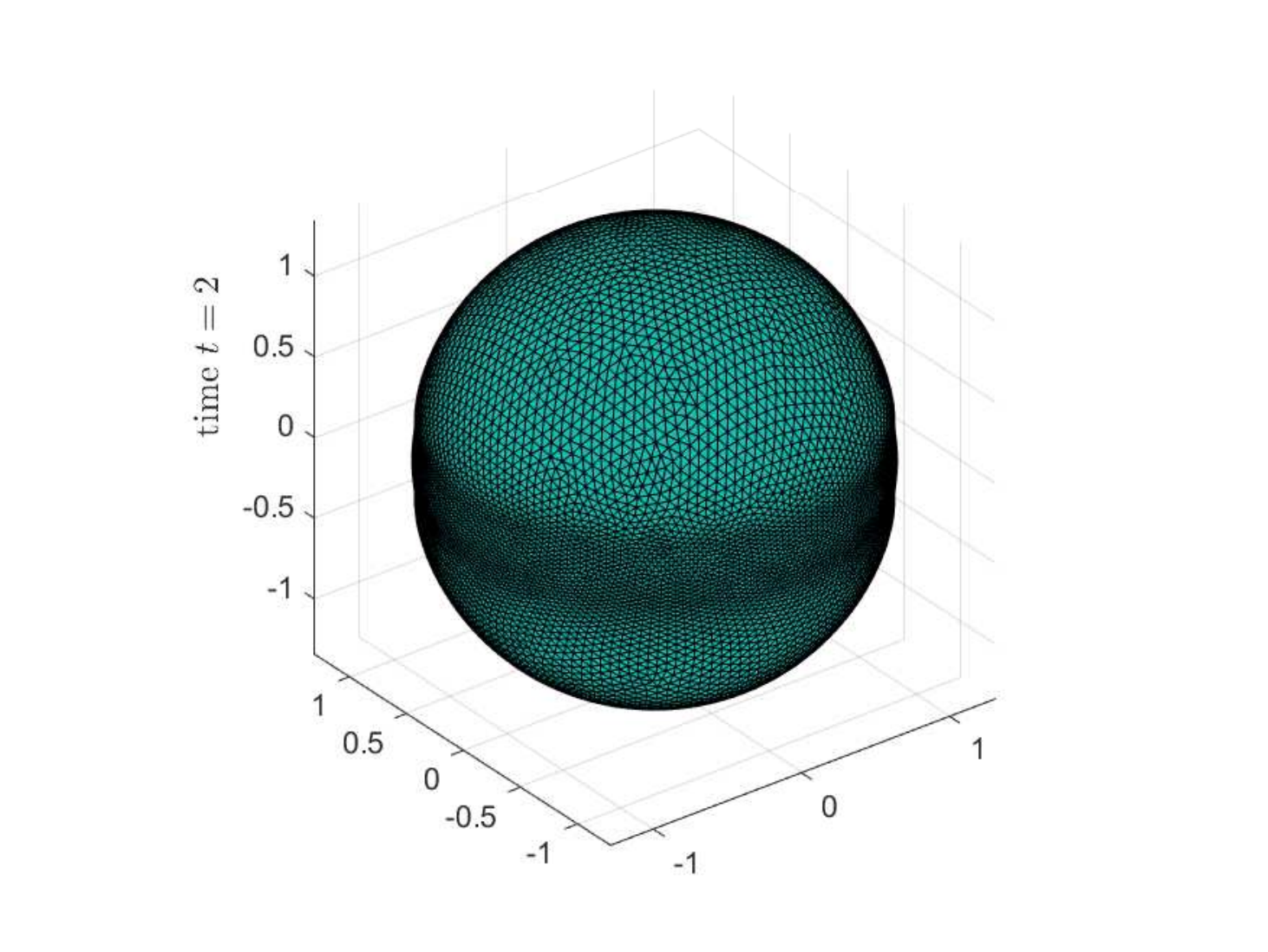} 
	\includegraphics[width=0.32\textwidth]{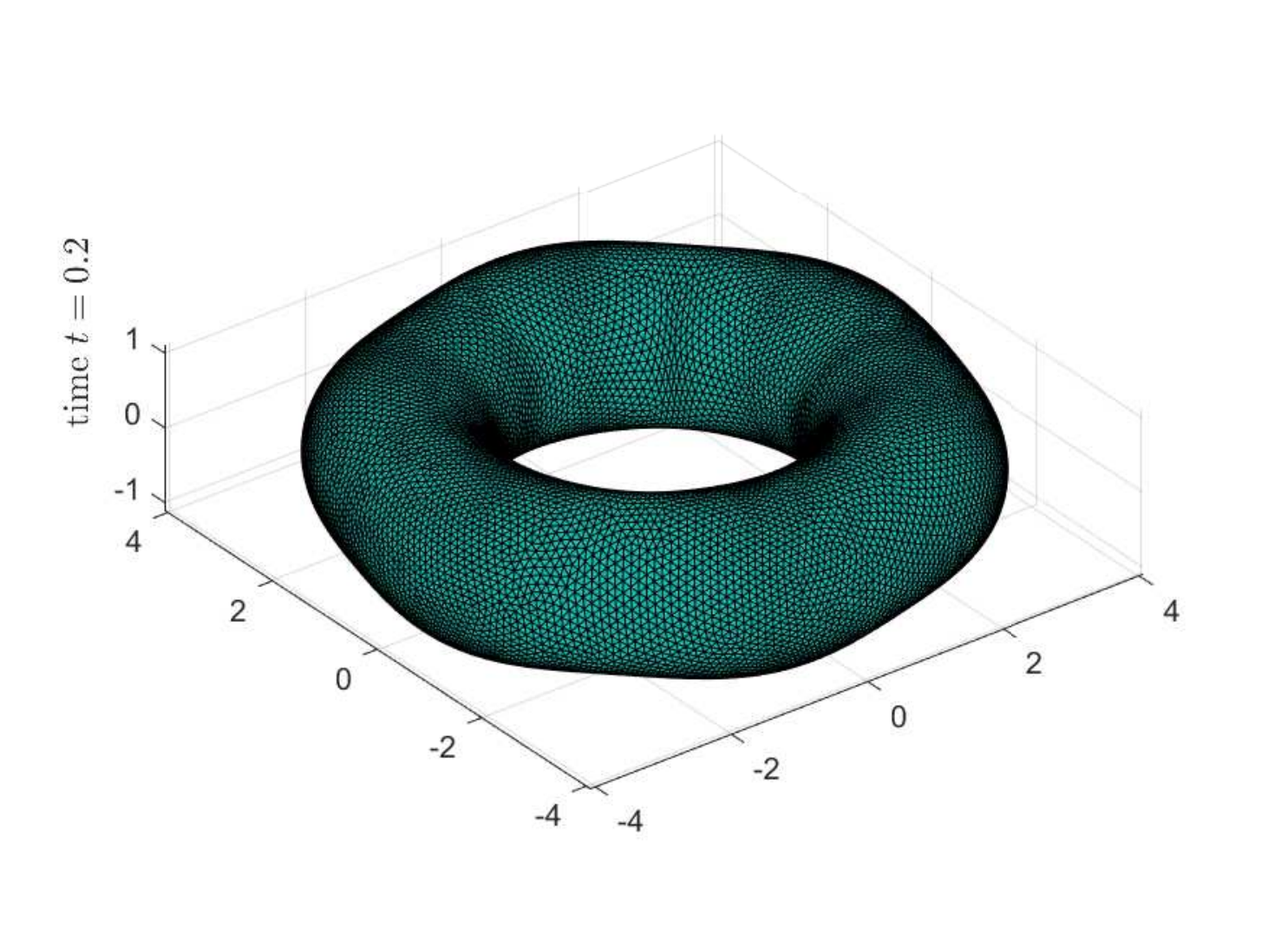} 
	\\
	\includegraphics[width=0.32\textwidth]{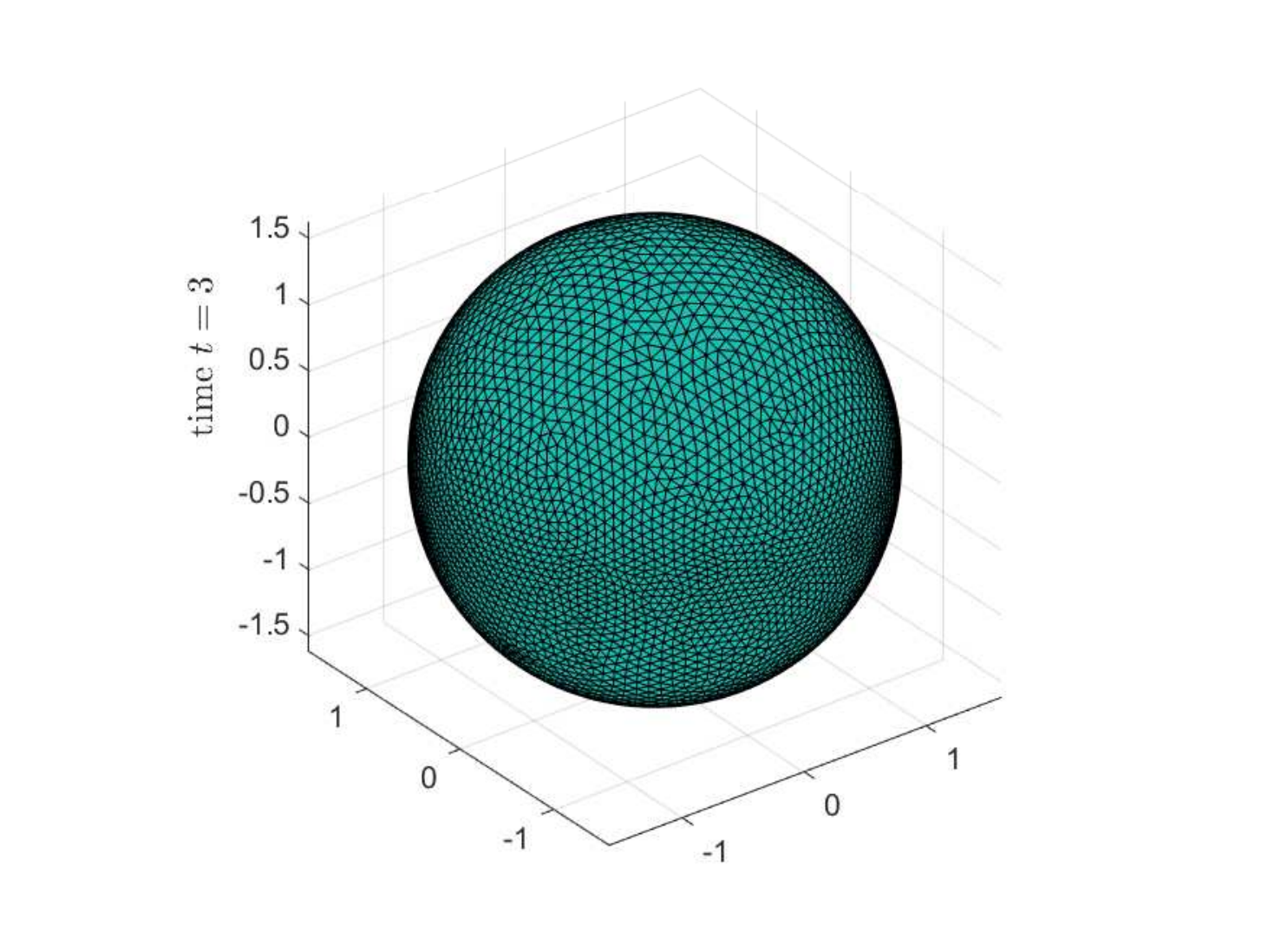} 
	\includegraphics[width=0.32\textwidth]{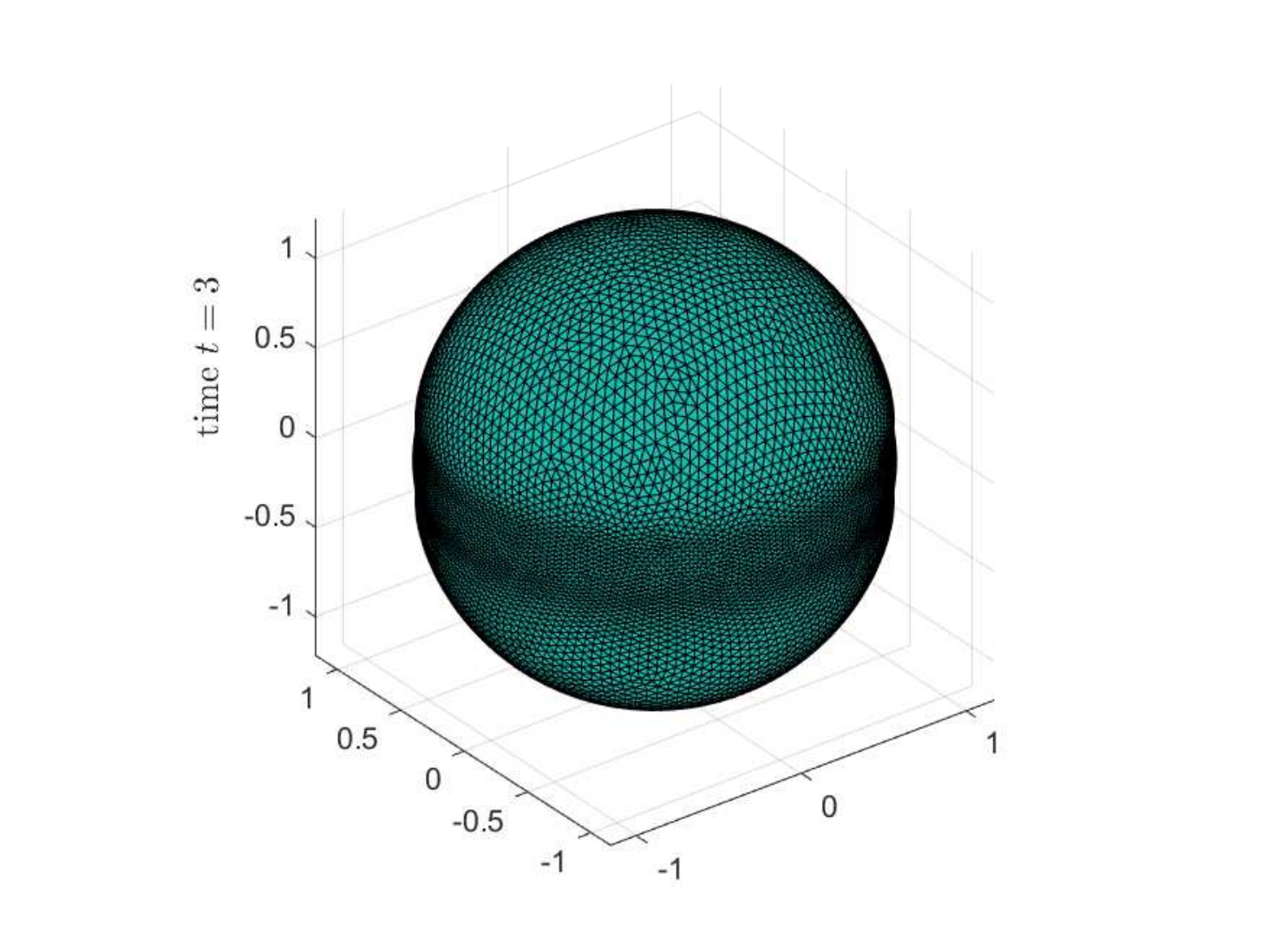} 
	\includegraphics[width=0.32\textwidth]{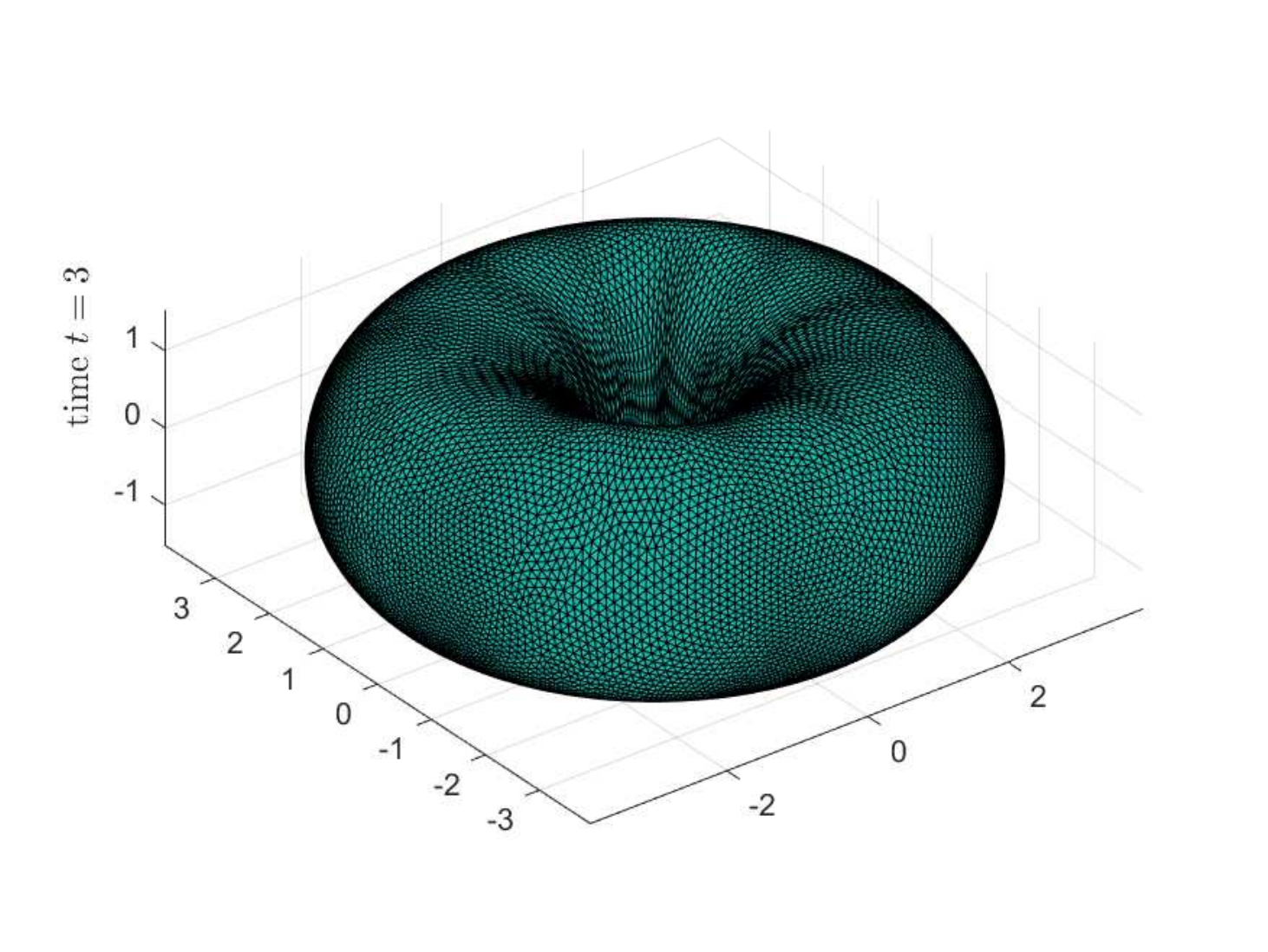} 
	\\
	\includegraphics[width=0.32\textwidth]{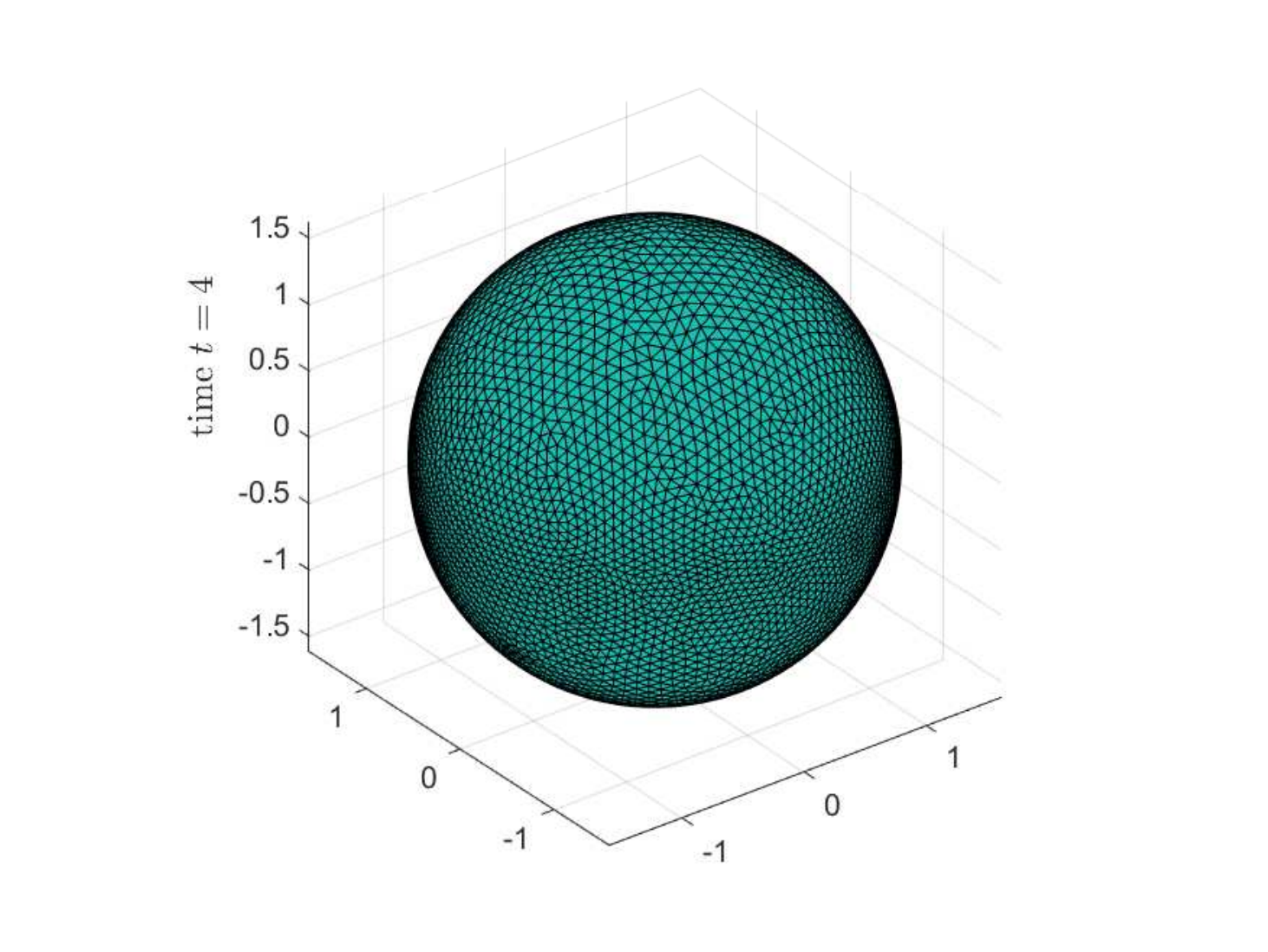} 
	\includegraphics[width=0.32\textwidth]{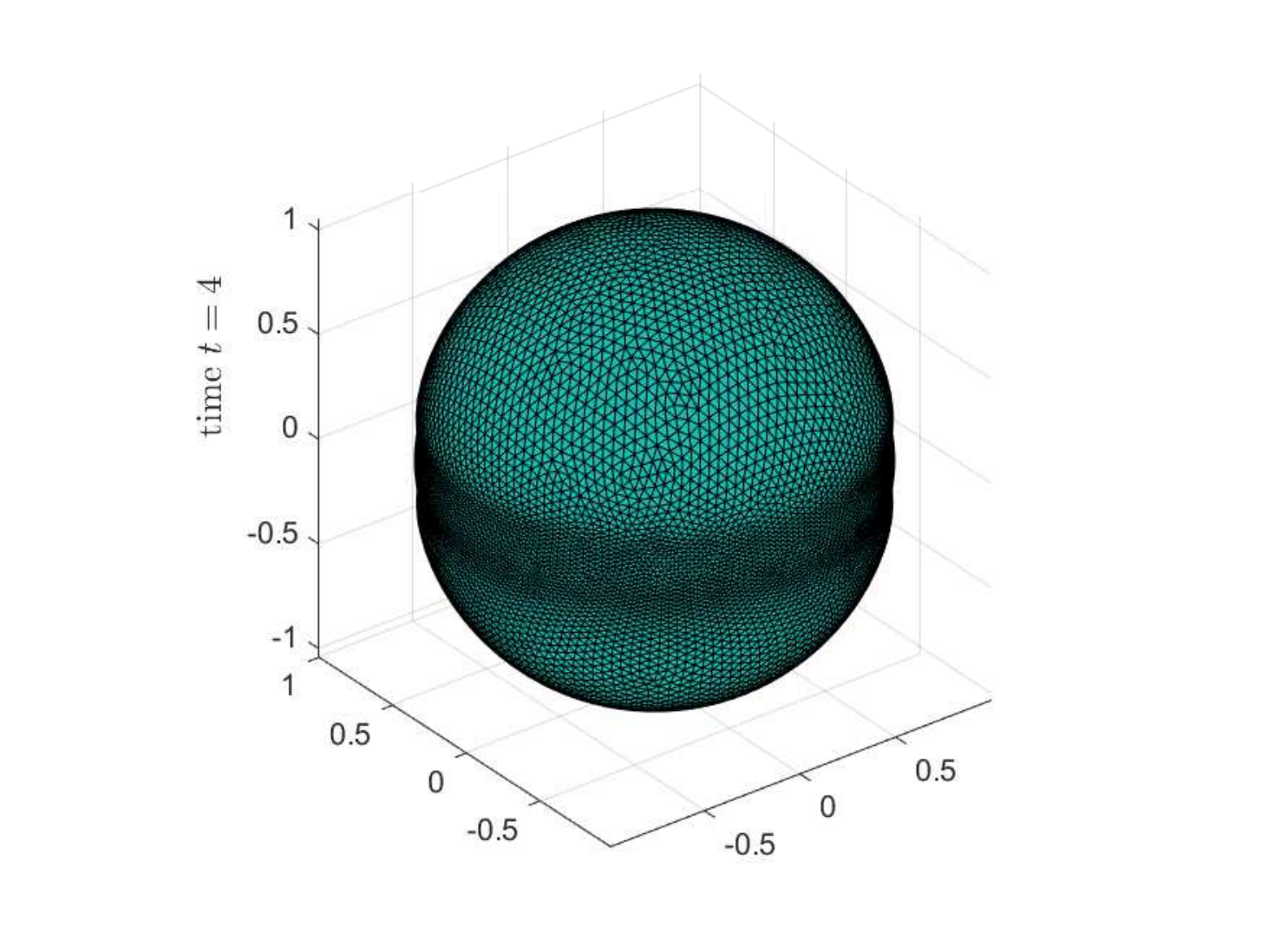} 
	\includegraphics[width=0.32\textwidth]{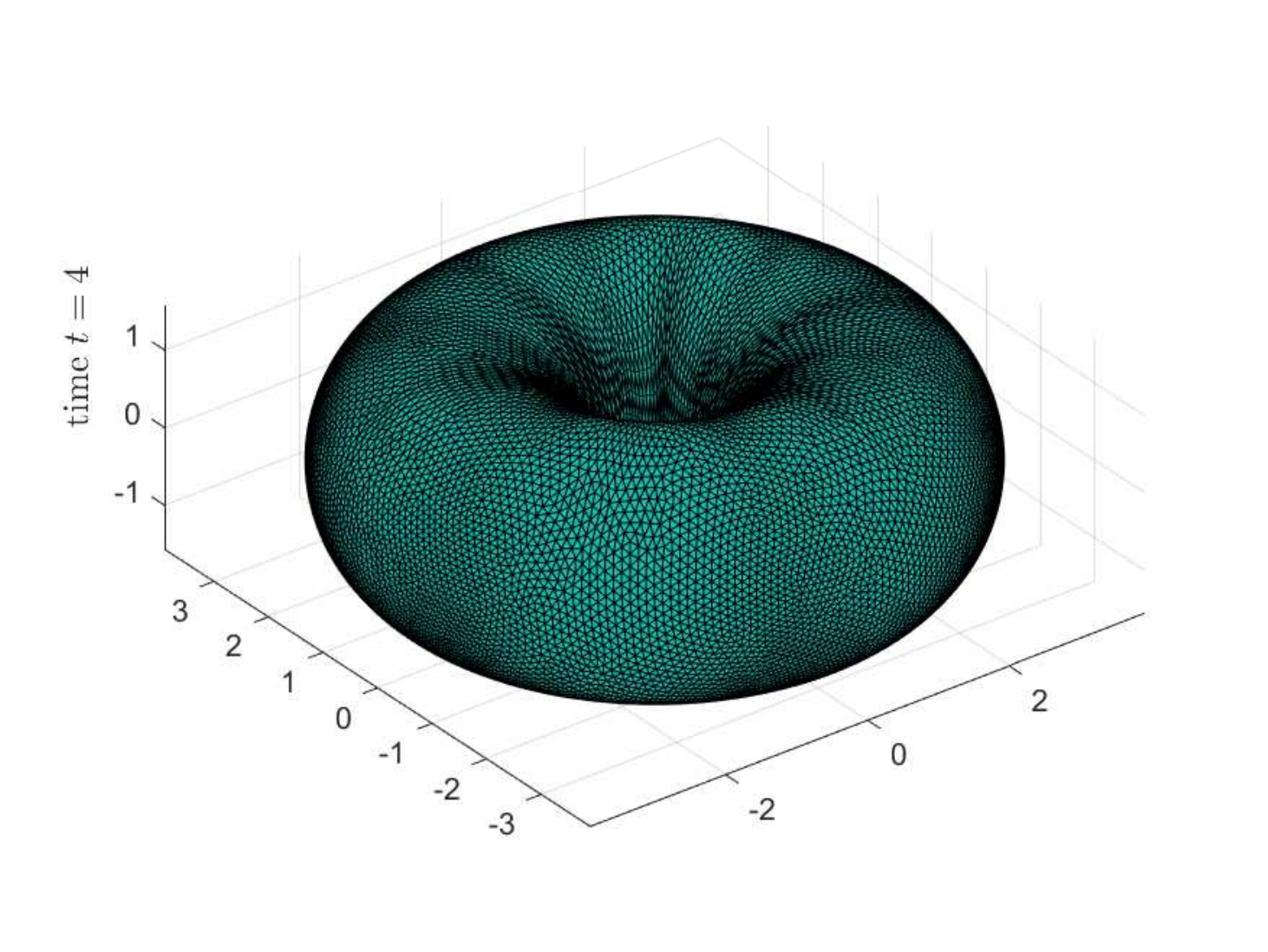} 
	\\
	\includegraphics[width=0.32\textwidth]{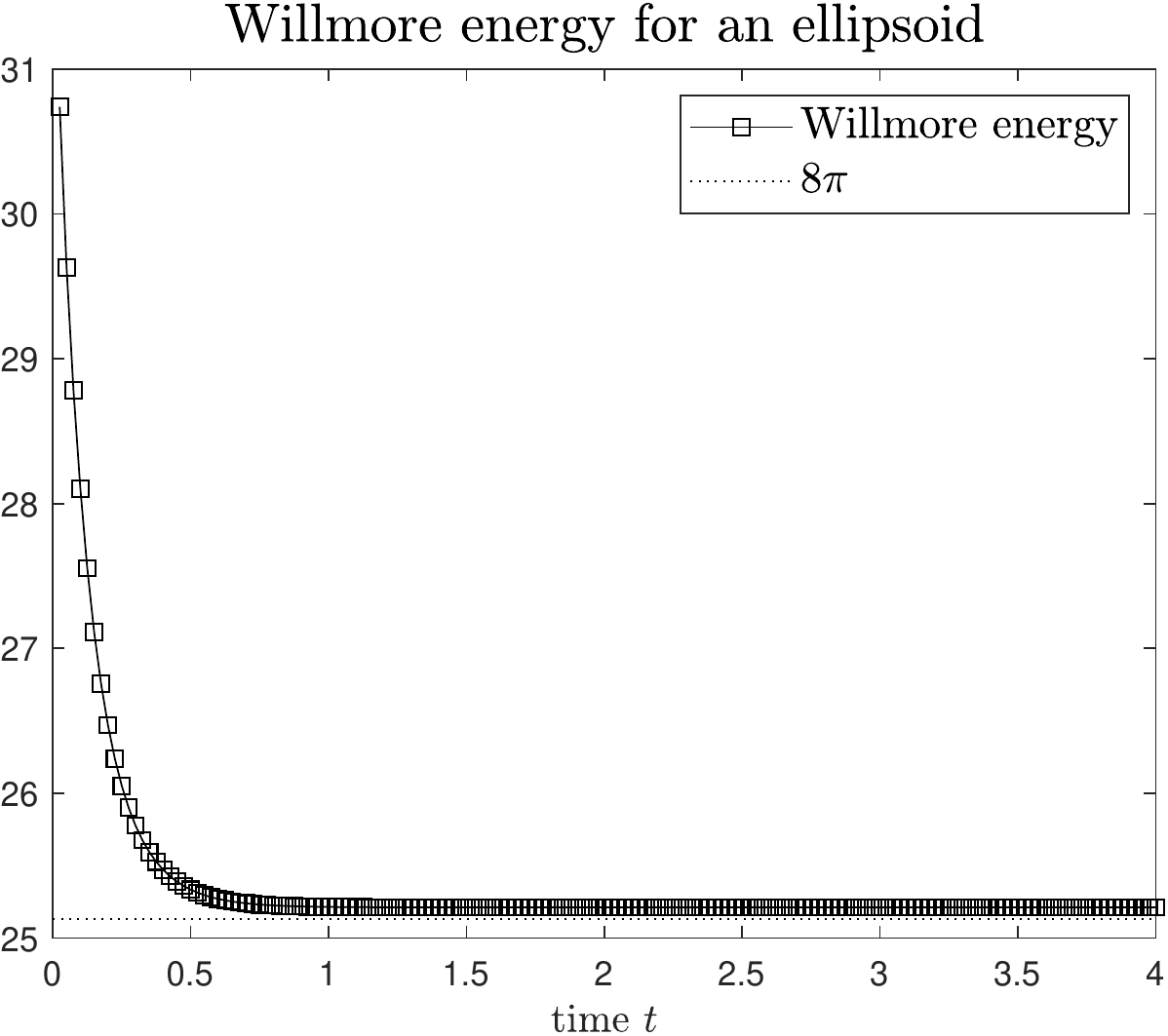}
	\includegraphics[width=0.32\textwidth]{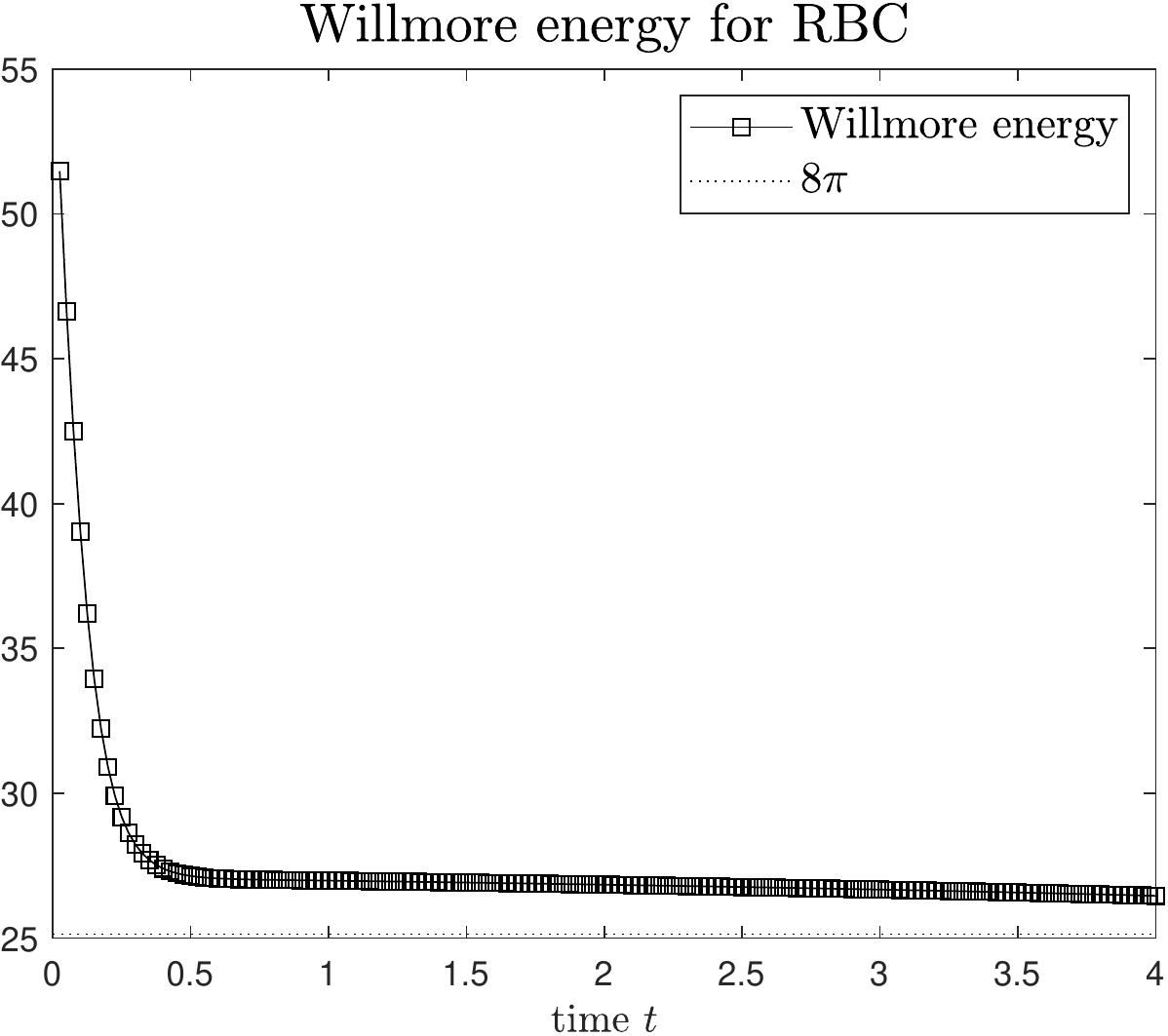}
	\includegraphics[width=0.32\textwidth]{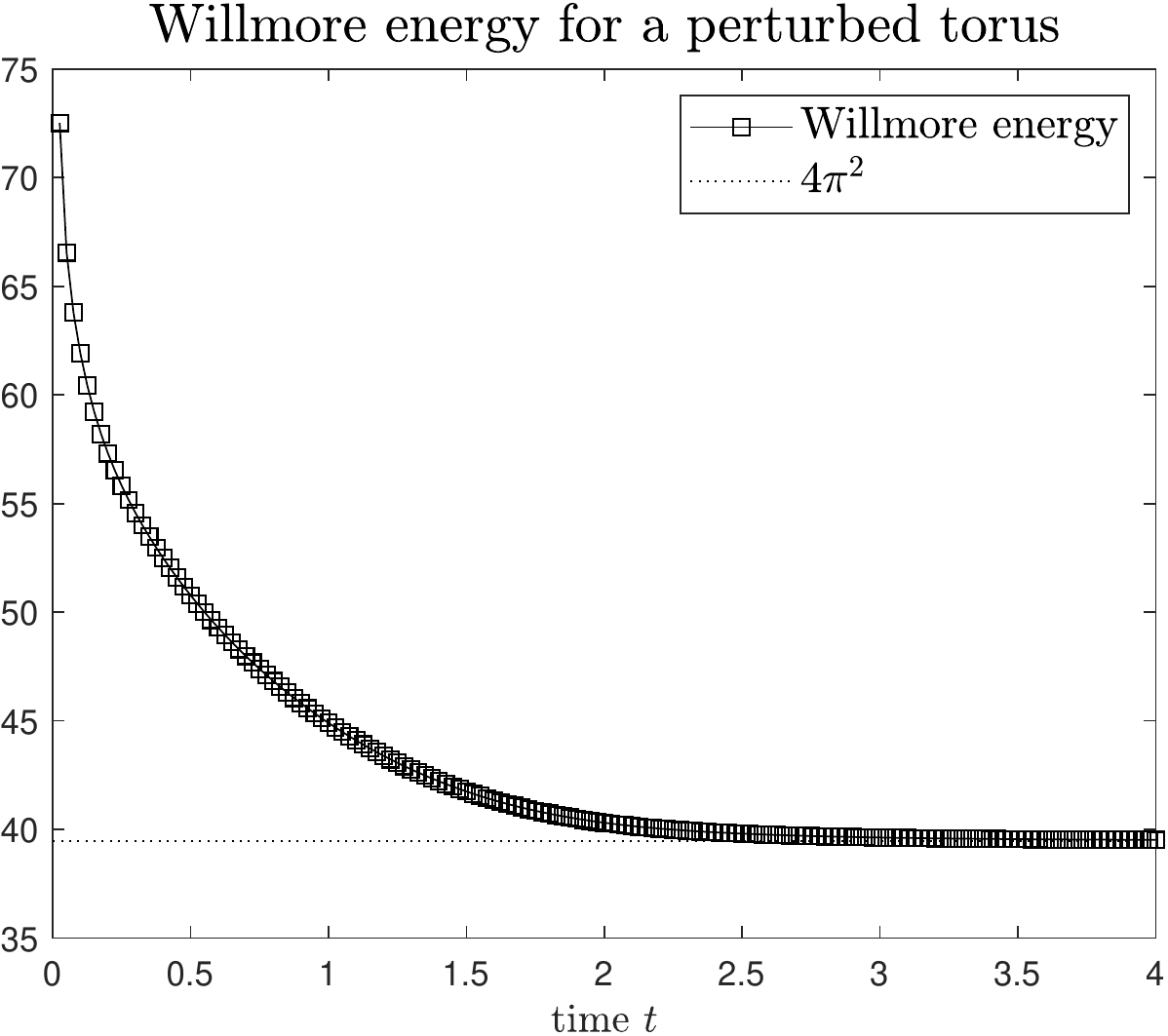}
	\caption{Surface evolutions at different times and corresponding Willmore energy}
	\label{fig:W_E}
\end{figure}

\clearpage

\section*{Appendix: Proof of Lemma \ref{lemma:gradient-laplace commutator formula}}

We came up with two independent proofs of Lemma \ref{lemma:gradient-laplace commutator formula}, one based on local coordinates and the other one based on the formalism of Dziuk and Elliott in \cite{DziukElliott_acta}. As should be, both approaches yield the same result as stated in Lemma \ref{lemma:gradient-laplace commutator formula}. Here we present the more straightforward proof based on \cite{DziukElliott_acta}.

The following commutator formula for tangential differential operators ($D_i = (\nbg)_i$) is shown in \cite[Lemma~2.6]{DziukElliott_acta}:
\begin{equation}
\label{eq:commutator formula - surface derivatives}
	D_i D_j u - D_j D_i u 
%	= (A \nbg u)_j \nu_i - (A \nbg u)_i \nu_j 
	= A_{jk} D_k u \nu_i - A_{ik} D_k u \nu_j \qquad (i , j = 1, 2, 3) ,
\end{equation}
where we use the convention to sum over repeated indices.
%We now prove the following commutator formula, valid on any regular two-dimensional surface $\Ga$ and for any regular function~$f$ on $\Ga$,
%\begin{equation}
%\label{eq:commutator formula}
%	\begin{aligned}
%		\laplace_{\Ga} \nb_{\Ga} f - \nb_{\Ga} \laplace_{\Ga} f = &\ \\
%		D_i D_i D_j f - D_j D_i D_i f = &\ .......
%		K  \, \nb_{\Ga} f .
%	\end{aligned} 
%\end{equation}
%
We also recall that 
	$A_{ij} = \ D_i \nu_j = A_{ji}$ and $H = {\rm tr}(\nbg \nu) = D_i \nu_i$.
	
Using \eqref{eq:commutator formula - surface derivatives} we obtain (cf.~the proof of Lemma~3.2 in \cite{DziukElliott_acta})
\begin{align*}
	D_i D_i D_j u = &\ D_i \big( D_j D_i u + A_{jk} D_k u \nu_i - A_{ik} D_k u \nu_j \big) \\
	= &\ D_i D_j D_i u + D_i \big( D_k u (A_{jk} \nu_i - A_{ik} \nu_j) \big) \\
	= &\ D_i D_j D_i u + D_i D_k u \big( A_{jk} \nu_i - A_{ik} \nu_j \big) \\
	&\ \phantom{D_i D_j D_i u } + D_k u D_i \big( A_{jk} \nu_i - A_{ik} \nu_j \big) .
\end{align*}
The second term is then rewritten as
\begin{align*}
	&\hspace{-5mm}D_i D_k u \big( A_{jk} \nu_i - A_{ik} \nu_j \big) 
	= \ A_{jk} D_i D_k u \nu_i - A_{ik} D_i D_k u \nu_j \\
	= &\ A_{jk} \big( D_k D_i u + A_{kl} D_l u \nu_i - A_{il} D_l u \nu_k \big) \nu_i - A_{ik} D_i D_k u \nu_j \\
	= &\ A_{jk} (D_k D_i u \nu_i) + A_{jk} A_{kl} D_l u \nu_i\nu_i - A_{jk} A_{il} D_l u \nu_k \nu_i - A_{ik} D_i D_k u \nu_j \\
	= &\ - A_{jk} A_{ki} D_i u + A_{jk} A_{kl} D_l u - A_{jk} A_{il} D_l u \nu_k \nu_i - A_{ik} D_i D_k u \nu_j \\
	= &\ - (A^2)_{ji} D_i u + (A^2)_{jl} D_l u - A_{jk} A_{il} D_l u \nu_k \nu_i - A_{ik} D_i D_k u \nu_j \\
	= &\ - A_{jk} A_{il} D_l u \nu_k \nu_i - A_{ik} D_i D_k u \nu_j \\
	= &\  - A_{ik} D_i D_k u \nu_j . 
\end{align*}
%[kb: Furthermore, the second term is orthogonal to tangential gradients, cf.~\cite[Lemma~3.2]{DziukElliott_acta}. I cannot explain this.]
In particular for the last term we have:
\begin{align*}
	D_i \big( A_{jk} \nu_i - A_{ik} \nu_j \big) 
	= &\ D_i A_{jk} \nu_i - D_i A_{ik} \nu_j + A_{jk} D_i \nu_i - A_{ik} D_i \nu_j \\
	= &\ D_i A_{jk} \nu_i - D_i A_{ik} \nu_j + A_{jk} H - A_{ki} A_{ij} \\
	= &\ -\Delta_{\Gamma}\nu_k \nu_j\eby + (HA - A^2)_{jk} .
\end{align*}
%[kb: The first two terms should be zero, cf.~\cite[Lemma~3.2]{DziukElliott_acta}. I cannot explain this.] [kb update: Because this is not zero, but eventually will give the $\Delta_{\Ga} \nu$ term.]
%\bby
%[by: In \cite[Lemma~3.2]{DziukElliott_acta} I didn't find $D_i A_{jk} \nu_i - D_i A_{ik} \nu_j=0$ ? 
%
%$D_i A_{jk} \nu_i= \nu\cdot\nabla_{\Gamma} A_{jk} =0$ 
%and $D_i A_{ik} \nu_j= D_i D_i\nu_k \nu_j=\Delta_{\Gamma}\nu_k \nu_j$. 
%]
%\eby
The third-order term is again rewritten using \eqref{eq:commutator formula - surface derivatives} as
\begin{align*}
	D_i D_j D_i u = &\ D_j D_i D_i u + A_{jk} D_k D_i u \nu_i - A_{ik} D_k D_i u \nu_j , 
\end{align*}
which is further rewritten using
\begin{align*}
	D_k D_i u \nu_i = &\ D_k (D_i u \nu_i) - D_k \nu_i D_i u = - A_{ki} D_i u , \\
	D_k D_i u \nu_j = &\ D_k (D_i u \nu_j) - D_k \nu_j D_i u .
\end{align*}
Therefore,
\begin{align*}
	D_i D_j D_i u = &\ D_j D_i D_i u -  A_{jk} A_{ki} D_i u - A_{ik} D_k D_i u \nu_j \\
	= &\ D_j D_i D_i u -  (A^2)_{ji} D_i u - A_{ik} D_k D_i u \nu_j .
\end{align*}
Altogether, the above calculations yield
\begin{align*}
&D_i D_i D_j u \\
&= D_j D_i D_i u -  (A^2)_{ji} D_i u -  2A_{ik} D_k D_i u \nu_j  -\Delta_{\Gamma}\nu_k \nu_j D_k u + (HA - A^2)_{jk} D_k u \\
&= D_j D_i D_i u -  (A^2)_{ji} D_i u -  2D_k (A_{ik} D_i u) \nu_j 
+ 2D_k A_{ik} D_i u \nu_j   
\\
&\quad -\Delta_{\Gamma}\nu_k \nu_j D_k u + (HA - A^2)_{jk} D_k u \\
&= D_j D_i D_i u -  (A^2)_{ji} D_i u -  2D_k (A_{ik} D_i u) \nu_j 
+ \Delta_{\Gamma}\nu_k \nu_j D_k u + (HA - A^2)_{jk} D_k u .
\end{align*}
This implies
\begin{align*}
\Delta_{\Gamma} \nabla_{\Gamma} u 
	= &\ \nabla_{\Gamma} \Delta_{\Gamma} u -  A^2 \nabla_{\Gamma} u 
	-2\nabla_{\Gamma}\cdot (A\nabla_{\Gamma}u) \nu  \\
	&\quad + (\Delta_{\Gamma}\nu\cdot\nabla_{\Gamma}u) \nu + (HA - A^2) \nabla_{\Gamma} u ,
\end{align*}
which is the same as \eqref{eq:gradient-laplace commutator formula}. 
%\eby
%
%\bbk
%Altogether, the above calculations yield
%\begin{align*}
%	D_i D_i D_j u 
%	= &\ D_j D_i D_i u -  (A^2)_{ji} D_i u - A_{ik} D_k D_i u \nu_j 
%	+ D_i D_k u \big( A_{jk} \nu_i - A_{ik} \nu_j \big) \\ 
%	&\ + (HA - A^2)_{jk} D_k u \\
%	= &\ D_j D_i D_i u - A_{ik} D_k D_i u \nu_j + D_i D_k u \big( A_{jk} \nu_i - A_{ik} \nu_j \big) \\ 
%	&\ + (HA - 2 A^2)_{jk} D_k u \\
%	= &\ D_j D_i D_i u - A_{ik} D_k D_i u \nu_j \\
%	&\ - A_{jk} A_{il} D_l u \nu_k \nu_i - A_{ik} D_i D_k u \nu_j \\ 
%	&\ + (HA - 2 A^2)_{jk} D_k u \\
%	= &\ D_j D_i D_i u - A_{ik} ( D_k D_i u + D_i D_k u ) \nu_j - A_{jk} A_{il} D_l u \nu_k \nu_i \\ 
%	&\ + (HA - 2 A^2)_{jk} D_k u \\
%	= &\ D_j D_i D_i u - 2 A_{ik} D_k D_i u \nu_j + \Big[ (A^2)_{il} D_l u \nu_i \nu_j + (A_{ik} A_{il} D_l u \nu_j - A_{jk} A_{il} D_l u \nu_i) \nu_k \Big] \\ 
%	&\ + (HA - 2 A^2)_{jk} D_k u \\
%\end{align*}
%
%[kb update: Christian have realized that the three terms in the bracket should vanish.]
%
%....
\qed

\section*{Acknowledgement}

We thank J\"org Nick for helpful discussions and his nice debugging idea.\\
The work of  Bal\'azs Kov\'acs and Christian Lubich is supported by Deutsche Forschungsgemeinschaft -- Project-ID 258734477 -- SFB 1173.

\bibliographystyle{alpha}
\bibliography{evolving_surface_literature}
	
\end{document}